\documentclass[11pt, oneside]{amsart}  
\usepackage{geometry} 
\usepackage{graphicx}
									
\usepackage{amssymb,amsmath,amsthm}
\usepackage{subcaption}
\usepackage{hyperref}
\usepackage{algpseudocode}
\usepackage[section]{algorithm}
\usepackage{tikz}
\usepackage{tikz-cd}
\usepackage{mathrsfs}
\usepackage{enumitem}
\usepackage{textcomp}

\DeclareMathOperator{\arccosh}{arccosh}
\DeclareMathOperator{\arcsinh}{arcsinh}
\DeclareMathOperator{\arctanh}{arctanh}

\DeclareMathOperator*{\esssup}{ess\,sup}

\newtheorem{theorem}{Theorem}[section]
\newtheorem*{theorem*}{Theorem}
\newtheorem{lemma}[theorem]{Lemma}
\newtheorem*{lemma*}{Lemma}

\newtheorem{prop}[theorem]{Proposition}
\newtheorem*{prop*}{Proposition}

\theoremstyle{definition}
\newtheorem{definition}[theorem]{Definition}
\newtheorem*{definition*}{Definition}
\newtheorem{example}[theorem]{Example}

\theoremstyle{remark}
\newtheorem{remark}[theorem]{Remark}

\numberwithin{equation}{section}

\title[Branch points in elastic pseudospherical surfaces]{Distributed branch points and the shape of elastic surfaces with constant negative curvature} 

\author{Toby L. Shearman}
\address{Program in Applied Mathematics, 617 N. Santa Rita Avenue, University of Arizona, Tucson, AZ 85721}
\email{toby.shearman@gmail.com}

\author{Shankar C. Venkataramani$^\dag$}
\address{Department of Mathematics, 617 N. Santa Rita Avenue, University of Arizona, Tucson, AZ 85721}
\thanks{$^\dag$Corresponding author.}
\curraddr{}
\email{shankar@math.arizona.edu}


\date{\today}							

\begin{document}

\begin{abstract}
We develop a theory for distributed branch points and investigate their role in determining the shape and influencing the mechanics of thin hyperbolic objects. We show that branch points are the natural topological defects in hyperbolic sheets, they carry a topological index which gives them a degree of robustness, and they can influence the overall morphology of a hyperbolic surface without concentrating energy. We develop a discrete differential geometric (DDG) approach to study the deformations of hyperbolic objects with distributed branch points. We present evidence that the maximum curvature of surfaces with geodesic radius $R$ containing branch points grow sub-exponentially, $O(e^{c\sqrt{R}})$ in contrast to the exponential growth $O(e^{c' R})$ for surfaces without branch points. We argue that, to optimize norms of the curvature, i.e. the bending energy, distributed branch points are energetically preferred in sufficiently large pseudospherical surfaces. Further, they are distributed so that they lead to fractal-like recursive buckling patterns. 
\end{abstract}

\maketitle

{\scriptsize }

\section{Introduction} \label{sec:intro}

Leaves, flowers, fins, wings and sails are examples of the ubiquity of thin sheets in natural and engineered structures. These objects often display intricate rippling and buckling patterns around their edges. 
\begin{figure}[htbp]
        \begin{subfigure}[t!]{0.3\textwidth}
                \centering
                {\includegraphics[height=4cm, keepaspectratio, trim={0cm, 0, 0cm, 0}, clip]{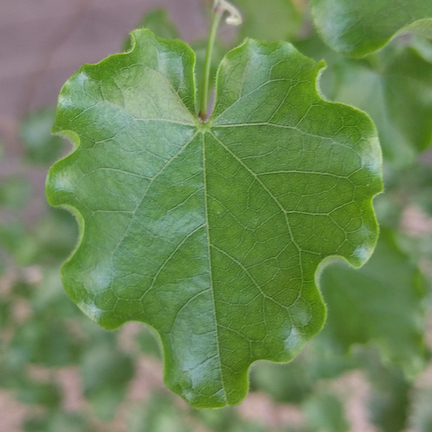}}
                \caption{}
                \label{fig:enr2leaves}
        \end{subfigure}%
        \begin{subfigure}[t!]{0.35\textwidth}
                \centering
                {\includegraphics[height=4cm, keepaspectratio, trim={0cm, 0, 0cm, 0}, clip]{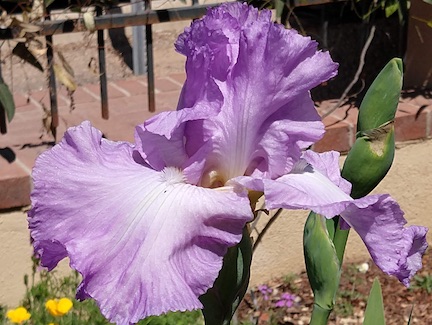}}
                \caption{}
                \label{fig:purpleiris}
         \end{subfigure}%
        \begin{subfigure}[t!]{0.25\textwidth}
                \centering
                {\includegraphics[angle=-90,width=3cm]{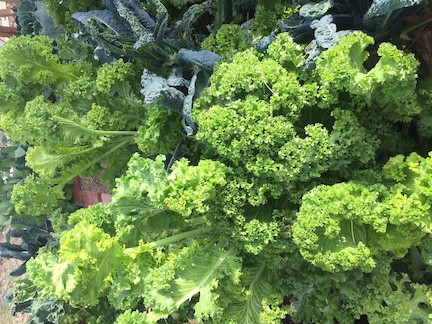}}
                \caption{}
                \label{fig:mustard}                       
                \end{subfigure}%
        \caption{(a) A leaf with regular undulations (photo by TS). (b) An Iris with 3 generations of undulations (photo by SV). (c) Curly mustard leaves with multiple generations of buckling (photo by J Watkins, U. Arizona).
       }
        \label{fig:examplesofnaturalobjects}
\end{figure}
Figure~\ref{fig:examplesofnaturalobjects} displays some of the complex shapes of leaves and flowers that result from such hierarchical, ``multiple-scale" buckling. In the physics literature, a relation between these buckling patterns and the growth of a leaf at its margins was first identified by Nechaev and Voituriez \cite{nechaev2001plant} (See also \cite{sharon2002buckling,eran2004leaves,sharon2007geometrically,LiangMaha2011,sharon2018mechanics}). This phenomenon is not restricted to living organisms, where it might be explained as a genetic trait selected for by evolution; it is seen in torn plastic sheets \cite{sharon2007geometrically}. Also, a wavy pattern can be induced in a naturally flat leaf; Sharon et al. show that application of the growth hormone auxin to the edge of an eggplant leaf, which is naturally flat,  induces growth at the margin, ultimately causing buckling out of plane \cite{eran2004leaves}. 

Qualitatively similar patterns are observed in torn plastic \cite{sharon2002buckling,sharon2007geometrically} and temperature sensitive hydrogels \cite{klein2007shaping,Kim2012Hydrogel}. These patterns, and their bifurcations, have been studied intensively over the last 20 years \cite{sharon2002buckling,marder2003shape,marder2003theory,audoly2003self,klein2007shaping,efrati2009elastic,Klein2011Experimental,Gemmer2013Shape}. 
The changes to the internal structure during the growth of a leaf, or through the stretching of a plastic sheet at a tear, result in surfaces whose intrinsic geometries, i.e. Riemannian metrics, are no longer ``compatible'' with a flat shape; significant external forces compressing the elastic sheet would need to be imposed for the surface to lay flat. The analogy between the localized stretching near the edge of a torn plastic sheet and the preferential growth of leaves along to their edge motivates the need for a purely mechanical explanation for the observed self-similar, fractal-like buckling patterns \cite{marder2003theory,audoly2003self,sharon2007geometrically,Liang2009shape,Efrati2013Metric,EPL_2016}.

Hydrogels have emerged as a useful system for exploring thin sheets with complex geometries in a controllable and reproducible manner \cite{klein2007shaping,efrati2007spontaneous,Kim2012Hydrogel}. Experimental techniques can prescribe a desired Riemannian metric in a hydrogel sheet that is initially flat, but acquires the programmed metric upon ``activation" \cite{klein2007shaping,Kim2012Hydrogel,kim2012thermally}. A variety of environmental stimuli, such as a shining light or temperature changes can activate the programmed metric. A gel sheet that swells more near the center leads to an ultimately spherical shape. Alternatively, if the differential swelling is larger near the margins and reproduces the effect seen in leaves, producing a wavy surface \cite{klein2007shaping,efrati2007spontaneous,Huang2018Differential}. Hydrogels which undergo such controlled shape transitions, due to a switch in the metric, 
have a variety of potential applications in medical devices, micro- and nano-scale robotics and flexible electronics. 

Another ``experimental'' system, less quantitative, but beautifully pairing art and mathematics is `hyperbolic crochet' \cite{henderson2001crocheting,bridges2013,wertheim2015crochet}. Through crochet, artists and mathematicians have rendered embeddings of (subsets of) the hyperbolic plane $\mathbb{H}^2$ in $\mathbb{R}^3$. Hyperbolic crochet is constructed by increasing the perimeter exponentially with the radius. Sprawling hyperbolic crochet provides striking resemblance to sea creatures and plant life and has been exhibited through `The Crochet Coral Reef project' \cite{wertheim2015crochet}. In `Floraform', a project inspired by the differential growth in plant structures and the ruffles of lettuce sea slugs, the authors simulate growth of a thin surface using techniques from differential geometry and physics, to uncover novel design principles and also to create art \cite{floraform2014}.

There is remarkable unity of form in leaves and hyperbolic hydrogels \cite{Huang2018Differential}, in corals and crochet \cite{wertheim2015crochet}, in sea slugs and jewellery made using simulated differential growth \cite{floraform2014}. Why is this so? This is the fundamental question we seek to address in this paper -- {\em Why do systems, with completely different physics, some directed by complex evolutionary processes and others generated by simple mathematical rules, end up with similar fractal-like buckling patterns}? 

A commonly held explanation is that hyperbolic surfaces, i.e. objects whose perimeter grows exponentially with the radius, develop complex buckling patterns because there are no smooth ways to embed them in $\mathbb{R}^3$ without stretching \cite{henderson2001crocheting}. Putative evidence for this picture includes
results that imply a dependence of the buckling wavelength on the thickness of the sheet \cite{audoly2003self,Klein2011Experimental,bella2014metric,vetter2013subdivision} suggesting a competition between localized stretching energy and regularization from bending energy. However, these scaling laws arise from (sometimes implicit) boundary or ``forcing" conditions. There are no proofs (yet) that these scaling laws also apply to free sheets.  Theorems on non-existence \cite{hilbert1901,holmgren1902surfaces} and singularities  \cite{amsler1955surfaces,efimov1964generation} for isometric immersions of complete surfaces with negative curvature are sometimes invoked in this context. This argument, however, is a misunderstanding of 
the results in \cite{hilbert1901,efimov1964generation} which 
apply to {\em complete} surfaces that are necessarily unbounded. Any finite piece of a smooth hyperbolic surface can always be smoothly and isometrically embedded in $\mathbb{R}^3$  \cite{han2006isometric}. 

As we argue in this paper, the answer is somewhat more subtle, and it is tied to the regularity of the allowed configurations of a hyperbolic sheet in $\mathbb{R}^3$. In particular, we demonstrate that the class of $C^{1,1}$ isometric immersions (no stretching, uniformly bounded curvatures that are not necessarily continuous) are ``flexible" while $C^2$ (continuous curvatures) isometric immersions are ``rigid". `Singular' $C^{1,1}$ isometries can have substantially smaller elastic energy than `smooth' $C^2$  isometries, which seems, on the surface, completely counter-intuitive. Further, the organizing principle for minimizing the energy of $C^{1,1}$ isometries is approximate ``local" balance between the principal curvatures \cite{EPL_2016}, and this naturally leads to fractal-like buckling patterns, as we illustrate in this work. The key to the flexibility of $C^{1,1}$ immersions is a novel topological defect in  pseudospherical surfaces-- {\em branch points} \cite{kirchheim2001Rigidity,gemmer2011shape} that are the principal objects of interest in this work.

After a review of non-Euclidean elasticity in \S \ref{sec:noneuclideanelasticity}, we present our main results in \S \ref{sec:branchedsurfaces}, \S \ref{sec:ddg} and \S \ref{sec:applications}. We conclude with a short discussion  of our results and their implications in \S\ref{sec:discussion}. We believe this work will be of interest to readers with diverse backgrounds, so we summarize our key results here to give readers an overview of the entire paper in broadly accessible language. This introduction is necessarily informal, and we refer the readers to the discussion in the body of the paper for the precise mathematical statements. 

We define branch points in Definition~\ref{def:m-star}. At  ``regular points", a surface negative Gauss curvature is saddle-shaped and has 4 `sectors', two above and two below the tangent plane. In contrast, at a branch point, the surface has $2m > 4$ sectors. We construct pseudospherical immersions containing branch points by assembling multiple  sectors together --

\noindent{\bf Prop. \ref{prop:assembly}.} Given $2m \geq 4$ smooth curves $\gamma_i$, originating at a point $p$, tangent to a common plane through $p$, and with alternating torsions $\pm 1$, there is a branched pseudospherical surface, with bounded principal curvatures, that contains (sufficiently small segments of) 
the curves $\gamma_i$.

Our next main result is that branched points are topological defects since they carry a topological charge that cannot be smoothed away. A key preliminary step is Definition~\ref{def:Jp}. that identifies the appropriate quantity which measures the topological charge.

\noindent{\bf Thm.~\ref{thm:branched}.} If a pseudospherical surface $S$ can be approximated in $W^{2,2}_{\mathrm{loc}}$, i.e. the local difference in curvatures as measured by the elastic bending energy can be made as small as desired, through surfaces with bounded curvature and no branch points, then the surface $S$ itself cannot have branch points.

In \S\ref{sec:introducingabranchpoint} we outline a procedure we call surgery, that allows us to add additional branch points to  surfaces (See Lemma~\ref{lem:surgery}). We then generalize the classical sine-Gordon equation for smooth pseudospherical surfaces, $\partial_{uv} \varphi = \sin \varphi$, to surfaces with branch points.

\noindent{\bf Thm. \ref{thm:sg-branched}.} With an appropriate definition of $\varphi(u,v)$, the angle between the asymptotic directions as a function of the asymptotic coordinates, we have 
$$
\oint_{\partial \Gamma} \frac{1}{2} (\partial_v \varphi dv - \partial_u \varphi du) = \iint_\Gamma \sin(\varphi) du dv - \pi \sum_{p_i \in \Gamma} (m_i-2) 
$$
where $\Gamma$ is any domain bounded by asymptotic curves and the correction is the $\pi$ times the sum of the topological charges, of all the branch points contained in $\Gamma$.

In \S \ref{sec:poincare-ddg} we introduce a new class of discrete nets that represent the extrinsic geometry of pseudospherical surfaces (i.e. the second fundamental form) in intrinsic coordinates, and allow for branch points. This is useful in applications to the elasticity of thin sheets, since it naturally discretizes the class of low-energy (isometric) deformations  of a pseudospherical surfaces. Using this discretization, we formulate {\bf Algorithm~\ref{alg:greedy-cuts}}, a greedy algorithm for finding (heuristically) the distribution of branch points that optimizes the elastic energy, i.e. solving the min-max problem of finding $\arg\min_{r} \esssup_{x \in \Omega} |H(x)|$ over immersions $r: \Omega \to \mathbb{R}^3$ with branch points where $H(x)$ is the mean curvature at $r(x)$.

In \S\ref{sec:applications} we present a `physics-style' back of the envelope calculation that allows us to estimate the energy and the number of wrinkles of nearly energy optimal immersions of disks with constant negative curvature, while allowing for branch points. Our arguments reveal the role of the branch points in significantly decreasing the elastic energy, from $\log \inf \mathcal{E} \sim R$ for smooth immersions to $\log \inf \mathcal{E} \sim \sqrt{R}$ for branched immersions of a disk of radius $R$, {\em cf.} Eqs.~\eqref{e2}~and~\eqref{ebs}. We compare our estimates with numerical simulations.

\section{Non-Euclidean elasticity}
\label{sec:noneuclideanelasticity}

We model our elastic bodies as hyperelastic materials, so that the observed configurations are minimizers of an elastic  energy functional. The functional quantifies the elastic energy due to strains in a particular deformed configuration of the body relative to the intrinsic (non-Euclidean) geometry  which can be represented as a Riemannian manifold $(\mathcal{B},\mathbf{G})$. This suggests a candidate for the resulting three-dimensional elastic energy 
\begin{equation}
\mathcal{I}[\tilde{y}] = \int_{\mathcal{B}} \| \partial_i \tilde{y}\cdot \partial_j \tilde{y} - G_{ij}\|^2\, dV,
\label{eqn:threedenergy}
\end{equation}
with $\tilde{y}:\mathcal{B}\to\mathbb{R}^3$ representing the deformation \cite{audoly2002ruban,marder2003theory,efrati2009elastic}. Though Eq.~\eqref{eqn:threedenergy} is arguably a prototypical model elastic energy, this functional is not  appropriate from variational perspective \cite{lewicka2011scaling} 
because of the possibility of fine-scale, orientation-reversing ``folded structures''. An appropriate elastic energy is defined using a polar decomposition of the deformation gradient $\nabla \tilde{y}$ to measure its deviation from an ``energy well" 
$\mathcal{F}(x) = \left\{RA(x): R\in SO(3) \right\}$, 
where $A = \sqrt{\mathbf{G}}$ is the symmetric, positive definite root of the 
Riemannian metric $\mathbf{G}$ \cite{lewicka2011scaling}.  $\mathcal{F}(x)$ contains all the {\em orientation preserving} isometric linear maps, from the   tangent space $T_x\mathcal{B}$ to $\mathbb{R}^3$ and this defines the elastic energy
\begin{equation}
I[\tilde{y}] = \int_{\mathcal{B}} \textrm{dist}^2\left(\nabla \tilde{y}(x), \mathcal{F}(x)\right)\, dx,
\label{eq:hyperelastic}
\end{equation}
The fully 3-dimensional variational problem for \eqref{eq:hyperelastic} is analytically intractable motivating the development of reduced models for shells, plates and rods \cite{love2013treatise,timoshenko1959theory}. For plates,
$$
\mathcal{B} = \Omega \times \left(-\frac{h}{2},\frac{h}{2}\right), \quad G_{ij} = \begin{pmatrix}g_{11} & g_{12} & 0 \\ g_{21} & g_{22} & 0 \\ 0 & 0 & 1 \end{pmatrix},
$$
the F\"oppl-von K\'arm\'an approximation \cite{ciarlet1980justification} is one such asymptotic reduction of the full 3-dimensional system to a 2-dimensional system on the center-surface $\Omega$ in the limit of vanishing thickness $h \to 0$. Here and henceforth $g$ will represent the 2d metric on $\Omega$. For a sheet of thickness $h$, scaling the in- and out-of-plane displacements to be $O(\sqrt{h})$ and $O(h)$ respectively gives an energy functional, called the FvK energy in the physics literature:
\begin{equation}
\mathcal{E}^h = h \, \mathcal{E}_{\textrm{stretching}} + h^3 \,\mathcal{E}_{\textrm{bending}},
\label{FvK}
\end{equation}
The resulting variational formulation, also known as the F\"oppl-von K\'arm\'an (FvK) equations, are coupled PDEs representing the equilibrium conditions associated with the reduced energy and have been used extensively to model thin elastic sheets. Efrati et al. extended the FvK theory to non-Euclidean plates, {\em i.e.} cases where the reference metric $g$ is not the Euclidean metric \cite{efrati2009elastic}. Using the formalism in \cite{efrati2009elastic}, the energy of a non-Euclidean plate  
with elastic modulus $Y$,  Poisson ratio $\nu = 0$, and setting $y = \tilde{y}|_\Omega$ is 
\begin{align}
\mathcal{E}^h 
& = \frac{Yh}{2}\int_\Omega \|dy\cdot dy - g\|^2 dA + \frac{Yh^3}{24}\int_\Omega (4H^2 - 2K) dA.
\label{eqn:finitethicknessenergy}
\end{align}
The first integral measures the stretching energy, quantifying the deviation of the induced metric from an assumed reference metric. The second integral, also known as the Willmore functional, describes the energy due to bending. $H=\frac{\kappa_1 + \kappa_2}{2}$ is the mean curvature and $K = \kappa_1\kappa_2$ is the Gauss curvature,
where $\kappa_1$ and $\kappa_2$ are the principal curvatures of the immersion $y:\Omega \to \mathbb{R}^3$. In this work $K=-1$ and we expect $\mathcal{E}^h \sim h^3$ if $y$ is an isometry.

The energy functional~\eqref{FvK} obtains from making an ansatz ``lifting" an immersion  $y: \Omega \to \mathbb{R}^3$ of the center surface to a deformation $\tilde{y}^h: \mathcal{B} \to \mathbb{R}^3$ given by the {\em Kirchhoff-Love extension} that maps fibers orthogonal to the center surface $\Omega$ in $\mathcal{B}$ to fibers orthogonal to the image $y(\Omega)$ in $\mathbb{R}^3$ (isometrically for $\nu = 0$). In contrast, rigorous derivations of the $h \to 0$ limit energy for plates are ansatz-free and are obtained through $\Gamma$--convergence \cite{friesecke2002foppl,friesecke2006hierarchy}. In the $\Gamma$--convergence approach, one assumes that, for a sequence of mappings $\tilde{y}^h: \Omega \times[-\frac{h}{2}, \frac{h}{2}] \to \mathbb{R}^3$, the elastic energy satisfies a uniform bound $h^{-\alpha} I[\tilde{y}^h] \leq C$, where $I[\cdot]$ is the ``bulk" elastic energy defined in~\eqref{eq:hyperelastic}. With no further assumptions, one shows that a subsequence of the immersions $\tilde{y}^h$ (appropriately rescaled) converges (in an appropriate sense). One then defines a space of limit configurations and a limit energy $\bar{E}$, so that for any allowed limit configuration $\bar{y}$, one can recover a sequence of configurations $\tilde{y}^h$ such that $\tilde{y}^h \to \bar{y}, h^{-\alpha} I[\tilde{y}^h] \to \bar{E}[\bar{y}]$. The limiting space and the limit energy can depend on $\alpha$, and, in general, one obtains  a hierarchy  of limiting elastic energy functionals, distinguished by the scaling of the energy with $h$ \cite{friesecke2006hierarchy,lewicka2014models}.

In our work, we are in the scaling regime $I[\tilde{y}^h] \leq C h^3$, and the corresponding limit theory is called the {\em Kirchhoff plate theory} in the literature on rigorous dimension reduction for slender elastic objects \cite{friesecke2006hierarchy,Schmidt2007Plate,lewicka2011scaling}. The scaled energy $h^{-3} I$ converges
\begin{equation}
\frac{24 h^{-3}}{Y} I[y] \stackrel{\Gamma}{\longrightarrow}   \mathcal{E}_2[y]= \begin{cases} \int (\kappa_1^2 + \kappa_2^2) \, dA & \textrm{if } y \in W^{2,2}, dy\cdot dy \equiv g,\\
+\infty &  \textrm{otherwise}.
\end{cases}
\label{eqn:gammalimit}
\end{equation}
to the isometry restricted Willmore energy, for various problems in incompatible elasticity of thin objects \cite{Schmidt2007Minimal,Schmidt2007Plate,lewicka2011scaling,kupferman2014riemannian,bhattacharya2016plates}. 
In this work, we will also consider an alternative bending energy, the isometry restricted {\em max curvature} $\mathcal{E}_\infty[y] = \max_\Omega (|\kappa_1|,|\kappa_2|)$ for $y \in W^{2,\infty}, dy\cdot dy \equiv g$ and $+\infty$ otherwise. For all bounded domains, the limit (Willmore) energy $\mathcal{E}_2$ is bounded by (the square of) the $\mathcal{E}_\infty$, so finding configurations with $\mathcal{E}_\infty$ finite is sufficient for showing the existence of finite Willmore energy isometries. We also note that $\kappa_1 \kappa_2 =-1$ a.e. for $C^{1,1}$ surfaces with $K=-1$. Consequently, 
$$
2 |H(x)| = |\kappa_1(x) + \kappa_2(x)| \leq \max(|\kappa_1(x)|,|\kappa_2(x)|) \leq 2|H(x)| + 1
$$
so that, for surfaces of constant curvature, so the max-curvature energy $\mathcal{E}_\infty$ is essentially the same as the max mean curvature $\max_{x \in \Omega} |H(x)|$.

A significant obstruction to finding these configurations is the \emph{singular edge} ; see Example~\ref{ex:bobbin} and Fig.~\ref{fig:bobbin}. The singular edge  is  an example of a {\em cuspidal edge singularity}, and is a generic feature of isometric immersions of $\mathbb{H}^2$ into $\mathbb{R}^3$ \cite{amsler1955surfaces,Ishikawa2006Singularities}.  One of the principal curvatures diverges at the singular edge so the $W^{2,\infty}$ energy is locally unbounded. 
As we show elsewhere, the Willmore energy also diverges in  any neighborhood of a point on the singular edge.
Our principal concern in this work will therefore be the question of how to evade or stave off the occurrence of a singular edge. 

The question of isometric embeddings and immersions of a Riemannian  $2$-manifold $(\Omega,g)$ as a surface in $\mathbb{R}^{3}$ has a long history, reviewed in \cite[Chaps.~2~\&~3,~\S 4.2]{han2006isometric}. We are specifically interested in the case of pseudospherical surfaces, i.e. when $g$ has constant negative curvature \cite[Chap. 4]{stoker}.  In 1901, Hilbert showed that there exists no geodesically complete, analytic immersion into $\mathbb{R}^3$ of a metric with constant negative curvature \cite{hilbert1901}. This result was later extended by Efimov to $C^2$ isometric immersions into $\mathbb{R}^3$  for any metric with negative curvature bounded away from zero \cite{efimov1963impossibility,milnor1972efimov}:
\begin{theorem*}[Efimov]
No surface with negative Gauss curvature bounded away from zero $K\leq - \delta <0$ can be $C^2$ immersed in Euclidean 3-space so as to be complete in the induced Riemannian metric.
\end{theorem*}

Alternatively, Nash \cite{nash1954c1} and Kuiper \cite{kuiper1955c1}  showed that, for a general metric $g$, there exists a $C^1$ isometric immersion, indeed even an embedding:
\begin{theorem*}[Nash-Kuiper]
Let $(\mathcal{M}, g)$ be an $m$-dimensional Riemannian manifold and $f:\mathcal{M}\to\mathbb{R}^n$ a short immersion (resp. embedding), where $n\geq m+1$. Given an $\epsilon > 0$,  there exists an isometric immersion (resp. embedding) $f_{\varepsilon}$ of class $C^1$ satisfying
\begin{equation}
g(v, w) = \langle df_{\varepsilon}(v), df_{\varepsilon}(w) \rangle,
\end{equation}
which is uniformly $\varepsilon$-close to $f$ in the Euclidean norm on $\mathbb{R}^n$:
\begin{equation}
\|f(x) - f_{\varepsilon}(x)\| < \varepsilon\textrm{ for all } x\in\mathcal{M}.
\end{equation}
\end{theorem*}
The juxtaposition of these two results provides a strong motivation to explore isometric immersions with regularities between $C^1$ and $C^2$. There is a substantial body of work investigating the existence of isometric immersions of surfaces into $\mathbb{R}^3$ with H\"{o}lder regularity in the class $C^{1,\alpha}$ \cite{Borisov1959On,Borisov2004Irregular,Conti2012hPrinciple,DeLellis2018immersions,DeLellis2020Isometric}, with proofs of flexibility for $\alpha < 1/5$ and rigidity for $\alpha > 2/3$.  Our interest is in isometric immersions with $W^{2,2}$ Sobolev regularity \cite{Pakzad2004Sobolev}, motivated by the need to define a meaningful bending (i.e. Willmore) energy for the immersion, as is clear from the reduced energy \eqref{eqn:gammalimit}. Provided that the space of $W^{2,2}$-isometric immersions is nonempty, containing potentially many immersions, we use the elastic energy as a selection process: the observed surface is the isometric immersion which minimizes the bending energy.

\begin{remark}
Bella and Kohn prove that wrinkles do arise through a competition between stretching and bending energies, for $h > 0$, with additional ``forcing" conditions that restrict the class of allowed deformations~\cite[Thm.~1]{bella2014metric}. In this circumstance, the $W^{2,2}$ energy of minimizers does not stay bounded as $h \to 0$, i.e. the limiting isometries are not $W^{2,2}$.

We consider a different scenario in this work, namely {\em free sheets} with no imposed forces or boundary conditions. To analyze equilibrium states we have to impose boundary conditions that are appropriate for isometric immersions of free sheets, namely zero net forces and moments \cite{Guven2019Isometric}. In this work, we take a variational perspective for the  problem of minimizing~\eqref{eqn:gammalimit}, or the simpler problem of minimizing $\mathcal{E}_\infty = \kappa_{\max}$. Our candidate states are therefore ``test functions" for the energy and, unlike equilibria, they need neither satisfy the appropriate Euler-Lagrange equations nor the corresponding boundary conditions. 
\end{remark}

\section{Pseudospherical surfaces with branch points}\label{sec:branchedsurfaces}

The preceding discussion highlights the role of the {\em regularity} of isometries. Beyond the existence/non-existence of isometries, it is crucial whether a candidate isometry is in $W^{2,2}$. This motivates the following problem: $(\Omega,g)$ is a Riemannian 2-manifold.  
\begin{equation}
\mbox{Find } y: \Omega \to \mathbb{R}^3  \mbox{ such that } y \in W^{2,2}_{\text{loc}}(\Omega,\mathbb{R}^3), \quad dy\cdot dy = g \,\,\mathrm{a.e.} 
\label{W22isometry}
\end{equation}
If $y: \Omega \to \mathbb{R}^3$ is $C^1$, the Gauss normal map is given by $\displaystyle{N = \frac{\partial_1 y \times \partial_2 y}{\|\partial_1 y \times \partial_2 y\|}}$  with $\displaystyle{\partial_{i} = \frac{\partial}{\partial x^i}}$  for (arbitrary) coordinates $(x^1,x^2)$ on $\Omega$. If $y$ and $g$ are $C^2$, it follows that $N$ is $C^1$ and Gauss' {\em Theorema Egregium} implies that~\eqref{W22isometry} is equivalent to the Monge-Ampere Exterior differential system (EDS) \cite[\S6.4]{ivey2003cartan}:
\begin{equation}
N \cdot dy = 0, \qquad N^*(d\Omega) = \kappa \,dA, \qquad \kappa \equiv \kappa[g] \mbox{ is determined by } g,
\label{thmegrg}
\end{equation}
where $d \Omega$ is the area form on the sphere $S^2$ and $\kappa$ is the Gauss curvature.

Classical results in differential geometry imply that smooth solutions of~\eqref{thmegrg} with $\kappa < 0$ are hyperbolic surfaces and locally saddle shaped.
In contrast, the curly mustard leaf in Fig.~\ref{fig:mustard} is ``frilly", i.e. buckled on multiple scales with a wavelength that refines (``sub-wrinkles") near the edge \cite{eran2004leaves}.
This ``looks" very unlike smooth saddles  ({\em cf}. Fig.~\ref{fig:saddle}). 

 If $\Omega \subset \mathbb{R}^2$ is a bounded domain with a smooth boundary, and $g$ is a smooth metric on $\Omega$ with negative curvature, $g$ can be extended to a smooth metric $\bar{g}$ on $\mathbb{R}^2$ with Gauss curvature $\kappa[\bar{g}] < 0$ decaying (as rapidly as desired) at infinity. The existence of isometric immersions into $\mathbb{R}^3$, of smooth metrics with decaying negative curvature \cite{hong1993realization}, therefore implies that bounded smooth hyperbolic surfaces can be smoothly and isometrically embedded in $\mathbb{R}^3$. A smooth ($C^2$ is sufficient) hyperbolic surface cannot refine its buckling pattern and is thus ``non-frilly", as we show in \S\ref{sec:kminusonemonkeysaddle}. Why do we see frilly shapes in natural surfaces, as in Fig.~\ref{fig:mustard}, rather than a smooth saddle (see Fig.~\ref{fig:saddle})?

We have addressed aspects of this puzzle in recent work \cite{gemmer2011shape,gemmer2012defects,Gemmer2013Shape,EPL_2016,acharya2020continuum}  and find that frilly surfaces, somewhat counterintuitively, can have {\em smaller} bending energy than the smooth saddle, despite being (seemingly) rougher. It is true that $C^2$ hyperbolic surfaces are saddle-like near every point. A key result in this work is the identification of a topological invariant, the winding number (ramification index) of the normal map at a branch point, that distinguishes sub-wrinkled surfaces from saddles locally (See Lemma~\ref{lem:covering-number} and Fig.~\ref{fig:winding_number_normalfield}). With branch points, the surfaces are only $C^{1,1}$, like the monkey saddle in Fig.~\ref{fig:mkysaddle}, but the gain the additional flexibility to 
refine their buckling pattern and thus lower their energy \cite{EPL_2016}. This flexibility {\em is not available} to smooth saddles, and constitutes a key property of surfaces with branch points \cite{EPL_2016}. 

The additional flexibility for $C^{1,1}$ immersions of hyperbolic surfaces has been explored since the 1960s. Rozendorn discussed the branched hyperbolic paraboloid as an important example of a $C^{1,1}$ hyperbolic surfaces  \cite{Rozendorn-Chapter-92}, and constructed $C^{1,1}$ immersions of geodesically complete, uniformly negatively curved ($K \leq -\delta < 0$) surfaces that are smooth except at finitely many points \cite{Rozendorn1962Complete,Rozendorn1966Weakly,Rozendorn-Chapter-92}. In contrast to Rozendorn's construction \cite{Rozendorn1962Complete}, with a focus on  minimizing the ``singular set" of $C^{1,1}$ points and leaving the metric ``free", the constructions in \cite{EPL_2016,gemmer2011shape} exactly preserve a prescribed metric, but need ``larger" sets of singular $C^{1,1}$ points. The goals for this approach include enlarging the domain that can be immersed isometrically into $\mathbb{R}^3$ or optimizing the bending energy over isometries. In this work we follow the latter approach and seek $C^{1,1}$ isometric immersions of a prescribed metric, namely one with constant negative curvature $K=-1$. 

\begin{definition}[Hyperbolic plane] The hyperbolic plane $\mathbb{H}^2$ is the maximally symmetric, simply connected, 2-manifold with with constant negative curvature $-1$. An explicit model for this space is the Poincar\'e disk $x^2+y^2 <1$ with the metric $\displaystyle{g = \frac{4(dx^2+dy^2)}{(1-(x^2+y^2))^2}}$. 
\label{def:H2}
\end{definition}

\subsection{Pseudospherical surfaces} \label{sec:geometry}

Here and henceforth we will use the adjective pseudospherical to mean ``pertaining to subsets of the Hyperbolic plane". We will build branched $C^{1,1}$ pseudospherical surfaces in $\mathbb{R}^3$ by patching together $C^{2}$ immersions of subsets of $\mathbb{H}^2$, such that the pieces join with continuous tangent planes. To this end, we collect and also extend a few properties of $C^2$ pseudospherical surfaces (See \cite[Chaps. V \& VI]{eisenhart1909treatise} \cite[\S 1.1 \& \S 1.2]{rogers2002backlund} and \cite{dorfmeister2016pseudospherical}). 
\begin{enumerate}[leftmargin=*,label={(\Alph*)}]
\item Every $C^2$ immersion with $K=-1$ admits a pair of asymptotic coordinates $(u,v)$ (locally) so that parametrized surface $(u,v) \mapsto r(u,v)$ satisfies $r_u \times r_v \neq 0, N\cdot r_{uu} = N \cdot r_{vv} = 0$ where $N = \pm \, r_u \times r_v/\|r_u \times r_v\|$ \cite{hartman1951asymptotic}. 
The sign choice in the definition of $N$ is immaterial if $\|r_u \times r_v\|$ never vanishes. 
\item By the Beltrami-Enneper theorem \cite[Chap. V]{eisenhart1909treatise}, the unit-speed asymptotic curves $r(\cdot,v_0)$ and $r(u_0,\cdot)$ have constant torsions $\pm 1$. 
We choose the $u$ and $v$ coordinates so that the corresponding asymptotic curves have torsions -1 and +1 respectively. Since $r_{u} \perp N$ and $r_{v} \perp N$, $(r_u, N \times r_u, N)$ is an orthonormal Frenet frame for the $u$-asymptotic lines $r(\cdot,v_0)$ and $(r_v, N \times r_v, N)$ is a frame for the $v$-asymptotic lines. The Frenet-Serret formulae \cite[Chap. V]{eisenhart1909treatise} read
\begin{align}
\partial_u \begin{pmatrix}
r_u \\ N \times r_u  \\ N
\end{pmatrix} & = \begin{pmatrix}
0 & \kappa^u & 0 \\ -\kappa^u & 0 & -1 \\ 0 & 1 & 0
\end{pmatrix}\begin{pmatrix}
r_u \\  N \times r_u \\ N
\end{pmatrix}, 
\nonumber \\
\partial_v \begin{pmatrix}
r_v \\ N \times r_v \\ N
\end{pmatrix} & = \begin{pmatrix}
0 & \kappa^v & 0 \\ -\kappa^v & 0 & 1 \\ 0 & -1 & 0
\end{pmatrix} \begin{pmatrix}
r_v \\ N \times r_v  \\ N
\end{pmatrix}.
\label{eq:frames}
\end{align}
$\kappa^u$ and $\kappa^v$ are the geodesic curvatures of the $u$ and $v$ asymptotic lines.
\item The Frenet-Serret equations yield $N_u = N \times r_u$ so $(r_u,N_u,N)$ is a right-handed orthonormal frame. Similarly, $(r_v,-N_v,N)$ is a right-handed orthonormal frame. This gives the {\em Lelieuvre formulae} \cite[\S 1.6]{rogers2002backlund}
\begin{align}
r_u(u,v) & = N_u(u,v) \times N(u,v), \nonumber \\
r_v(u,v) & = -N_{v}(u,v) \times N(u,v).
\label{eq:lelieuvre}
\end{align}
\item The Lelieuvre equations are consistent if and only if $\partial_v(r_u) = \partial_u(r_v)$ which is equivalent to the condition that the normal field $(u,v) \mapsto N(u,v)$ is {\em Lorentz harmonic}\begin{equation}
N \times N_{uv} = 0.
\label{eq:lorentz_harmonic}
\end{equation}  
It immediately follows that $r_{uv} = N_u \times N_v$.
\item \label{item:symmetries} Note that Eqs.~\eqref{eq:lelieuvre}~and~\eqref{eq:lorentz_harmonic} and the signs of the torsions in~\eqref{eq:frames} are invariant under three separate symmetries, $N \to -N, u \to -u$ or $v \to -v$.  Also, the transformations $u \to -u, v\to -v$  or $N \to - N$,  respectively, reverse the sign of the geodesic curvature $\kappa^u$, reverse the sign of $\kappa^v$, and reverse the signs of both $\kappa^u$ and $\kappa^v$ in~\eqref{eq:frames}.
\item Note that $-u$ (resp. $-v$) is as much a valid asymptotic coordinate as is $u$ (resp. $v$). This is not an issue with global (smooth) asymptotic coordinates, but will be an issue for the branched surfaces that are our principal objects of interest.  

We will define $N$ so that it is continuous in situations where the underlying surface is $C^1$, {\em independent of the specific asymptotic parametrization}. Let $\omega$ be an orientation (a non-vanishing 2 form)  on this surface. If the surface is a graph $(x_1,x_2,w(x_1,x_2))$, a canonical choice is $\omega = dx_1 \wedge dx_2$. We define the normal $N$ so that the orientation $\omega$ on the surface 
is consistent with the cross product in the ambient space $\mathbb{R}^3$ 
i.e. $\omega(X,Y) =  \beta (X \times Y) \cdot N$ for all vector fields $X,Y$ tangential to the surface and a strictly positive function $\beta$. This is equivalent to defining
\begin{equation}
N \equiv N^\omega = \mathrm{sign}(\omega(r_u,r_v)) \frac{r_u \times r_v}{\|r_u \times r_v\|} = \sigma \frac{r_u \times r_v}{\|r_u \times r_v\|}
\label{eq:Nfront}
\end{equation}
where we have defined   $\sigma \equiv \mathrm{sign}(\omega(r_u,r_v))$ to keep the notation compact. It is easy to see that this definition of $N$ is insensitive to  ``flips" $u \to - u$ or $v \to -v$ in the asymptotic parametrization. A related issue is addressed in the definition of the normal $N_{\text{front}}$ for a {\em pseudospherical front} in Ref.~\cite{dorfmeister2016pseudospherical}, where the consideration was the potential vanishing of $\|r_u \times r_v\|$.

\item \label{item:ruv} If we define the angle between the asymptotic directions by $\cos \varphi = r_u \cdot r_v$, this definition is not invariant under the flips $u \to -u$ or $v \to -v$.  We therefore pick an ``invariant" definition for the angle between the asymptotic directions by
\begin{align}
\cos(\varphi) & = \sigma r_u \cdot r_v = 
- \sigma N_u \cdot N_v, \nonumber \\
\sin(\varphi) & =  
\sigma (r_u \times r_v) \cdot N 
\nonumber \\
& =  \|r_u \times r_v\|  \nonumber \\
r_{uv} & = N_u \times N_v = 
- \sigma \sin \varphi N
\label{eq:def_phi}
\end{align} 
For this definition, $\sin \varphi \geq 0$ so $0 \leq \varphi \leq \pi$. $r$ is an immersion only if $r_u$ and $r_v$ are linearly independent, so this precludes $\varphi$ from attaining the values $0$ or $\pi$ on a smooth pseudospherical surface. 
Initially, we work on open sets where 
$\omega(r_u,r_v)$ does not change sign and $\|r_u \times r_v\|$ is nonvanishing.
\item In terms of this angle $\varphi$ and the normal $N = N^\omega$, the first and second fundamental forms of the pseudospherical surface are given by
\begin{align}
g & = dr \cdot dr = du^2 + 2 \sigma \cos \varphi \,du dv + dv^2  \nonumber \\
h & = dN^\omega \cdot dr = -2 \sigma  \sin \varphi\, du dv
\label{eq:metrics}
\end{align}
\item \label{item:moutard} $N_u = N \times r_u$ and $N_v = -N \times r_v$ are in the plane perpendicular to $N$ that is spanned by $r_u,r_v$. Indeed $N_u$ is obtained by rotating $r_u$ by $\pi/2$ and $N_v$ is $r_v$ rotated by $-\pi/2$.  Differentiating, and using~\eqref{eq:def_phi}, we get 
\begin{align}
N_{uv} & = N_v \times r_u = - (N \times r_v) \times r_u \nonumber \\
& = N(r_u \cdot r_v) - r_v(r_u \cdot N) =  \sigma \cos \varphi N = - (N_u \cdot N_v) N
\label{eq:moutard} 
\end{align}
\item To extract all the compatibility conditions encoded in~\eqref{eq:frames}, we also need the derivatives of the Frenet frame for the $u$-lines with respect to $v$ and vice versa. Recognizing that $N \times r_u = N_u$ and combining the results in the previous items, we have
$$
\partial_v \begin{pmatrix}
r_u \\ N \times r_u  \\ N
\end{pmatrix}  = \sigma \begin{pmatrix}
0 & 0 &  -\sin \varphi \\ 0 & 0 & \cos \varphi \\ \sin \varphi & -\cos \varphi & 0
\end{pmatrix}\begin{pmatrix}
r_u \\  N \times r_u \\ N
\end{pmatrix}. 
$$
Writing these equations abstractly as $\partial_u F^u = A F^u, \partial_v F^u = B F^u$, where $F^u$ denotes the frame $(r_u,N_u,N)$, compatibility $\partial_v(\partial_u F^u) = \partial_u(\partial_v F^u)$ is equivalent to the {\em zero-curvature condition} $\partial_v A - \partial_u B + [A,B] = 0$ \cite[\S 1.2]{rogers2002backlund}. Computing the matrix entries for this system, and the corresponding 
system for the frame $F^v$, we get
\begin{align}
\kappa^u  = -\partial_u \varphi,  \quad \kappa^v & = \partial_v \varphi, \nonumber  \\
-\partial_v(\kappa^u) = \partial_u(\kappa^v) = \varphi_{uv} & = \sigma \sin \varphi, \label{eq:sg}
\end{align}
the {\em Sine-Gordon equation} for $\varphi$ and  relations between the geodesic curvatures $\kappa^u,\kappa^v$ of the asymptotic curves and the derivatives of $\varphi$. In obtaining this equation, we have assumed that $\sigma$ is a constant, so this only applies to open sets where $\omega(r_u,r_v)$ does not change sign.
In \S\ref{sec:sinegordon} 
we generalize the Sine-Gordon equation to situations where $\sigma$ can change sign (see Theorem~\ref{thm:sg-branched}).
\end{enumerate}

We are now in position to define the basic building block of a branched pseudospherical surface. We will follow the discussion 
 in Ref.~\cite{dorfmeister2016pseudospherical}:
\begin{definition} A function $(u,v) \mapsto f(u,v) \in \mathbb{R}^n$ is $C^{1M}$ if each component is $C^1$, and has continuous mixed partial derivatives $f_{uv} = f_{vu}$ on the domain of $f$.
\end{definition}
Note that $C^{1M}$ functions are not necessarily $C^2$ and neither  $f_{uu}$ nor $f_{vv}$ needs to exist. Also, a smooth reparametrization $(u,v) = g(r,s)$ of a $C^{1M}$ function $f$ can yield a function $h(r,z) = f\circ g(r,s)$ that is {\em not} $C^{1M}$ \cite{dorfmeister2016pseudospherical}.

\begin{definition} Let $D \subseteq \mathbb{R}^2$ be equipped with global coordinates $(u,v)$. A $C^{1M}$ mapping $N:D \to S^2$ is {\em weakly (Lorentz) harmonic} if 
\begin{enumerate}
\item $N_u \cdot N_u > 0$ and $N_v \cdot N_v > 0$ on $D$.
\item  $N$  is {\em Moutard}, i.e. there is a continuous function $\nu:D \to \mathbb{R}$ such that $N_{uv} = N_{vu} = \nu N$ (\cite[Thm. 1.12]{bobenko2008bdiscrete}). 
\end{enumerate}
\end{definition}

Weakly harmonic mappings $D \to S^2$ allows us to generalize the class of smooth pseudospherical surfaces \cite{dorfmeister2016pseudospherical}. In particular, if $D$ is simply connected and $N:D \to S^2$ is weakly harmonic, then there is a corresponding {\em pseudospherical front}, or PS-front for short \cite{dorfmeister2016pseudospherical}, i.e. a $C^{1M}$ solution $r:D \to \mathbb{R}^3$ to the Lelieuvre equations~\eqref{eq:lelieuvre}, that is weakly regular, i.e. $r_u \cdot r_u > 0, r_v \cdot r_v > 0$.  PS-fronts allow for the possibility of singularities, i.e. sets where $r$ is not an immersion, and classical examples include the pseudosphere (see \cite[\S 6]{dorfmeister2016pseudospherical}), and Minding's bobbin, as we discuss further in Ex.~\ref{ex:bobbin}. 

We also have a necessary and sufficient condition for ruling out such singularities -- $r$ is an immersion at every point where $N$ is an immersion, i.e. $N_u \times N_v \neq 0$ \cite{dorfmeister2016pseudospherical}. 

\begin{example}A Minding's bobbin, depicted in Fig.~\ref{fig:bobbin}, is a surface of revolution given in cylindrical polar coordinates $(\rho,\theta,z)$  by $\rho(s) = \kappa^{-1} \cosh(s), z(0) = 0, z'(s)^2 + \rho'(s)^2 = 1$, where $s$ is the arclength along a meridian and $\kappa$ is the curvature of the `throat' of the bobbin, the equatorial circle $s=0$. The induced metric is $g = ds^2 + \rho^2 d \theta^2$ and the corresponding Gauss curvature is $K = - \rho''(s)/\rho(s) = -1$. The maximal extension of a Minding's bobbin has a singular edge at a finite distance from the equator, since it cannot be extended smoothly beyond $s = \pm L, L= \arcsinh(\kappa)$ where $\rho'(s) = \pm 1, z'(s) = 0$.
Minding's bobbin has the topology of a cylinder $S^1 \times (-L,L)$ and its universal cover  is a `strip' $\mathbb{R} \times (-L,L)$. The diameter of any geodesic disk than can be smoothly and  isometrically embedded in the universal cover is therefore bounded by $2 \arcsinh(\kappa)$ \cite{gemmer2011shape}.
\begin{figure}[ht]
        \centering
      \includegraphics[width=9cm, trim={1cm, 2cm, 1.5cm, 1.5cm}, clip]{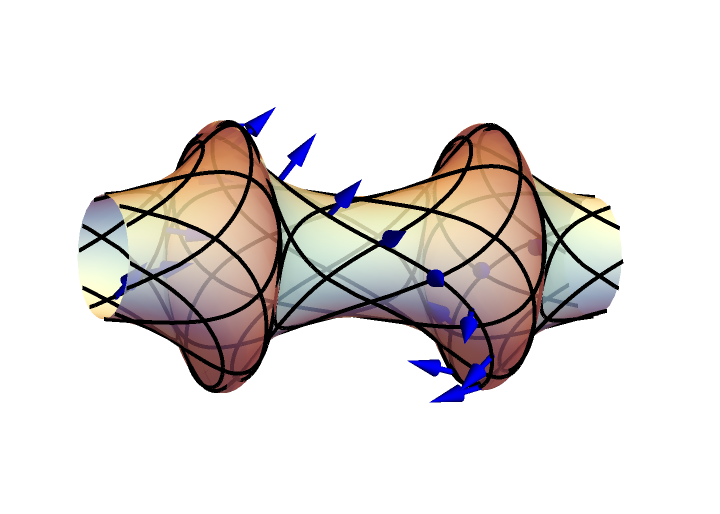}
         \caption{Minding's bobbin with smooth asymptotic curves and cuspidal singular edges.  The normal $N^\omega$ is also shown along an asymptotic curve.}
        \label{fig:bobbin}
\end{figure}
\label{ex:bobbin}
\end{example}
 {
In order to get a diameter $2R$, it follows that the max curvature $\mathcal{E}_\infty > \kappa > \sinh(R)$. Note that, this bound obtains from the throat, and not, as one might have  imagined, from the region near the singular edge.
The longitudinal curvature is given by $\frac{d}{ds} \arcsin \rho'(s) = \frac{\cosh(s)}{\sqrt{\kappa^2 - \sinh^2(s)}}$, and it diverges as the distance to the singular edge to the power $-\frac{1}{2}$ \cite{EPL_2016}. In particular, the Willmore energy also diverges, logarithmically, on any neighborhood of a point on the singular edge. For bobbins that can ``contain" a disk with radius $R$, we have,
\begin{equation}
\inf \mathcal{E}_\infty \geq \inf_{\kappa \geq \sinh(R)} \max\left(\kappa, \cosh(R)(\kappa^2 - \sinh^2(R))^{-1/2} \right),
\label{eq:c2energy}
\end{equation}
and optimization requires a ``global" balance between the `azimuthal' principal curvature at the throat and the `longitudinal' principal curvature near the edge.

Reflecting a pseudospherical surface of revolution about a plane through antipodal meridians preserves the arc length parameter $s \mapsto s$, inverts the torsion so $u$-asymptotic curves map to $v$-asymptotic curves and vice versa, and also inverts angular derivatives $\partial_\theta \mapsto -\partial_\theta$. It therefore follows that the vector $\partial_s \parallel (\partial_u + \partial_v)$ and $\partial_\theta \parallel (\partial_u - \partial_v)$. Indeed, more is true. The fact that the angular separation in $\theta$ between two $u$- (or $v$-) asymptotic curves is the same {\em at any `height'} $z(s)$ (equivalently independent of the arc-length coordinate $s$) implies that $\theta \propto u-v$ for any pseudospherical surface of revolution. Consequently, we can choose $u,v$ such that $r_u\cdot r_u = r_v \cdot r_v = 1, s = s(u+v), \theta = \alpha u- \alpha v$ for some constant $\alpha$.

 With these `normalizations' for the asymptotic coordinates $u$ and $v$, Minding's bobbin can be expressed as $s = s(u+v)$ in terms of elliptic functions \cite[\S 21]{gray1998modern},\cite{gemmer2011shape}. Rather than recapitulate the exact solutions, our goal here is to illustrate various features of PS-fronts using Minding's bobbin as an example.

$\partial_\theta = \frac{1}{2\alpha} (\partial_u - \partial_v)$ is the Killing vector generating the azimuthal symmetry. For scalar quantities $q \in \{s,\varphi,\sigma\}$, invariance under this symmetry implies $q = q(u+v)$. 
Comparing the metric $g=ds^2 + \kappa^{-2} \cosh^2(s) d \theta^2$ with the expression in asymptotic coordinates
$$
g = du^2 + 2 \sigma \cos \varphi \, du dv + dv^2 = \begin{cases} \cos^2 \frac{\varphi}{2} (du+dv)^2 + \sin^2 \frac{\varphi}{2} (du-dv)^2 & \sigma = +1, \\ \sin^2 \frac{\varphi}{2} (du+dv)^2 + \cos^2 \frac{\varphi}{2} (du-dv)^2 & \sigma = -1, \end{cases}
$$
we get, after setting $\xi = u+v$, 
\begin{align}
\frac{ds}{d\xi} & = \frac{1+\sigma}{2} \cos \frac{\varphi}{2} +\frac{1-\sigma}{2} \sin \frac{\varphi}{2}  \nonumber \\
 \frac{\alpha}{\kappa} \cosh(s(\xi)) & = \frac{1+\sigma}{2} \sin \frac{\varphi}{2} +\frac{1-\sigma}{2} \cos \frac{\varphi}{2} 
\label{eq:bobbin}
\end{align}
We can determine the constant $\alpha$ by imposing the requirement that, at the singular edge, whether approached from a region with $\sigma =1$ or from a region with $\sigma = -1$, we should get $ \frac{\alpha}{\kappa} \cosh(s(\xi)) \to 1$. This suggests setting $\alpha = \frac{\kappa}{\sqrt{\kappa^2+1}}$ in~\eqref{eq:bobbin} will yield a pseudospherical surface of revolution with a profile $\rho(s) = \kappa^{-1} \cosh(s)$. This is indeed true as we now prove:

\noindent{\bf Lemma:} Let $\kappa > 0, \sigma \in \{-1,1\}$ be given, and let $s(\xi)$ be a solution to the ODE
\begin{equation}
\left(\frac{ds}{d\xi}\right)^2 + \frac{1}{\kappa^2 + 1} \cosh^2(s) = 1.
\label{eq:econs}
\end{equation}
Then, on domains where $s'(u+v) \neq 0$,  
\begin{equation}
\varphi(u,v) =  (1+\sigma) \arcsin \left(\frac{\cosh(s(u+v))}{\sqrt{\kappa^2+1}}\right) +  (1-\sigma) \arccos \left(\frac{\cosh(s(u+v))}{\sqrt{\kappa^2+1}}\right)
\label{eq:sg-soln}
\end{equation}
solves the sine-Gordon equation $\partial_{uv} \varphi = \sigma \sin \varphi$.

\begin{proof}
It is straightforward to verify that any solution of~\eqref{eq:econs} followed by a definition of $\varphi$ through~\eqref{eq:sg-soln} will give $s(\xi),\varphi$ that satisfy~\eqref{eq:bobbin}. These solutions are smooth whenever $\sigma$ is smooth, i.e. constant. Multiplying the two equations in~\eqref{eq:bobbin} yields
$$
\frac{d}{d\xi} \left[\frac{\sinh(s(\xi))}{\sqrt{\kappa^2+1}}\right] = \frac{\cosh(s(\xi))}{\sqrt{\kappa^2+1}} \left(\frac{ds}{d\xi}\right) = \frac{1}{2} \sin \varphi.
$$
Differentiating the second equation in~\eqref{eq:bobbin} assuming $\sigma$ is locally constant and dividing by $s'(\xi) \neq 0$ from the first equation gives
$$
\frac{\sinh(s(\xi))}{\sqrt{\kappa^2+1}} = \frac{ \frac{1+\sigma}{2} \cos \frac{\varphi}{2} -\frac{1-\sigma}{2} \sin \frac{\varphi}{2} }{ \frac{1+\sigma}{2} \cos \frac{\varphi}{2} +\frac{1-\sigma}{2} \sin \frac{\varphi}{2}} \left(\frac{\varphi'(\xi)}{2}\right) = \sigma \frac{\varphi'(\xi)}{2} \quad \because \sigma \in \{-1,1\}
$$
Combining these two equations, we get $\partial_{uv} \varphi = \varphi''(u+v)  = \sigma \sin \varphi.$
\end{proof}
Note that~\eqref{eq:econs} is the statement of conservation for an energy for a unit mass particle moving in a potential $V(s) = \frac{1}{2(\kappa^2+1)} \cosh^2(s)$ if we interpret $\xi = u+v$ as time. The corresponding orbits are bounded periodic functions $s=s(\xi)$ and the turning points where $s'=0$ are when $s = \pm L$ as expected. This mechanical analogy shows that, at the turning points, $s'(\xi) = 0$ and $s''(\xi) = -V'(s) \neq 0$, and further, the solutions $s=s(u+v)$ are `global', i.e exist for all $(u,v) \in \mathbb{R}^2$.
Since $ds^2 = d(\kappa^{-1} \cosh(s))^2 + dz^2$, it follows from~\eqref{eq:econs} that 
\begin{equation}
\left(\frac{dz}{d\xi}\right)^2 = \left(1- \frac{\sinh^2(s)}{\kappa^2}\right)  \left(\frac{ds}{d\xi}\right)^2 = \frac{(\kappa^2 - \sinh^2(s))^2}{\kappa^2(\kappa^2+1)}
\label{eq:zprime}
\end{equation}
The right hand side vanishes quadratically in $(L^2-s^2)$, so it follows that $z'(\xi_c) =z''(\xi_c) =0$  at the turning points $\xi_c$ where $s(\xi_c) = \pm L$ and we can pick the square root so that $z'(\xi) \geq 0$ for all $\xi$.  Near a turning point, we therefore get 
\begin{align*}
\rho(\xi)  = \rho_c- c_1 (\xi-\xi_c)^2 + O((\xi-\xi_c)^3), &\quad
z(\xi)  = z_c + c_2 (\xi - \xi_c)^3 + O((\xi-\xi_c)^4),
\end{align*}
where $\rho_c = \rho(\xi_c) = \kappa^{-1}\sqrt{\kappa^2+1}, z_c = z(\xi_c)$ and  $c_1,c_2 > 0$.  
The mapping 
\begin{equation}
(u,v) \mapsto r(u,v) = \left(\rho(u+v) \cos\left(\frac{\kappa(u-v)}{\sqrt{\kappa^2+1}}\right), \rho(u+v) \sin\left(\frac{\kappa(u-v)}{\sqrt{\kappa^2+1}}\right), z(u+v)\right)
\label{eq:bobbin_psfront}
\end{equation}
 is not an immersion on the circles given by $u+v = \xi_c$  and exhibits cuspidal singularities at these points, as we illustrate in Fig.~\ref{fig:bobbin}. Nonetheless, the asymptotic curves $u \mapsto r(u,v_0)$ are smooth and satisfy $r_u \cdot r_u =1$, and likewise for the curves $v \mapsto r(u_0,v)$.
 
 Defining the normal $\displaystyle{N = \frac{r_u \times r_v}{\|r_u \times r_v\|}}$ yields, with a positive constant of proportionality, 
 $$
 N(u,v) \propto s'(u+v)\left(-\frac{dz}{ds}\cos\left(\frac{\kappa(u-v)}{\sqrt{\kappa^2+1}}\right), -\frac{dz}{ds}\sin\left(\frac{\kappa(u-v)}{\sqrt{\kappa^2+1}}\right), \frac{\sinh(s(u+v))}{\kappa}\right).
 $$
Since $\frac{dz}{ds} = 0$ at the turning points, this definition of the normal flips between $N = \pm \mathbf{e}_3$ at every turning point and is thus discontinuous. In contrast, the definition $\displaystyle{N^\omega = \sigma \frac{r_u \times r_v}{\|r_u \times r_v\|}}$, $\sigma = \mathrm{sgn}(s'(u+v))$ yields a continuous (even $C^{1M}$) definition of the normal. 
 
Using~\eqref{eq:econs}~with~\eqref{eq:zprime} and recognizing that $\frac{dz}{ds} = \sigma\left|\frac{dz}{ds}\right|$ we obtain
\begin{align}
N_z & =  \kappa^{-1} \sinh(s(u+v)), \quad \sigma = \mathrm{sgn}(s'(u+v)), \quad \theta = \kappa(\kappa^2+1)^{-1/2}(u-v)  \nonumber \\
N^\omega & = \left(-\sigma \sqrt{1-N_z^2}\cos\theta, - \sigma \sqrt{1-N_z^2}\sin \theta, N_z\right).
\label{eq:NFront}
\end{align}
$N^\omega$, in conjunction with the PS-front $r$ in~\eqref{eq:bobbin_psfront}, satisfies the Lelieuvre equations~\eqref{eq:lelieuvre}. }

\begin{definition}
An Amsler sector is a PS-front $r:[0,\infty) \times [0,\infty) \to \mathbb{R}^3$ such that the the bounding $u$- and $v$-asymptotic curves $r(\cdot,0)$ and $r(0,\cdot)$ are geodesics in $\mathbb{R}^3$. A pseudo-Amsler sector is a PS-front $r:[0,u_0) \times [0,v_0) \to \mathbb{R}^3$ such that the one of the bounding $u$- and $v$-asymptotic curves, either $r(\cdot,0)$ or $r(0,\cdot)$ is geodesic in $\mathbb{R}^3$.
\label{def:amsler-sector}
\end{definition}

 Amsler and pseudo-Amsler sectors will play a fundamental role in this work. Amsler sectors can be constructed by solving the sine-Gordon equation $\varphi_{uv} = \sin \varphi$ on the first quadrant $u \geq 0, v \geq 0$ with boundary data $\varphi(u,0) = \varphi(0,v) = \phi_0$ \cite{amsler1955surfaces}. These solutions admit a self-similar reduction of the form $\varphi(u, v) = \varphi(z),$ with $z=2\sqrt{uv}$. {This self-similar ansatz  gives
$ \partial_{uv} = \frac{1}{z}\partial_{z} + \frac{\partial^2}{\partial z^2}$ and the sine-Gordon equation reduces to
\begin{equation}
    \varphi''(z) + \frac{\varphi'(z)}{z} - \sin\varphi(z) = 0,
    \label{PIII}
\end{equation}
known as Painlev\'{e} III  in trigonometric form \cite[Chap. 2]{bobenko2000painleve}. The preimage of $z=0$ is the set $\{(u, v) : u=0 \textrm{ or } v=0\}$, and hence we see immediately that $\varphi(u, v)$ is a constant along the axes, and there is an (unbounded) open neighborhood of the axes on which the PS-front is actually an immersion since $\varphi$ is close to $\phi_0$ and away from $0$ and $\pi$. This is in stark contrast with Minding's bobbin where every $u$-asymptotic curve hits the cuspidal singular edge at a finite value of the parameter $u$ and likewise for $v$-asymptotic curves.

For an Amsler sector, along the asymptotic curves given by $u=0$, we have $\kappa^u = \partial_u \phi =0$ by~\eqref{eq:sg}, and it follows from Eq.~\eqref{eq:frames} that $\partial_u r_u = 0$ showing that this curve is geodesic in $\mathbb{R}^3$. A similar argument applies to the asymptotic curve given by $v=0$. }

\subsection{Assembling a pseudospherical surface with branch points} \label{sec:kminusonemonkeysaddle}

As a first illustration of the procedure to construct $C^{1,1}$ pseudospherical immersion we construct a monkey saddle with constant negative curvature, $K=-1$. Fix an even integer $2m\geq 4$. The number $2m$ determines the number of asymptotic rays extending from the origin and the resulting topological structure of the asymptotic coordinate system. 

\begin{definition}[$m$-star] 
Given angles $\alpha_i\in(0, \pi)$, $i\in\{1\dots 2m\}$ satisfying $\sum_i \alpha_i = 2\pi$ and lengths $l_i >0, i = 1,2,\ldots,2m$, set $\beta_0 = 0, \beta_i = \beta_{i-1} + \alpha_{i}$ for $i = 1,2,\ldots 2m$, 
and define the unit vectors $\mathbf{s}_i = \cos(\beta_i) \mathbf{e}_1 + \sin(\beta_i) \mathbf{e}_2$. Define the sectors $S_i \subset \mathbb{R}^2$ by
\begin{align}
S_{i}  & = \{c\, \mathbf{s}_{i-1} + d\, \mathbf{s}_{i} \, | 0 \leq c < l_{i-1}, 0 \leq d < l_i\}, \quad i = 1,2,\ldots,2m.
\label{def:sectors}
\end{align}
An $m$-star $T$  is a topological space with the set $T= T(\{\alpha_i\},\{l_i\}) =\bigcup_{i=1}^{2m} S_i$  constructed as above and equipped with the subspace topology given by the inclusion $T \subset \mathbb{R}^2$.
\label{def:m-star}
\end{definition}

We define coordinates $(\xi_i,\eta_i)$ so that $\mathbf{x} = \xi_i \mathbf{s}_i + \eta_i \mathbf{s}_{i+1}$ for $\eta_i \geq 0$ and $\mathbf{x} = \xi_i \mathbf{s}_i - \eta_i \mathbf{s}_{i-1}$ for $\eta_i < 0$. This gives a bi-Lipschitz mapping $(\xi_i,\eta_i) :(0 ,l_i)\times ( -l_{i-1} ,  l_{i+1}) \to (S_{i-1} \bigcup S_i)^0 \subset \mathbb{R}^2$, that is, in general, not smooth on any open set that intersects $\{\eta_i=0\}$.
\begin{remark} \label{smoothness}
 In order for all the coordinates $(\xi_i,\eta_i)$ to be smooth, we need $ \mathbf{s}_{i+1} = -  \mathbf{s}_{i-1}$ for all $i$, and this forces $m =2, \alpha_1+\alpha_2 = \pi, \alpha_1 = \alpha_3, \alpha_2 = \alpha_4$.
\end{remark}

The coordinate patches for $(\xi_i,\eta_i)$ and $(\xi_{i+1},\eta_{i+1})$ overlap on the interior of $S^i$, and the transition functions between the coordinates, given by $\eta_{i+1} = - \xi_i$ and  $\xi_{i+1} = \eta_i$, are Lipschitz (even smooth). On the sector $S_i$, we can compute the coordinate $(\xi_i,\eta_i)$ by 
$$
(\xi_i,\eta_i) =  \frac{1}{\mathbf{s}^*_{i+1} \cdot \mathbf{s}_{i}} (\mathbf{s}^*_{i+1} \cdot \mathbf{x}, - \mathbf{s}^*_{i} \cdot \mathbf{x})
$$
where the ``dual" vectors are given by $\mathbf{s}^*_j = \mathbf{e}_3 \times \mathbf{s}_j$.  Note that, $\mathbf{s}^*_{i+1} \cdot \mathbf{s}_{i} = \sin(\beta_{i+1}-\beta_i) = \sin(\alpha_i) > 0$, and these formulae extend the coordinates $\xi_i,\eta_i$ to the closure $\overline{S_i}$ as Lipschitz functions. The origin $\mathbf{x}=0$ is given by $(\xi_i,\eta_i) = (0,0)$. We define the {\em asymptotic coordinates} $(u_i,v_i)$ by 
\begin{equation}
(u_i,v_i) = \begin{cases} (\xi_i,\eta_i) & \mbox{if  $i$ is even} \\ (\eta_i,\xi_i) & \mbox{if  $i$ is odd} \end{cases}
\label{eq:sector-coords}
\end{equation}
The quantities $(u_i,v_i)$ are only defined on the sector $S_i$.  Also, for $i$ even (respectively $i$ odd),  $0 \leq u_i = u_{i+1} < l_i$ and $v_i = v_{i+1} = 0$ (resp. $u_i = u_{i+1} =0$ and $0 \leq v_i = v_{i+1} < l_i$) on $S_i \cap S_{i+1}$. We will fix the sector $S_i$ in the rest of this argument and henceforth drop the subscripts $i$ on $u_i$ and $v_i$. Given a point $\mathbf{z} \in \mathbb{R}^3$, a direction $\mathbf{n} \in S^2$ and unit vectors $\mathbf{e}_u$ and $\mathbf{e}_v$ that are pependicular to $\mathbf{n}$, we define the boundary conditions for an {\em Amsler sector} by 
\begin{align}
N(u,0) & = \cos(u) \mathbf{n} + \sin(u) \mathbf{n}\times \mathbf{e}_u \nonumber \\
N(0,v) & =  \cos(v) \mathbf{n} - \sin(v) \mathbf{n} \times \mathbf{e}_v \nonumber \\
r(u,0) & = \mathbf{z} + u \,\mathbf{e}_u, \quad r(v,0)  = \mathbf{z} + v \,\mathbf{e}_v
\label{amsler-frame}
\end{align}
 It is straightforward to verify that the definitions in~\eqref{amsler-frame} are solutions of~\eqref{eq:frames}. It follows that we can solve the Moutard equation~\eqref{eq:moutard}, a Goursat problem for for the normal $N(u,v)$ (see \cite[Thm 1.12]{bobenko2008bdiscrete} for the details),  to obtain smooth solutions in the interior of the sector $S_i$ that extend continuously to the boundary, and on the segment $u=0$ (respectively $v=0$), $N(0,v)$ (resp. $N(u,0)$) agrees with the definition in~\eqref{amsler-frame}.
 
We specialize by setting $\mathbf{z} = 0, \mathbf{n} = \mathbf{e}_3, \mathbf{e}_u = \mathbf{s}_{i},\mathbf{e}_v = \mathbf{s}_{i-1}$ if $i$ is even and $\mathbf{e}_u = \mathbf{s}_{i-1}$  and $\mathbf{e}_v = \mathbf{s}_{i}$ if $i$ is odd. Note that, for points that are in multiple sectors, i.e. points on the sector boundaries, either $u$ or $v$ is zero, $N$ and $r$ are defined consistently, i.e. they are same independent of which sector is taken in the definition. In particular, the point $u=v=0$, which belongs to all sectors, has $r_i(0,0)=\mathbf{z}=0, N_i(0,0) = \mathbf{n}=\mathbf{e}_3$ for all $i$.

In the interior of the sector $S_i$, the normal field $N_i$ which solves the Moutard equation~\eqref{eq:moutard} is weakly harmonic and thus determines a PS-front $r_i:S_i \to \mathbb{R}^3$ through the Lelieuvre equations~\eqref{eq:lelieuvre}. Since $\lim_{(u,v) \to (0,0)} N_u \times N_v = \pm \mathbf{s}^*_{i-1} \cdot \mathbf{s}_{i} \neq 0$, it follows that there exists $c_i > 0$ such that $N_u \times N_v$ does not vanish on the rectangular domain $J_i \equiv \{0 < u_i < c_i, 0 < v_i < d_i\} \subset S_i$. $r_i$ extends continuously to $\bar{J}_i$ and we have constructed a PS-front $r_i \in C^\infty(J_i) \cap C(\bar{J}_i)$ such that $r_i(0,0) = 0$ and the normal to the immersion is given by our choices for $N$ above, i.e. for points in $S_i \bigcap S_j$, $N$ is well defined since the two potential definitions of the normal, $N_i$ and $N_j$, agree. We can, after shrinking $c_i,d_i$ if needed, patch these  solutions to obtain an $m$-saddle, i.e. a  piecewise smooth PS-front $r: T \to \mathbb{R}^3$ where $T = \bigcup_i J_i$ is an $m$-star and $r(\mathbf{x}) = r_i(\mathbf{x})$ on $J_i$.

\begin{figure}[htbp]
\center
        \begin{subfigure}[b]{0.33\textwidth}
                \centering
                \includegraphics[width=.85\linewidth]{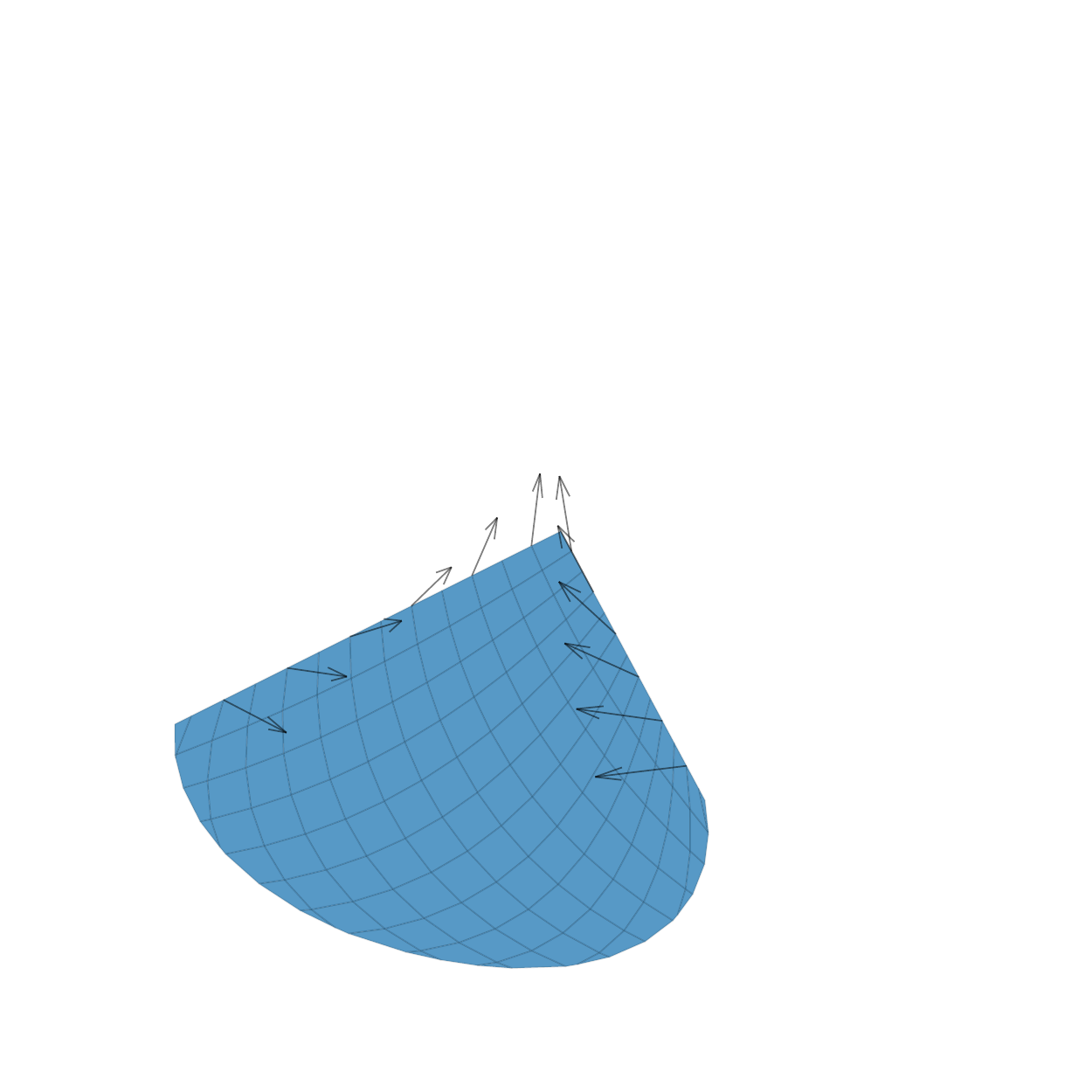}
                \caption{}
                \label{fig:monkey_saddle_construction_withnormal_stage1}
        \end{subfigure}%
        \begin{subfigure}[b]{0.33\textwidth}
                \centering
                \includegraphics[width=.85\linewidth]{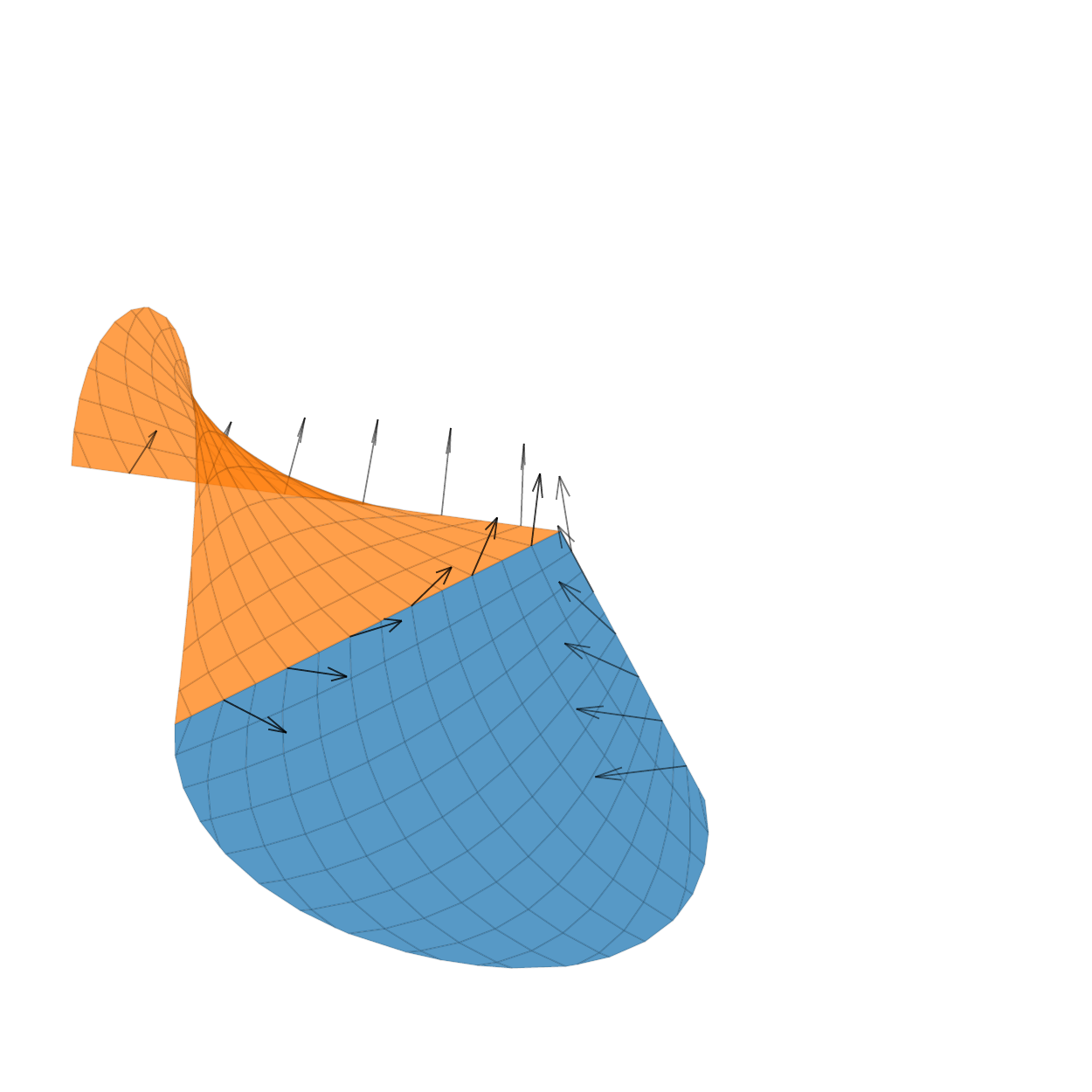}
                \caption{}
                \label{fig:monkey_saddle_construction_withnormal_stage2}
        \end{subfigure}%
        \begin{subfigure}[b]{0.33\textwidth}
                \centering
                \includegraphics[width=.85\linewidth]{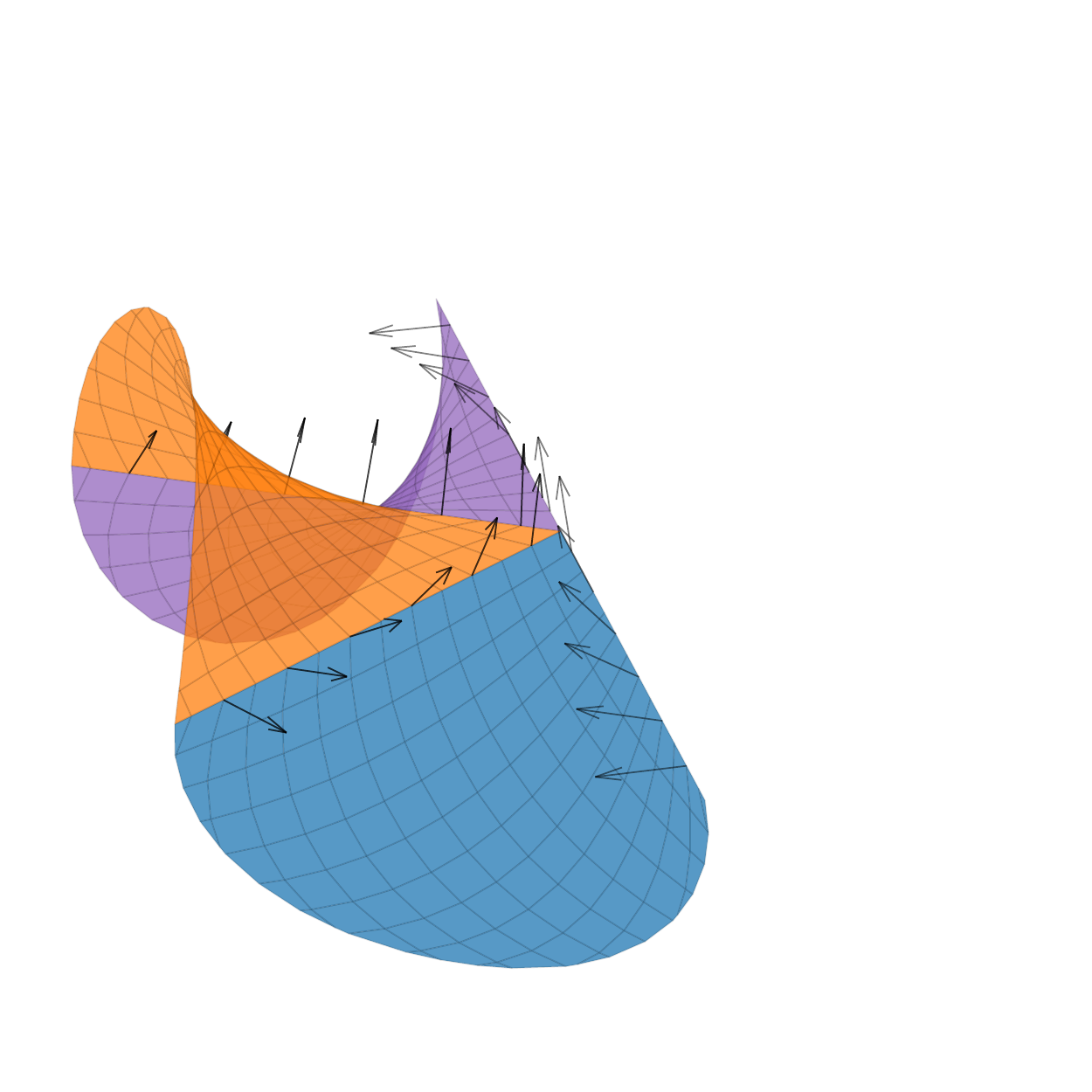}
                \caption{}
                \label{fig:monkey_saddle_construction_withnormal_stage3}
        \end{subfigure}

        \begin{subfigure}[b]{0.33\textwidth}
                \centering
                \includegraphics[width=.85\linewidth]{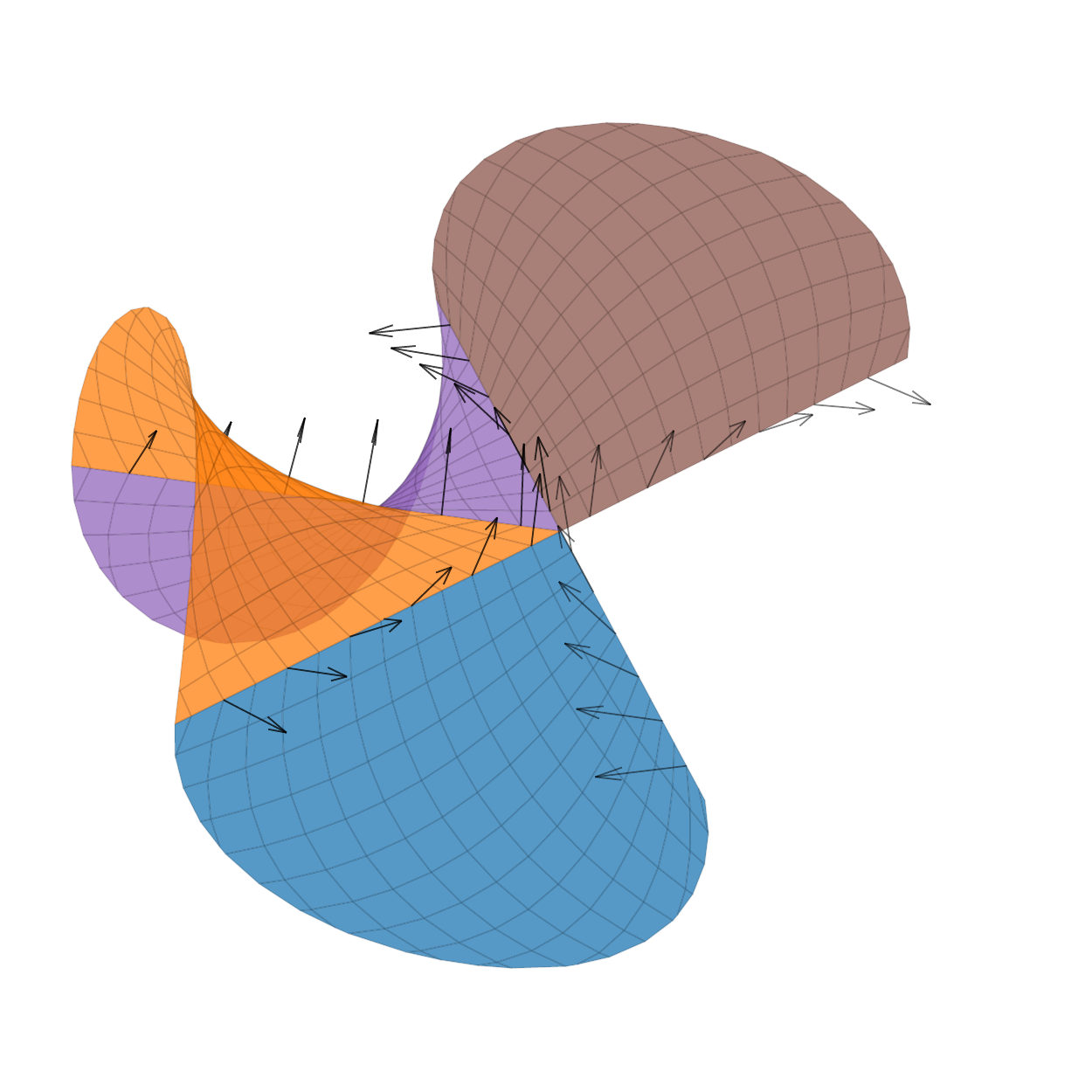}
                \caption{}
                \label{fig:monkey_saddle_construction_withnormal_stage4}
        \end{subfigure}%
        \begin{subfigure}[b]{0.33\textwidth}
                \centering
                \includegraphics[width=.85\linewidth]{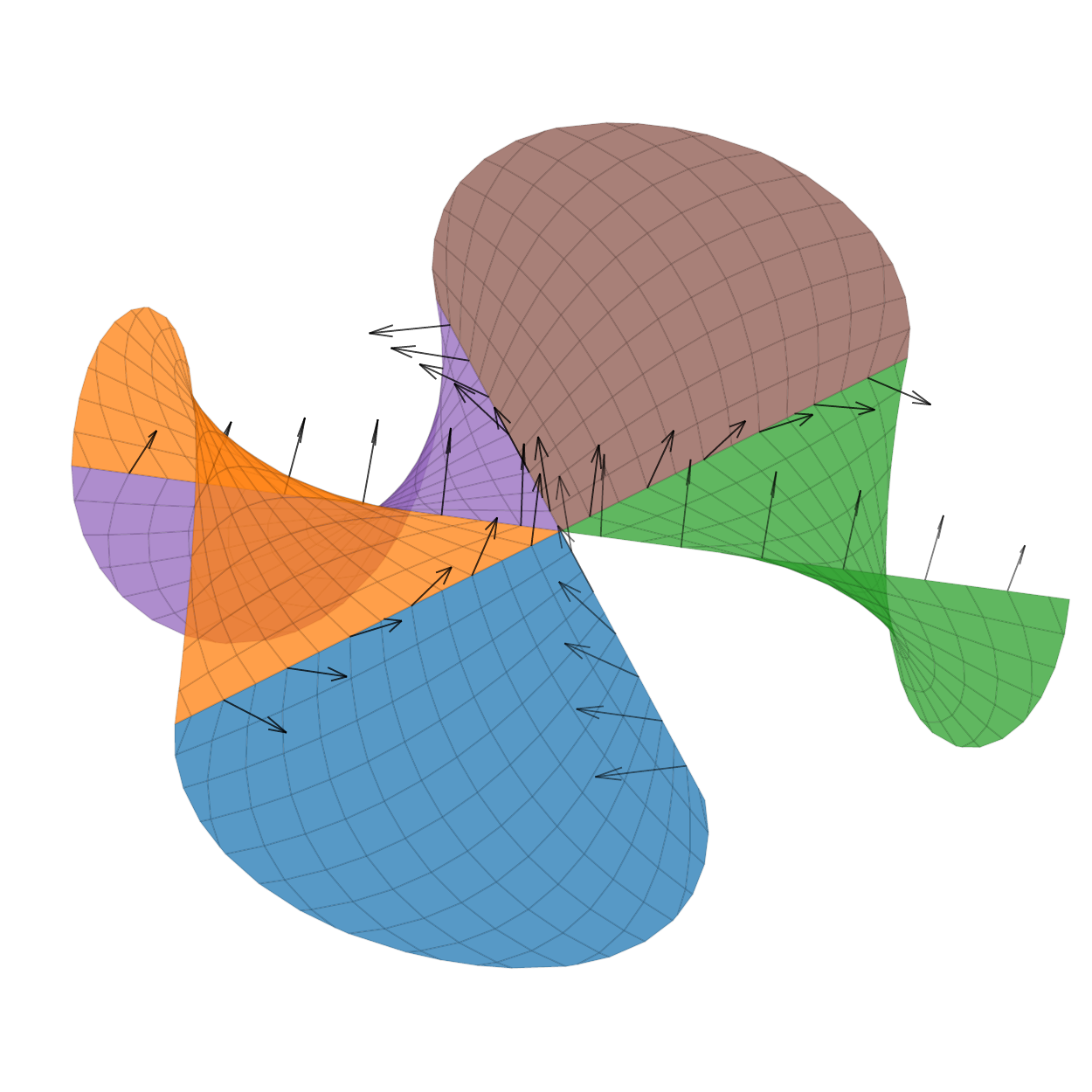}
                \caption{}
                \label{fig:monkey_saddle_construction_withnormal_stage5}
        \end{subfigure}%
        \begin{subfigure}[b]{0.33\textwidth}
                \centering
                \includegraphics[width=.85\linewidth]{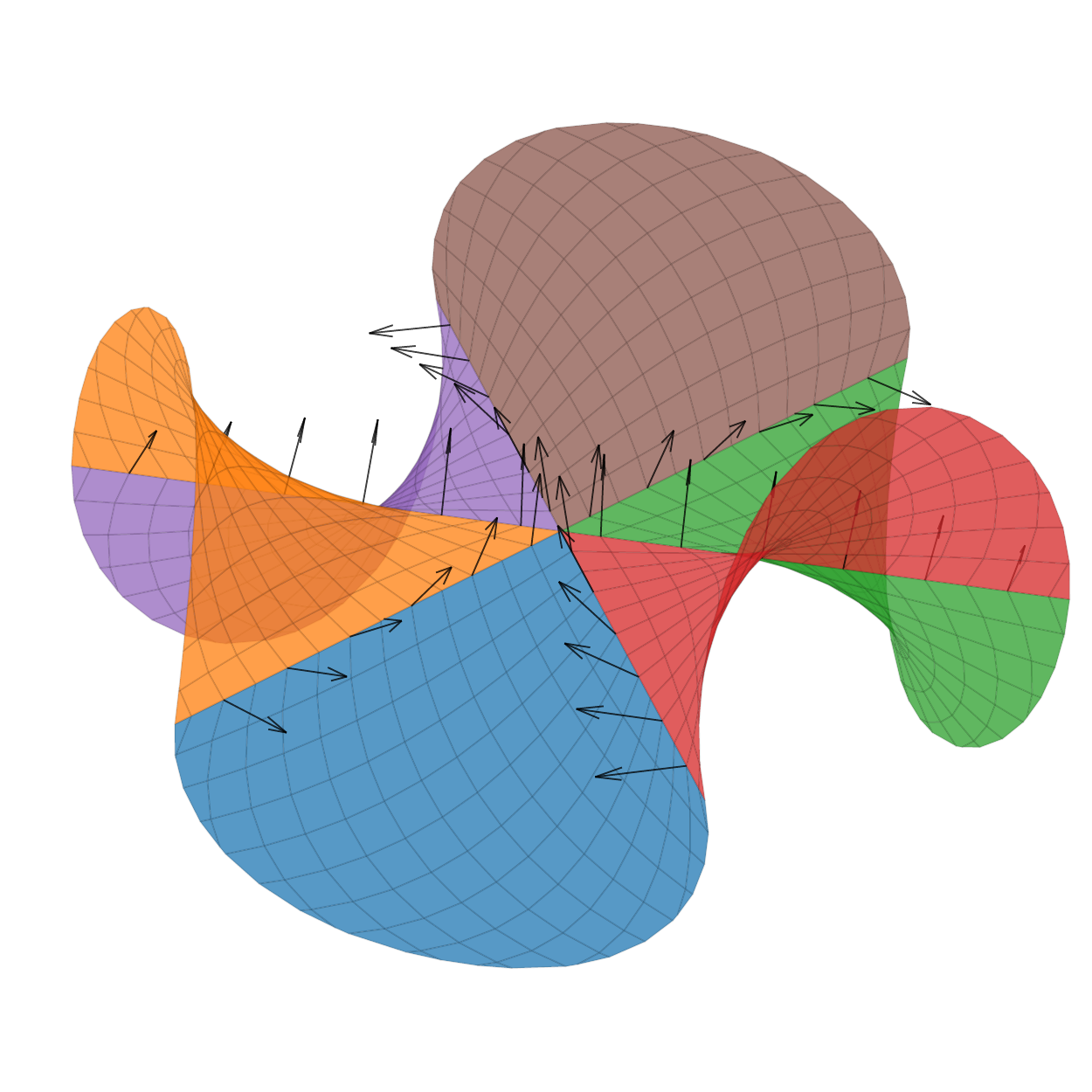}
                \caption{}
                \label{fig:monkey_saddle_construction_withnormal_stage6}
        \end{subfigure}
        \caption{Construction of a $K=-1$ 3-saddle (``monkey saddle") of geodesic radius 1. Each colored sector is smooth, and the gluing procedure maintains continuity of the normal field, shown by the arrows.}
        \label{fig:monkeysaddleconstruction}
\end{figure}

This procedure is illustrated in Fig.~\ref{fig:monkeysaddleconstruction} with $2m = 6, \alpha_k = \pi/3, k = 1,2,\ldots,6$. Since the resulting immersion is continuous and piecewise smooth, and has a continuous and piecewise smooth normal field, it follows that the normal field is (globally) Lipschitz, and the immersion is $C^{1,1}$. The immersion restricted to each sector is an example of an {\em Amsler sector} as in Definition~\ref{def:amsler-sector}, an object that will play a key role in our constructions below.

\begin{remark}
We will, for the most part, drop the subscript $i$ that indicates the domain of definition $S_i$, and refer to $u$ and $v$ simply as asymptotic coordinates. This has potential to cause confusion since $u$ and $v$ are {\em not coordinates} in the differential geometric sense and do not define a one-to-one map on any open set that intersects a boundary between sectors. This is mitigated somewhat since we usually work only of a single sector at a time, and on the intersection $S_i \bigcap S_j$ between two sectors, $u$ and $v$ have to agree. Indeed, this condition along with the requirement that $r(u,v)$ and $N(u,v)$ be well defined on the intersections of sectors $S_i \bigcap S_j$, {\em independent of whether} $(u,v)$ refer to the coordinates on $S_i$ or on $S_j$, allows up to patch  sectors together to obtain a continuous functions on the $m$-star $\bigcup_{i=1}^{2m} S_i$.
\end{remark}

We generalize the construction of patching Amsler sectors \cite{gemmer2011shape} by relaxing the requirements imposed in~\eqref{amsler-frame}.
\begin{definition}
An $m$-saddle is a $C^{1,1}$ mapping $r:T(\{\alpha_i\},\{l_i\}) \to \mathbb{R}^3$  from an $m$-star to $\mathbb{R}^3$ such that the restriction $r_i = r|_{S_i}$ is a PS-front, i.e. $r_i(u_i,v_i)$ and the corresponding normal $N_i(u_i,v_i)$ are $C^{1M}$ {\em in the coordinates} $(u_i,v_i)$, the normal is weakly regular and is Lorentz harmonic. $m$ is the {\em order of saddleness} at the point $u_i=v_i=0$
\label{def:m-saddle}
\end{definition}

We now define an algorithm for constructing $m$-saddles through {\em assembly}.

\begin{prop}[{\bf Assembly}] Let $2m \geq 4$ be an even number and let $L < \infty$. Assume that we are given  smooth functions $\kappa_i:[0,L) \to \mathbb{R}$ and angles $\alpha_i \in (0,\pi)$, for $i=1,2,\ldots,2m$, satisfying $\sum_{i=1}^{2m} \alpha_i = 2 \pi$. There exist $l_i \in (0,L), i=1,2,\ldots,2m$, sufficiently small,  $2m$ arc-length parameterized Frenet frames $F_i:[0,l_i) \to M_{3 \times 3}$ and an $m$-saddle $r:T(\{\alpha_i\},\{l_i\}) \to \mathbb{R}^3$ satisfying
\begin{enumerate}
\item $r(0,0) = 0$ and $N(0,0) = \mathbf{e}_3$,
\item For $i$ even (resp. $i$ odd) $F_i$ satisfies the first (resp. second) equation in~\eqref{eq:frames} with $\kappa^u = \kappa_i$ (resp. $\kappa^v = \kappa_i$) and the initial conditions $r_u(0) = \mathbf{s}_i$ (resp. $r_v(0) = \mathbf{s}_i$) and $N(0) =  \mathbf{e}_3$,
\end{enumerate}
where $\beta_i = \sum_{j=1}^i \alpha_j, \mathbf{s}_i = \cos(\beta_i) \mathbf{e}_1 + \sin(\beta_i)\mathbf{e}_2$ and $T$ is an $m$-star as in Definition~\ref{def:m-star}.
\label{prop:assembly}
\end{prop}

\begin{proof} The proof is by explicit construction. The existence and uniqueness for the Frenet frames follows from standard results for ODEs. The prescribed data therefore determines the normal field $N$ at the boundaries of the sectors $S_i$ where $T = \bigcup_{i=1}^{2m} S_i$, and we can solve~\eqref{eq:moutard} for $N_i(u,v)$ in the interiors of the sectors $S_i$. This normal field is weakly harmonic on each sector so we can construct the corresponding immersions using the Lelieuvre formulae. The solutions on the sectors $S_i$ can be patched on the intersections $S_i \cap S_{i+1}$ since both patches agree with the curve $t \mapsto r(t \mathbf{s}_i), 0 \leq t < l_i$ on this intersection, and the normals agree as well with the solution for the Frenet frame $F_i$. On the sector $S_i, \lim_{(u,v) \to (0,0)} \|N_u \times N_v\| = |\mathbf{s}_{i-1}^*\cdot\mathbf{s}_i| > 0$ so there is a $m$-star containing the origin, given by $\{l_i\}$ sufficiently small, such that patching the sectors gives a piecewise smooth, globally Lipschitz normal field $N$ and a $C^{1,1}$ immersion $r:T \to \mathbb{R}^3$.
\end{proof}

It follows from Definition~\ref{def:m-saddle} that the {\em order of saddleness} $m_p$ at any point $p$ is the number of times any sufficiently small deleted neighborhood of $p$ crosses from one side of (say ``below") the tangent plane  at $p$ to the other side (``above") \cite{Rozendorn1962Asymptotic}. $m_p$ thus measures the number of `undulations' at $p$. The $m_p-2$ ``excess" undulations, in comparison with a regular saddle, persist to the boundary. This mechanism allows hyperbolic surfaces to refine the buckling wavelength, isometrically, near the boundary \cite{EPL_2016}.

For the point $p$, that is common to all the sectors $S_k$ in~Prop.~\ref{prop:assembly}, the order of saddleness  $m_p = m$, corresponding to half the number of sectors at $p$. Since the asymptotic directions at $p$ are defined by the intersection between the surface and the tangent plane at $p$  ({\em cf. Dupin Indicatrix} \cite[\S 4.12]{stoker}), this relation, between the number of asymptotic directions at $p$ and $m_p$ holds more generally. This is illustrated in Figs.~\ref{fig:saddle}~and~\ref{fig:mkysaddle}. Every point in Fig.~\ref{fig:saddle} has $m=2$. In Fig.~\ref{fig:mkysaddle}, most points have $m=2$ but there is one point with $m=3$.

\begin{figure}[ht]
\centering
        \begin{subfigure}[t]{0.23\textwidth}
                \centering
                \includegraphics[trim={4.75cm 5.75cm 4cm 4.5cm}, clip, width=.95\linewidth]{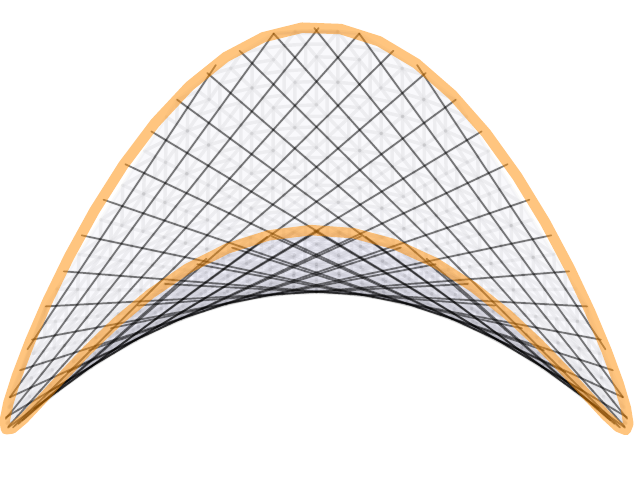}
                \caption{}
                \label{fig:saddle}
        \end{subfigure}%
        \begin{subfigure}[t]{0.23\textwidth}
                \centering
                \includegraphics[trim={0cm -1.25cm 0cm 0cm}, clip, width=.95\linewidth]{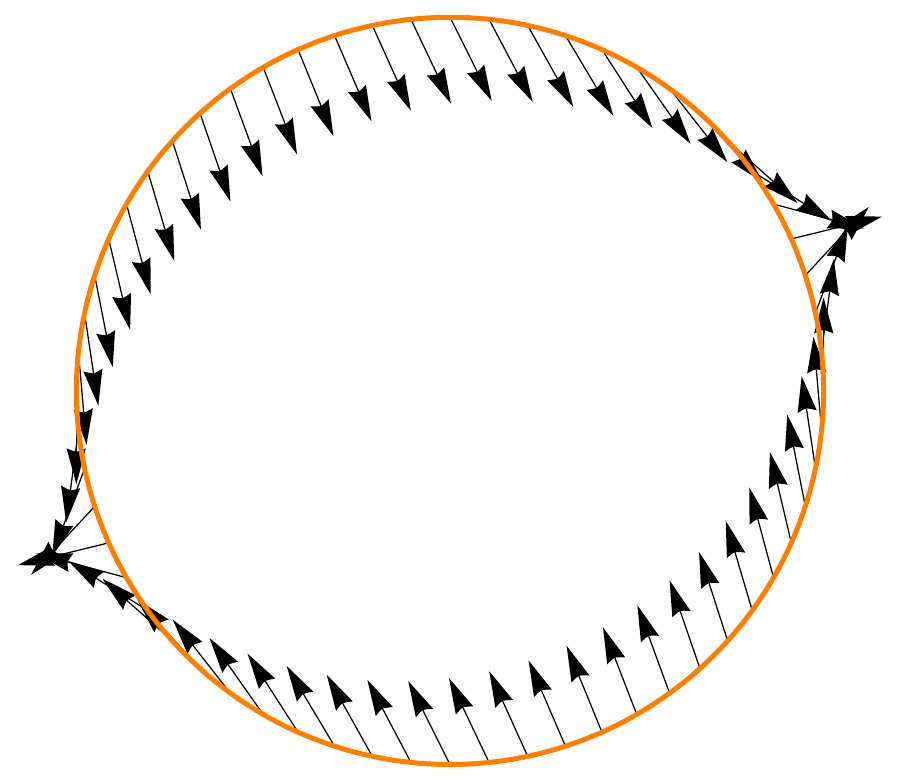}
                \caption{}
                \label{fig:nrmlsaddle}
        \end{subfigure}%
        \begin{subfigure}[t]{0.23\textwidth}
                \centering
                \includegraphics[trim={5cm 4.5cm 4.25cm 6.cm}, clip, width=.95\linewidth]{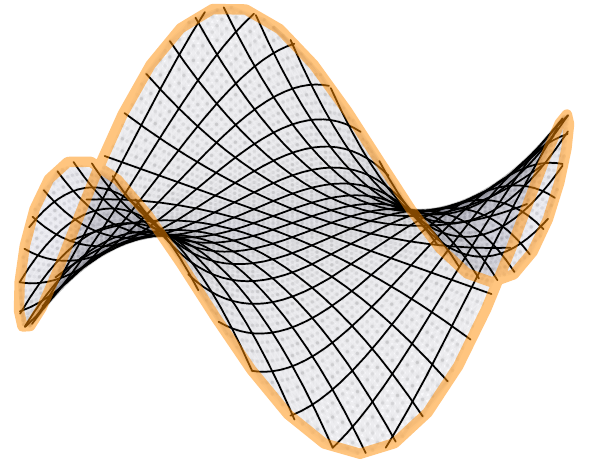}
                \caption{}
                \label{fig:mkysaddle}
        \end{subfigure}
        \begin{subfigure}[t]{0.23\textwidth}
                \centering
                \includegraphics[width=.95\linewidth]{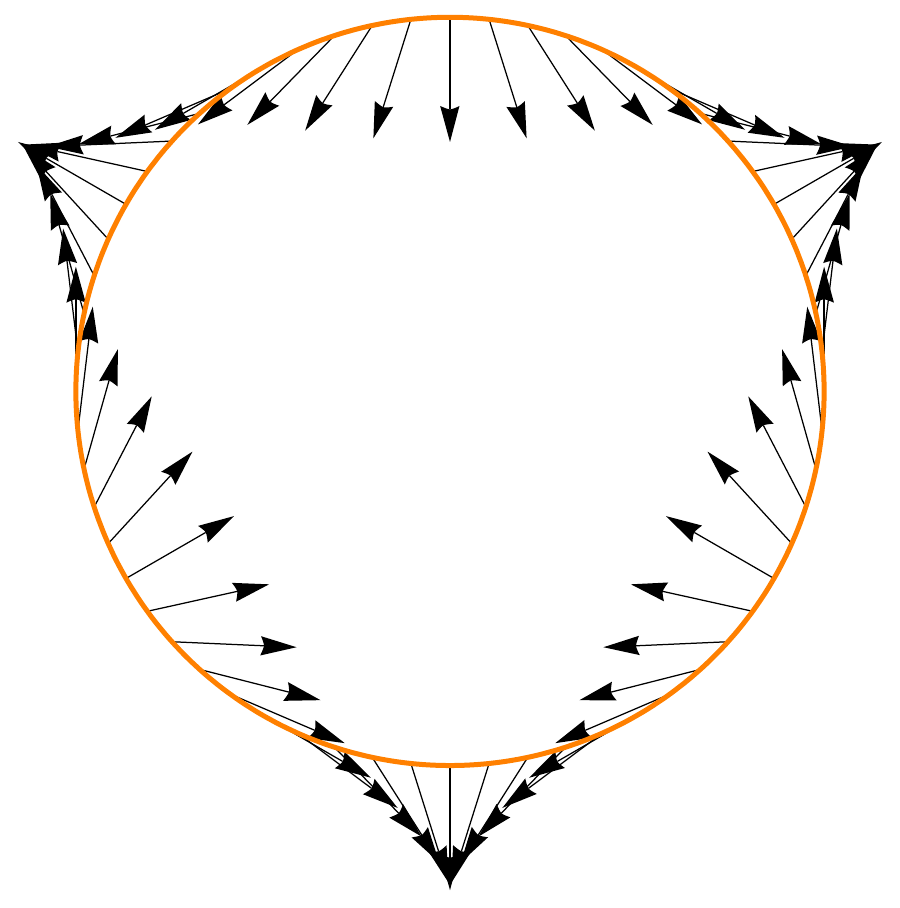}
                \caption{}
                \label{fig:nrmlmkysaddle}
        \end{subfigure}        
        \caption{The (local) winding number of the normal field about a point $p$ for two surfaces: (a) A smooth pseudospherical saddle and (c) A $C^{1,1}$ pseudospherical monkey saddle. (b),(d) Projections of the corresponding normal fields.  $p$ denotes the center of the disks. $m_p=2, J_p=-1$ for the saddle and $m_p = 3, J_p = -2$ for the monkey saddle.}\label{fig:winding_number_normalfield}
\end{figure}

\subsection{The topology of the normal map and obstructions to smoothing} \label{sec:approximation}

Our primary interest in this work is to immerse a geodesic disk $\Omega \equiv B_R \subset \mathbb{H}^2$ of radius $R$ and constant curvature $K=-1$ into $\mathbb{R}^3$ isometrically with essentially bounded principal curvatures. The local structure of this mapping near any point $p \in \Omega$ will be modeled by our construction of $m$-saddles and $m$-stars. This motivates the following definition.

\begin{definition} A {\em branched pseudospherical immersion} of a subset $\Omega$ of the Hyperbolic plane is a globally $C^{1,1}$ and piecewise $C^2$ isometric immersion $\psi: \Omega \to \mathbb{R}^3$ such that every $p \in \Omega$ has a neighborhood $O_p$, a homeomorphism $\tau_p: O_p \to T_p$, where $T_p$ is a $m_p$-star, and an associated $m_p$-saddle $r_p:T_p \to \mathbb{R}^3$  such that $\psi|_{O_p} = r_p \circ \tau_p$. 
\label{def:branched_immersion}
\end{definition}
In this work, we will consider branched pseudospherical immersions $\psi: \Omega \to \mathbb{R}^3$ where $m_p = 2$ except for finitely many points $p_1,p_2,\ldots,p_k \in \Omega$, the {\em branch points} of $\psi$. Note that our definition of branch points/immersions is local. For global considerations, we will use notions from the theory of cell complexes, and refer the reader to \cite[Chap. 0]{HatcherAlgTop} and \cite[\S 2.1]{KMM2004} for background material. We begin by stating the definition of a quadraph.
\begin{definition}[Quadgraph, {\em cf.} Def. 2, \cite{Huhnen-Venedey2014Discretization}] A quad-graph is a strongly regular polytopal cell decomposition of a surface, such that all faces are quadrilaterals (quads).
\end{definition}
A cell decomposition of a surface, given by vertices $\{V_i\}$, edges $\{E_j\}$ and faces $\{F_k\}$ is {\em strongly regular} if  (i) The edges and vertices of each face are pairwise distinct, (ii) the intersection of two faces is either empty, a single vertex, or the closure of an edge. For our purposes, the quadgraph is required to admit preferred `asymptotic coordinates'.
\begin{definition}[Asymptotic complex]An asymptotic complex $A$ is a quadgraph such that (i) each face $F_k$ is equipped with a bijection $\psi_k: F_k \to R_k$ where $R_k = [0,u_k] \times [0,v_k]$ is a rectangle, (ii) The collection of edges is partitioned into a family of $u$-edges $E^u$ and a family of $v$ edges $E^v$ such that adjacent edges on every face come from alternating families, and (iii) if a (closed) $u$ edge $E^u_j=F_k \cap F_l$, then $u_k=u_l$ and the attaching map is given by $(u,v_a) \in E^u_j \subset F_k  \mapsto (u,v_b) \in F_l$ or $(u,v_a) \in E^u_j \subset F_k  \mapsto (u_k-u,v_b) \in F_l$ where $v_a \in \{0,v_k\}, v_b \in \{0,v_l\}$. {\em Mutatis mutandis} a similar condition holds for the $v$-edges.
\label{def:A-complex}
\end{definition}
\begin{lemma} Let $A$ be an asymptotic complex. Then  A is checkerboard colorable, i.e. we can assign labels `red' and `black' to the faces in $F$ such that any pair of neighboring faces get different labels. Also, every interior vertex (a vertex not in $\partial A$) has even degree.
\end{lemma}
\begin{proof}
From out definition, there is a globally consistent assignment of the edges, i.e. elements of $X_1$, to $u$- and $v$- edges 
that alternate going around any vertex. This implies that every cycle in the dual graph, which crosses equal numbers of $u$ and $v$ edges in $X_1$ is even, and thus the dual graph is bipartite \cite[Chap. 2]{asratian1998bipartite}. In particular, the complex $A$ is checkerboard colorable, and every interior vertex has even degree, since the faces incident on an interior vertex constitute a cycle in the dual graph, the {\em link} of the vertex. 
These features are illustrated by the examples in Fig.~\ref{fig:skeleton}. The two grids are equivalent as graphs, although the grid in Fig.~\ref{fig:monkey-tcd} is naturally interpreted as the quadgraph for the surface obtained by assembly in \S\ref{sec:kminusonemonkeysaddle} while the grid in Fig.~\ref{fig:skew-tcd} is perhaps naturally interpreted as the result of surgery by excising a quadrant and replacing by 3 sectors, as in \S\ref{sec:introducingabranchpoint} below.
\end{proof}
\begin{remark}The bijection $\psi_k:F_k \to [0,u_k] \times [0,v_k]$ in Definition~\ref{def:A-complex} gives  asymptotic coordinates on the face $F_k \subseteq A$. We will henceforth assume that $A$ is simply connected and can be embedded into $\mathbb{R}^2$. The second condition actually follows from the first so  every simply connected asymptotic complex is homeomorphic to the disk \cite[Rmk. 7]{Huhnen-Venedey2014Discretization}.
\label{remark:simply_connected}
\end{remark}
\begin{figure}[ht]
\centering
        \begin{subfigure}[t]{0.45\textwidth}
                \centering
                {\includegraphics[width=0.85\linewidth]{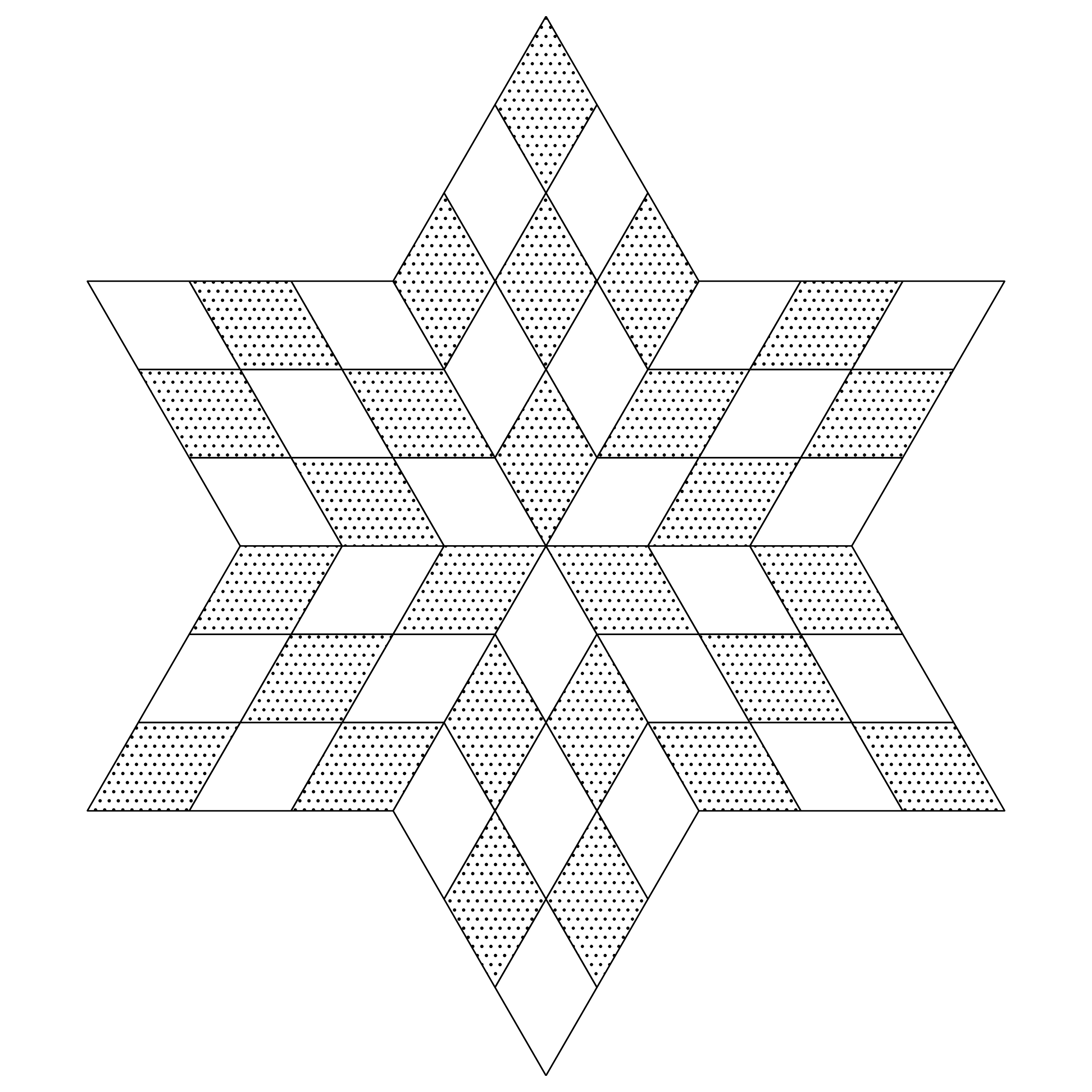}}
                \caption{}
                \label{fig:monkey-tcd}
        \end{subfigure}    
         \begin{subfigure}[t]{0.52\textwidth}
                \centering
                {\includegraphics[trim={1.5cm, 4cm, 1cm, 3cm}, clip, width=.9\linewidth]{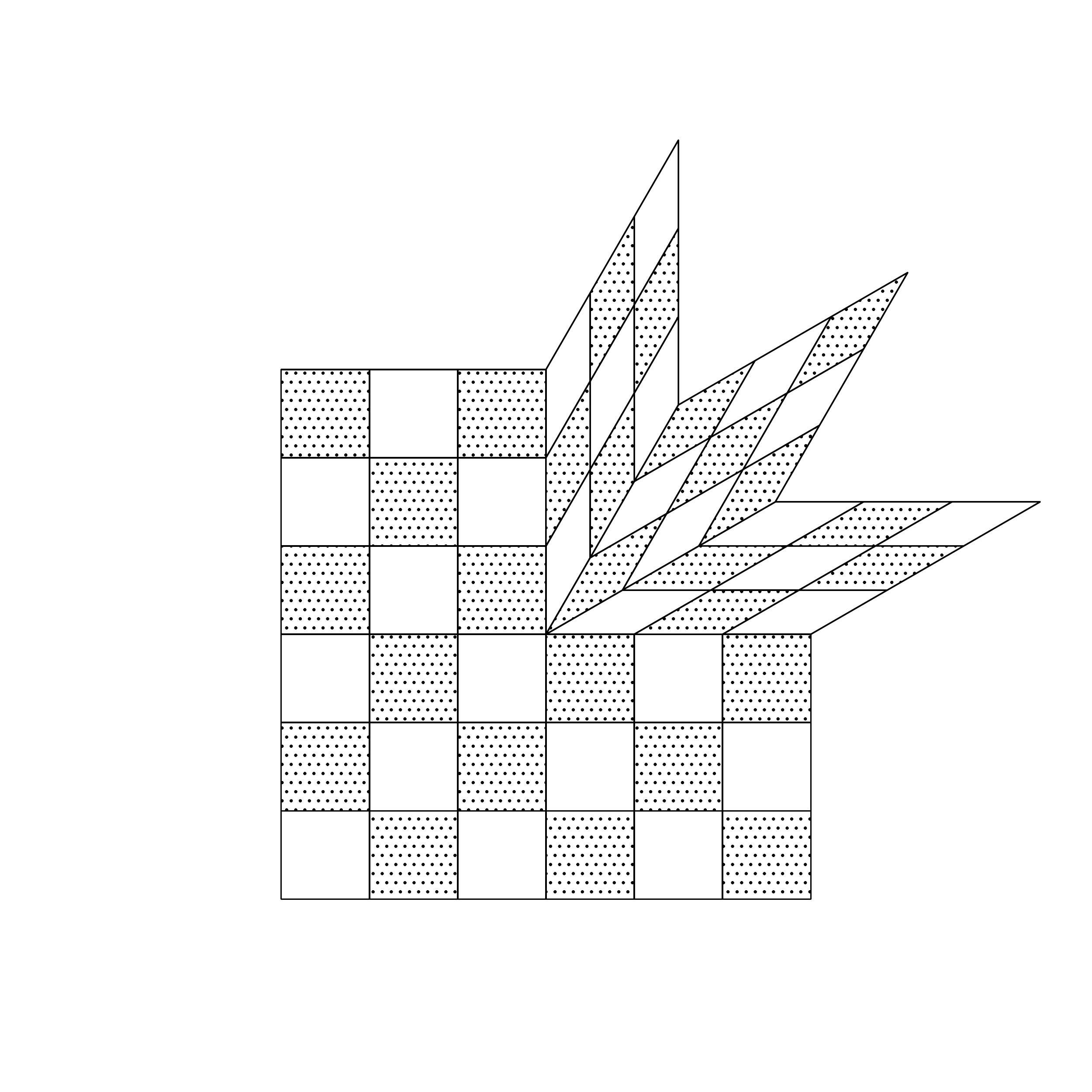}}
                \caption{}
                \label{fig:skew-tcd}
        \end{subfigure}        
        \caption{Examples of checkerboard-colorable, simply connected asymptotic complexes that are embedded in the plane. 
       }
        \label{fig:skeleton}
\end{figure}

\begin{definition}[Branched PS-front/Asymptotic quadrilateral] A {\em branched PS-front} is a mapping $r:A \to \mathbb{R}^3$ on an asymptotic complex $A$ such that the restriction $r_k = \left. r \right|_{F_k}$  is continuous on the face $F_k$ and a $C^{1M}$ PS-front on the interior $F_k^0$. An {\em asymptotic quadrilateral} is the image $r(F_k)$ of a face in a branched PS-front.
\label{defn:quad}
\end{definition}
An asymptotic quadrilateral is thus a ``rectangular" domain, bounded by 2 pairs of intersecting $u$ and $v$ asymptotic curves, on which we can define global asymptotic coordinates. 

\begin{definition}[Sector] Let $\psi: \Omega \to \mathbb{R}^3$ be a branched isometry.  A sector (at $p$) is a closed set $K \subset \Omega$, such that there is a  injection $\tau: K \to [0,u_0) \times [0,v_0)$, $\tau \in C(K) \cap C^2(K^0)$, and a PS-front $r:(0,u_0) \times (0,v_0) \times \mathbb{R}^3$ satisfying $\psi|_{K} = r \circ \tau$. Further $\tau(p) = (0,0)$ and $K$ contains the segments $\gamma_u = \tau^{-1}([0,u_0) \times\{0\})$ and  $\gamma_v = \tau^{-1}(\{0\} \times [0,v_0))$.
\label{def:sector}
\end{definition}
Informally, a sector at $p$ is a set bounded on `two sides' by a $u$- and a $v$- asymptotic curve through $p$, and contains no other asymptotic curves through $p$. Let $p \in \Omega \subset \mathbb{H}$ and let $\psi:\Omega \to \mathbb{R}^3$ be a branched isometry. 
The sum of the $m_p$ angles of the sectors at $p$ (in the surface) is  $2 \pi$.  The images of these sectors under the Gauss normal map $N$, however, can wind around the normal $N(p)$ {\em multiple times}, as depicted in Fig.~\ref{fig:nrmlmkysaddle}. This motivates 

\begin{definition} Let $V \subseteq \Omega$ be an open set, $p \in V$ and $U = V\setminus \{p\}$ denote a deleted neighborhood of $p$. Let $N: \Omega  \to S^2$ be a continuous map with the property that  $N(U) \subseteq S^2 \setminus\{\pm N(p)\}$, where $-N(p)$ is the antipodal point to $N(p)$. The {\em ramification index of the normal map} at $p$, denoted by $J_p$, is defined as the degree of the (composite) map
$$
S^1  \stackrel{\gamma}{\longrightarrow} U  \stackrel{N}{\longrightarrow} S^2 \setminus\{\pm N(p)\} \stackrel{\pi}\longrightarrow S^1,$$
where $\gamma$ is a simple closed curve in $U$, $x_{\perp} = x - \langle N(p),x\rangle N(p)$ and $\pi(x) = x_{\perp}/\|x_{\perp}\|$ is the canonical retraction $\pi: S^2 \setminus\{\pm N(p)\} \to S^1$ (``retracting to the equator").
\label{def:Jp}
\end{definition}  

For surfaces with negative extrinsic curvature, $J_p < 0$ everywhere since the normal winds clockwise for a counterclockwise circuit around $p$. 
If $J_p = -1$, the normal map is a local homeomorphism. However,  if $J_p < -1$, then $N(V)$ is a branched (``multiple-sheeted") covering of a neighborhood of $N(p)$, which is therefore a branch point for the inverse of the Gauss normal map. This justifies calling $p$ a branch point if $|J_p| > 1$, and is in keeping with standard usage \cite{kirchheim2001Rigidity,Gemmer2013Shape,EPL_2016}. 

Every immersion can be (locally) expressed as a graph $(x_1,x_2,w(x_1,x_2))$ where $(x_1,x_2)$ are coordinates in the tangent plane at $p$, and $w(x_1,x_2)$ is the normal displacement from this plane. In these coordinates, $\pi \circ N = \nabla w/\|\nabla w\|$, so we can compute the ramification index $J_p$ as the degree of the map $S^1 \to S^1$ given by 
$$
\{(x_1,x_2) \, | \,x_1^2 + x_2^2 = 1\}  \mapsto  \frac{\nabla w(\epsilon x_1, \epsilon x_2)}{\|\nabla w(\epsilon x_1, \epsilon x_2)\|},
$$ 
for any sufficiently small $\epsilon$. This computation of $J_p$ is illustrated in Fig.~\ref{fig:winding_number_normalfield}.

The winding number $J_p$ and the order of saddleness $m_p$ are related as follows 

\begin{lemma}
Let $y:\Omega \to \mathbb{R}^3$ be a $C^{1,1}$ pseudospherical immersion, and let $p$ be a point in $\Omega$. Then $J_p = 1-m_p$ where $J_p$ is the local degree of the Gauss normal map $N:\Omega \to S^2$ at $p$, and $m_p$ is the order of saddleness of the immersion $y$ at $p$.
\label{lem:covering-number}
\end{lemma}
\begin{proof}
We remark that the quantities $m_p$ and $J_p$ are well defined for $C^{1,1}$ immersion (and even for immersions with lower regularity), since $N: \Omega \to S^2$ is continuous 
(even Lipschitz) \cite{Hartman1959spherical}. The equality $J_p = (1-m_p)$ is known from the theory of weakly regular saddle surfaces (see \cite[Lemma  1.2]{Rozendorn1966Weakly}). We will have further use for the intuition behind this result so we give a short, self contained argument that holds for branched $C^{1,1}$ surfaces. Our argument is based on the Lelieuvre equations~\eqref{eq:lelieuvre}.

By invariance under Euclidean motions, we can, WLOG, assume that $y(p) = 0, N(p) = \mathbf{e}_3$. A saddle of order $m$ is defined by angles $0 = \beta_0 < \beta_1 < \cdots < \beta_{2m} = 2 \pi$ such that the tangent vectors, at $p$, to the $u$ and $v$ asymptotic curves are given by $\mathbf{s}_i = \cos(\beta_i) \mathbf{e}_1 + \sin(\beta_i) \mathbf{e}_2$ (cf. Eq.~\eqref{def:sectors}).   

From~\eqref{eq:lelieuvre}, we have,  $N_u = N \times r_u, N_v = - N \times r_v$, so the asymptotic curves lift to $S^2$ by the normal map $N$ into curves whose tangents at $N(p) = \mathbf{e}_3$ are given by $\mathbf{t}_i = \cos(\theta_i) \mathbf{e}_1 + \sin(\theta_i) \mathbf{e}_2$ where $\theta_i = \beta_i + \frac{\pi}{2} \bmod 2 \pi $ if $i$ is even and $\theta_i = \beta_i -  \frac{\pi}{2}  \bmod 2 \pi$ if $i$ is odd. We can determine the  values of $\theta_i$ by imposing the requirement $0 < \theta_{i}-\theta_{i+1} < \pi$, which is necessary to ensure that $N_u \times N_v = - r_u \times r_v$. Since $0 < \beta_{i+1}-\beta_i < \pi$, it follows that $\theta_{i+2} - \theta_i = \beta_{i+2} - \beta_i - 2 \pi$. Adding up the differences in the $\theta_i$, we thus get 
$$
\sum_{i=1}^{2m} [\theta_{i} - \theta_{i-1}] = \sum_{k = 1}^m [\theta_{2k} - \theta_{2k-2}] = \beta_{2m} - \beta_0 - 2 m \pi = 2(1- m) \pi,
$$ 
thus proving the claim that $J_p = 1 - m_p$.
\end{proof}

Fig.~\ref{fig:winding_number_normalfield} is an illustration of this result. It seems natural that there is no ``nice" way to approach the monkey saddle (Fig.~\ref{fig:mkysaddle}) through normal saddle surfaces (Fig.~\ref{fig:saddle}), since we cannot go from a winding number of 1 to a winding number of 2 continuously. This is indeed the case as we show in Theorem~\ref{thm:branched} below. This theorem encapsulates the principal motivation for an investigation of pseudospherical surfaces with branch points, namely that surfaces with branch points are distinct from smooth surfaces psuedospherical surfaces because they carry a topological index that cannot be smoothed away.  
Our approach is based on the ideas of Brezis and Nirenberg for the degree of BMO mappings \cite{Brezis1995Degree,Brezis1996Degree} with quantitative estimates from the theory of quasi-isometric mappings \cite{John1968quasi-isometricI,John1969quasi-isometricII}.  

Definition~\ref{def:Jp} for $J_p$ is through computing the index on a circle with sufficiently small radius $\epsilon$. We now show that the radius $\epsilon$ is only limited by the max curvature so that, for any minimizing sequence for $\mathcal{E}_\infty$ consisting of $C^2$ immersions, we have uniform control on the size of the circles that we may use to compute the ``local" degree of the normal map.

\begin{lemma} \label{lem:maximal} For all $k_{\max} < \infty$ there exist $\eta = \eta(k_{\max}) >0$ such that for all $0 < \delta < \eta$ and for all $C^2$ pseudospherical immersion $y : B_{3 \delta} \to \mathbb{R}^3$ with $\max(|\kappa_1(x)|,|\kappa_2(x)|) \leq k_{\max}$ for all $x \in \overline{B_{2\delta}}$, we have
\begin{enumerate}
    \item The normal map $N:B_{2 \delta} \to S^2$ is one-to-one.
   \item For all $x$ in the `collar' $B_{2 \delta} \setminus \overline{B_\delta}$, we have $\|N(x) - N_0\| \geq c(k_{\max}) \delta$, where $N_0$ is the image of the center of the geodesic ball $B_{2\delta}$.
\end{enumerate} 
\end{lemma}

\begin{proof}
For a $C^2$ immersion $y:B_{3\delta} \to \mathbb{R}^3$, there are global asymptotic coordinates $u,v$ on $B_{2\delta}$ and an angle field $\varphi: B_{2\delta} \to (0,\pi)$  such that the metric is given by $g=du^2 +2 \sigma \cos(\varphi) du dv + dv^2$ \cite{hartman1951asymptotic} (See also~\eqref{eq:metrics}), and the pull back of the metric on the sphere by the Normal map gives $G = du^2 +dv^2 - 2 \sigma \cos(\varphi) du dv$. The larger principal curvature is given by $\max(\tan \frac{\varphi}{2}, \cot \frac{\varphi}{2})$ so the hypothesis gives the restriction $2 \tan^{-1} \frac{1}{k_{\max}} \leq \varphi \leq \pi-2 \tan^{-1} \frac{1}{k_{\max}}$. We note here that $\kappa_1(x) \kappa_2(x) = -1$ so, necessarily, $k_{\max} \geq 1$.

{For any tangent vector $\mathbf{w} \in \mathrm{span}(\frac{\partial}{\partial u},\frac{\partial}{\partial v})$, we have 
\begin{equation}
\label{eq:BLD}
    k_{\max}^{-2} \leq \frac{1-|\cos \varphi|}{1+|\cos \varphi|} \leq \sqrt{\frac{G(\mathbf{w},\mathbf{w})}{g(\mathbf{w},\mathbf{w})}} \leq \frac{1+|\cos \varphi|}{1-|\cos \varphi|} \leq k_{\max}^{2}
\end{equation}

Setting $\eta = \frac{\pi}{6} k_{\max}^{-1}$, it is straightforward to see that the length of a spherical arc between $N_0$ and $N(x)$ is less than $\frac{\pi}{2}$. Consequently,  $N(x) \cdot N_0 >0$ for all $x \in B_{3\delta}$ and the image of $B_{3\delta}$ under the normal map is contained within a hemisphere.

We can identify $B_{3\delta}$ with a subset of the unit disk through the Poincar\'{e} disk embedding \cite[Chap. 4]{anderson2005hyperbolic} (see also \S \ref{sec:poincare-ddg}). Pre- and post- composing the Normal map $N$ with complex conjugation and projection $\perp : S^2 \to \mathbb{R}^2$ into the orthogonal complement of $N_0$, we obtain the map $\overline{N}^\perp:B_{3\delta} \to \mathbb{R}^2$ given by $(x+iy) \mapsto  N(x-iy) - \langle N(x-iy), N_0\rangle N_0$. 

We collect a few properties of the map $\overline{N}^\perp$:
\begin{enumerate}
    \item The image of this map is contained in the unit disk.
    \item This map is $C^1$ since $N$ is $C^1$ and the other maps are smooth.
    \item  $\overline{N}^\perp$ preserves orientation since complex conjugation and $N$ are both orientation reversing, while $\perp$ preserves orientation.
    \item It follows from  the smoothness of complex conjugation, the smoothness of the Poincar\'e disk identification of the unit disk $x^2+y^2 < 1$ with the Hyperbolic plane, the smoothness of the orthogonal projection from the (open) hemisphere to the unit disk,  the compactness of $\overline{B_{2\delta}}$, and from~\eqref{eq:BLD} that an analogous relation is true for the mapping $\overline{N}^\perp$, i.e. the (local) distortion of lengths by the mapping $\overline{N}^\perp $ is bounded away from $0$ and $\infty$ on the ball $B_{2\delta}$. The constants giving these bounds only depend on $\eta$ and the constants in~\eqref{eq:BLD}, so they only depend on $k_{\max}$.
\end{enumerate} 

It follows that $\overline{N}^\perp:B_{2\delta} \to \mathbb{R}^2$ is a regular quasi-isometry \cite{John1968quasi-isometricI} (i.e. a Bounded length distortion (BLD) local homeomorphism \cite[\S 4]{Martio1988Elliptic}). Our conclusions are a direct restatement of the results of F. John, \cite[Thm. III]{John1968quasi-isometricI} (see also \cite[Lemma 4.3]{Martio1988Elliptic}).}
\end{proof}

In the preceding proof, we used the following result, first proved in \cite[Thm. III]{John1968quasi-isometricI}. We present an equivalent statement using the notation in \cite{Martio1988Elliptic}. 

\begin{definition*}[BLD mapping, Def. 2.1, Martio and V\"ais\"al\"a \cite{Martio1988Elliptic}] 
Let $L > 1$. A Lipschitz mapping $f : G \subseteq \mathbb{R}^n \to \mathbb{R}^n$ is said to be of $L$-bounded length distortion, abbreviated L-BLD, if, for a.e. $x \in G$, (i)  $|h|/L \leq |f'(x)h| \leq  L |h|$ for all $h \in  \mathbb{R}^n$, and (ii) $\mathrm{det}(Df(x)) >0$.
\end{definition*}

\begin{theorem*}[Thm. III, John \cite{John1968quasi-isometricI}] If $f : G \subseteq \mathbb{R}^n \to \mathbb{R}^n$ is an L-BLD immersion and if $B_r(x) \subseteq G$, then $\|f(w)-f(z)\|/L \leq \|w-z\| \leq L \|f(w)-f(z)\|$ for all $w,z \in B_{r/L^2}(x)$. 
\end{theorem*}

The following lemma weakens the hypotheses in the previous lemma, by (i) allowing for branched i.e. globally $C^{1,1}$ and piecewise $C^2$ immersions, and (ii) removing the uniform bound $k_{\max}$ for the max curvature. The conclusions are also correspondingly weaker.

\begin{lemma} Let $\Omega \subset \mathbb{H}^2$ denote a (proper) open subset of the Hyperbolic plane and let $y : \Omega \to \mathbb{R}^3$ be a branched pseudospherical immersion. For every point $p \in \Omega$, there exist $\delta > 0$ and $d_0 > 0$ such that:
\begin{enumerate}
    \item The normal map $N:B_{2 \delta}(p) \to S^2$ satisfies $N(x) \neq N(p)$ for any $x$ in the punctured ball $B_{2 \delta}(p) \setminus \{p\}$.
    \item For all $x$ in the `collar' $B_{2\delta}(p) \setminus \overline{B_{\delta}(p)}$, we have $\|N(x) - N(p)\| \geq d_0$.
\end{enumerate} 
\label{lem:injective}
\end{lemma} 
\begin{proof} 

If $y$ is a $C^2$ immersion, the normal map $N:\Omega \to S^2$ is an immersion at $p$ and thus injective in a neighborhood of $p$, implying the existence of an appropriate $\delta>0$ such that for all $x \in B_{3\delta}(p)\setminus\{p\}$ we have $N(x) \neq N(p)$. Since $\overline{B_{2 \delta}(p)} \setminus B_{\delta}(p)$ is a compact subset of $ B_{3\delta}(p)\setminus\{p\}$ and $N$ is continuous, the conclusions follow.

If $y$ is a branched immersion, the normal map is not injective on any neighborhood of a branch point $p$ since $p$ is a ramification point for the  Gauss Normal map $N: \Omega \to S^2$.  However, if $S_i \subset \Omega$ is one sector at the branch point $p$, we can extend the asymptotic curves bounding $S_i$ smoothly so that the extensions satisfy Eq.~\eqref{eq:frames}. As in Prop.~\ref{prop:assembly}, we can now construct a $C^2$ immersion $\tilde{y}_i$ on a neighborhood of $p$, one that agrees with $y$ on the sector $S_i$. Thus there is a $\delta_i > 0$ such that $N(x) \neq N(p)$ on $S_i \bigcap B_{3\delta_i}(p)$. Setting $\delta = \min(\delta_0,\delta_1,\ldots,\delta_{2m_p-1})$ gives a $\delta >0$ with the required property. 
\end{proof}

\begin{theorem} \label{thm:branched} Let $\Omega$ denote an open, simply connected, domain in the Hyperbolic plane and $y: \Omega \to \mathbb{R}^3$ be a $C^{1,1}$ immersion, possibly with branch points. Assume that there exists a sequence of $C^2$ pseudospherical immersions  $y_n: \Omega \to \mathbb{R}^3$ such that 
\begin{enumerate}
    \item $y_n \to y$ in $W^{2,2}_{\text{loc}}$.
    \item $\mathcal{E}_\infty[y_n] \leq k_{\max}$ for all $n$.
\end{enumerate}    
Then $m_p[y] = 2$ for every point in $\Omega$.
\end{theorem}

\begin{proof}
$p \in \Omega$ is an arbitrary point. In what follows, let $\epsilon>0$ be sufficiently small so that $B_{3\epsilon}(p) \subseteq \Omega$ and  $\epsilon < \min(\eta(k_{\max}),\delta(p))$ for $\eta(k_{\max})$ as given by Lemma~\ref{lem:maximal}, and $\delta(p)$, as given by Lemma~\ref{lem:injective}. Also, there is a corresponding $\rho_0(p) = \min(c(k_{\max}) \epsilon, d_0(p))>0$, such that $x \in \overline{B_{2\epsilon}(p)} \setminus B_\epsilon(p)$ implies that $\|N(x)-N(p)\| \geq \rho_0(p)$ and  $\|N_n(x)-N_n(p)\| \geq \rho_0(p)$ for all $n$, where $N$ and $N_n$ are the normal maps for the immersions $y$ and $y_n$ respectively.

$\epsilon >0$ now gives uniform control on the size of the geodesic ball $B_{\epsilon}(p)$ whose boundary can be used to compute the local winding number $J_n(p)$ and the limiting winding number $J_p$, as in Definition~\ref{def:Jp}, at (a potential branch point) $p$.

For the  $C^2$ immersions $y_n$, $N_n$ is locally one to one \cite{Hartman1959spherical} and $J^{(n)}_p$, the local degree of the normal map $N_n$ at $N_n(p)$ is $-1$ (from the reversal of orientation). $W^{2,2}_{\text{loc}}$ convergence  $y_n \to y$ implies $W^{1,2}$ convergence of the Normal maps on compact sets (here $\overline{B_{2 \epsilon}(p)})$. Convergence of the normals in $W^{1,2}(\overline{B_{2 \epsilon}(p)})$ implies convergence in BMO \cite[\S 5.8.1]{evans2} as well as in $L^1(\overline{B_{2 \epsilon}(p)})$. Our maps $N_n$ thus satisfy the hypotheses required for the stability of degree under BMO convergence \cite[Property 2, \S II.2]{Brezis1996Degree}. This implies $J_p = -1$ for the immersion $y$. Lemma~\ref{lem:covering-number} now implies that $m_p = 2$.
\end{proof}

According to Thm.~\ref{thm:branched}, the monkey saddle in Fig.~\ref{fig:mkysaddle}, which has a point $p$ with $J_p = -2$, {\em cannot} be approximated, in  $W^{2,2}_{\text{loc}}$,  by sequences of $C^2$ pseudospherical immersions with uniformly bounded principal curvatures $\mathcal{E}_\infty(y_n) \leq k_{\max} < \infty$. This is a local statement, so the relevant issue is not that the principal curvatures are getting large away from the branch point $p$. Indeed, since $W^{2,\infty}_{\mathrm{loc}}$ convergence implies $W^{2,2}_{\mathrm{loc}}$ convergence, it follows that any sequence of smooth pseudospherical immersions that converges to the monkey saddle in $W^{2,2}_{\mathrm{loc}}$ necessarily has blowup of the principal curvatures on arbitrarily small neighborhoods of the branch point $p$ and therefore does not converge in $W^{2,\infty}_{\mathrm{loc}}$. In physical terms, the index $m_p$ (or equivalently $J_p$) makes branch points {\em topological defects}, and they cannot be ``smoothed out" while keeping the principal curvatures uniformly bounded.

 Theorem~\ref{thm:branched} allows/suggests  the possibility that the infimum of max curvature $\mathcal{E}_\infty$  for $C^{1,1}$ branched isometries can be strictly smaller than the infimum over $C^2$ or smoother isometries, since we cannot approximate isometries with a nonempty set of branch points $\{p_i \, | \, J(p_i) \geq 2\}$, in $W^{2,2}_{\mathrm{loc}}$, or {\em a fortiori} in $W^{2,\infty}_{\mathrm{loc}}$, by smooth isometries with locally uniformly bounded curvatures. Such an energy gap between these two regularity classes is certainly unexpected, since branched isometries can indeed be approximated by smooth mappings. Also, this behavior is in striking contrast to the case of flat \cite{Pakzad2004Sobolev,Hornung2011Approximation} and elliptic surfaces \cite{Hornung2018Regularity}, where $W^{2,2}$ isometries (respectively $C^{1,1}$ isometries) can be approximated in $W^{2,2}_{\mathrm{loc}}$ (resp. $W^{2,\infty}_{\mathrm{loc}}$) by smooth isometries.
 
 We present numerical evidence to support the existence of an energy gap for surfaces with constant negative curvature (see Fig.~\ref{fig:energywithradius}), and argue that rather than being merely a mathematical curiosity, this energy gap is key to explaining the observed ubiquity of undulating morphologies for hyperbolic sheets in nature, despite the existence of smoother isometries \cite{EPL_2016}. The existence of an energy gap for the max curvature and Willmore functionals restricted to isometries is an example of the {\em Lavrentiev} phenomenon \cite{Lavrentieff1927Quelques,Ball1985One}, \cite[\S 18.5]{Cesari1983Optimization}, and this has important consequences for numerically minimization of energy functionals \cite{Ball1987Numerical}. We discuss these issues further in \S \ref{sec:discussion}.

\begin{remark}
\label{rem:approximable}
Thm.~\ref{thm:branched} does not imply that $y$, a $W^{2,2}_{\text{loc}}$ limit of $C^2$ pseudospherical immersions is necessarily $C^2$, although the local degree of $y$ is $-1$ everywhere. Indeed, the construction from Eq.~\ref{def:sectors} with $m=2$ (4 sectors) but with $\alpha_1 + \alpha_2 \neq \pi$ and $\alpha_2 + \alpha_3 \neq \pi$ gives a piecewise smooth, non-$C^2$ surface in any neighborhood of $p$ since the $u$ and the $v$ asymptotic curves through $p$  (respectively $\gamma_u$ and $\gamma_v$) are not differentiable at $p$. However, $\gamma_u$ and $\gamma_v$ can  be uniformly approximated by smooth solutions of Eq.~\eqref{eq:frames} obtained by smoothing the (distributional) geodesic curvature(s) $\kappa^u$ (respectively $\kappa^v$) of $\gamma_u$ (resp. $\gamma_v$) giving a pair of intersecting ``initial curves". Solving the Lelieuvre equations yields smooth pseudospherical surfaces that converge to a $C^{1,1}$ immersion with $J =-1$ everywhere. This argument also gives approximations by smooth isometries for the $C^{1M}$ pseudospherical surfaces investigated by Dorfmeister and Sterling \cite{dorfmeister2016pseudospherical}, which have $J=-1$ everywhere, in contrast to the branched pseudospherical surfaces considered in this work.
\end{remark}

\subsection{Introducing a new branch point: Surgery}
\label{sec:introducingabranchpoint}
Here we outline another specific example of a branched surface, illustrating an approach that we call {\em surgery}, in contrast to the approach of {\em assembly} in the earlier section. 
In the process of surgery, we introduce a branch point into a ``pre-existing" PS-front. 
\begin{figure}[ht]
\center
        \begin{subfigure}[b]{0.33\textwidth}
                \centering
                {\includegraphics[trim={3cm, 0.5cm, 3.5cm, 0.5cm}, clip, width=.85\linewidth]{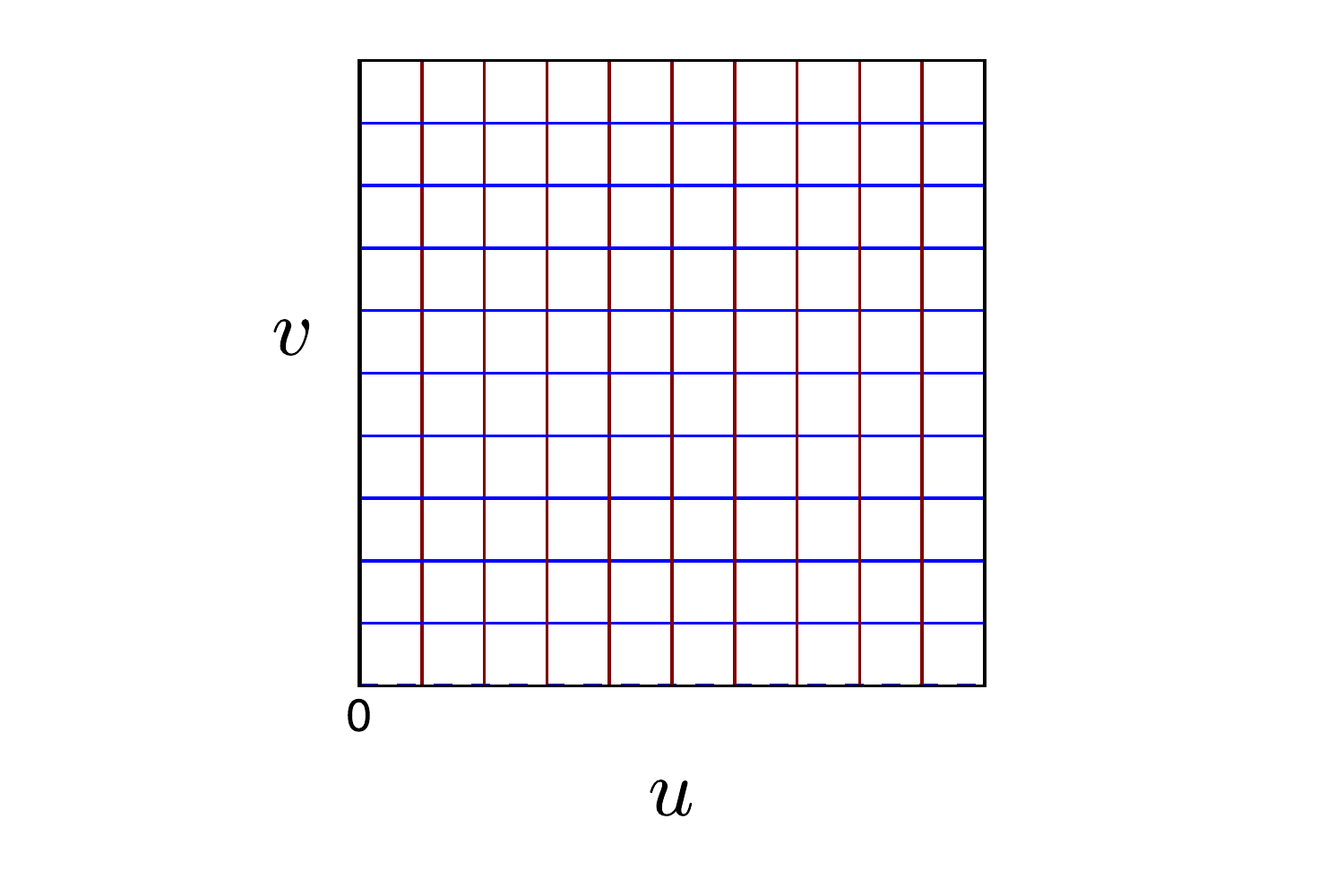}}
                \caption{}
                \label{fig:patching1}
        \end{subfigure}%
        \begin{subfigure}[b]{0.33\textwidth}
                \centering
                {\includegraphics[trim={3cm, 0.5cm, 3.5cm, 0.5cm}, clip, width=.85\linewidth]{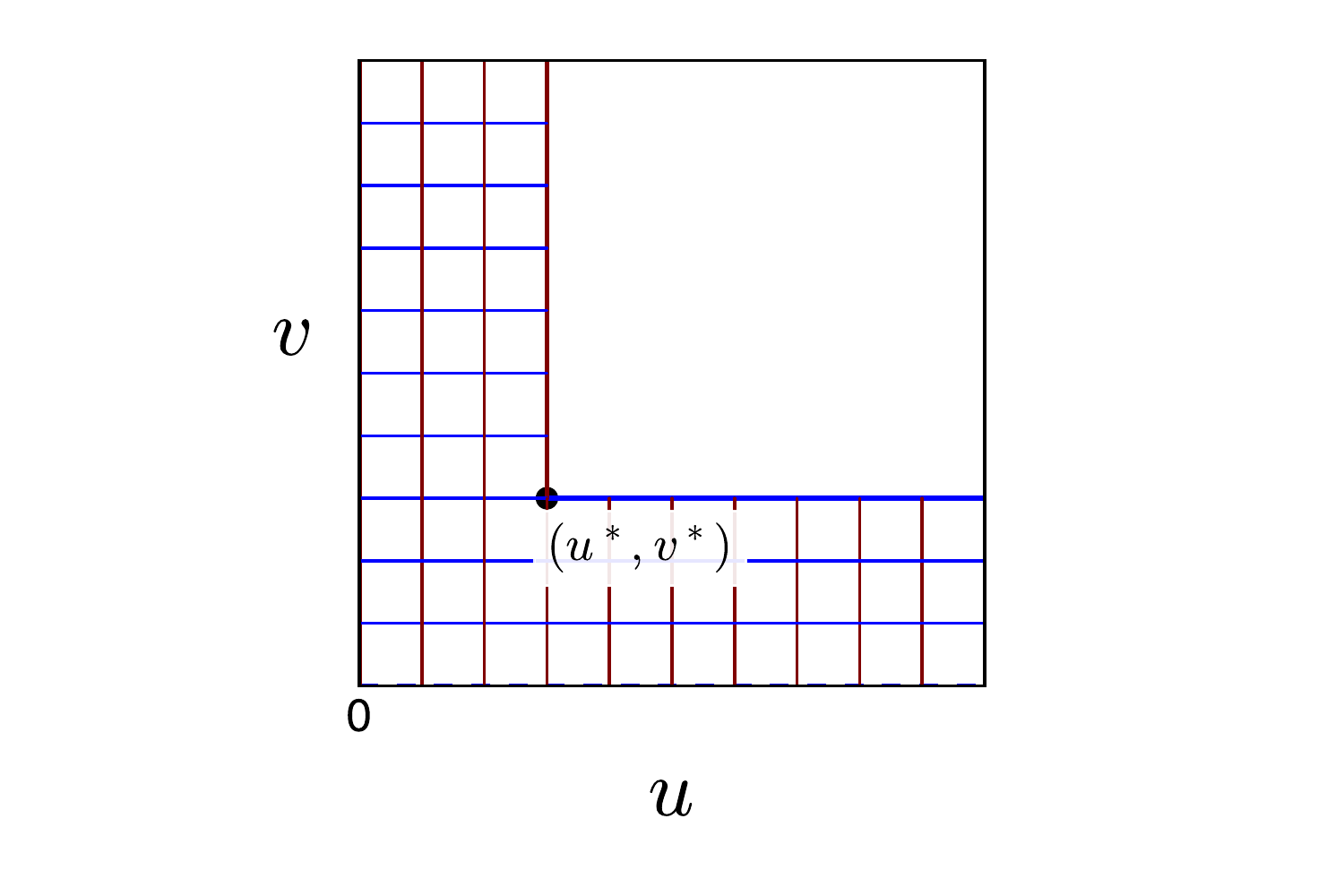}}
                \caption{}
                \label{fig:patching2}
        \end{subfigure}%
        \begin{subfigure}[b]{0.33\textwidth}
                \centering
                {\includegraphics[trim={3cm, 0.5cm, 3.5cm, 0.5cm}, clip, width=.85\linewidth]{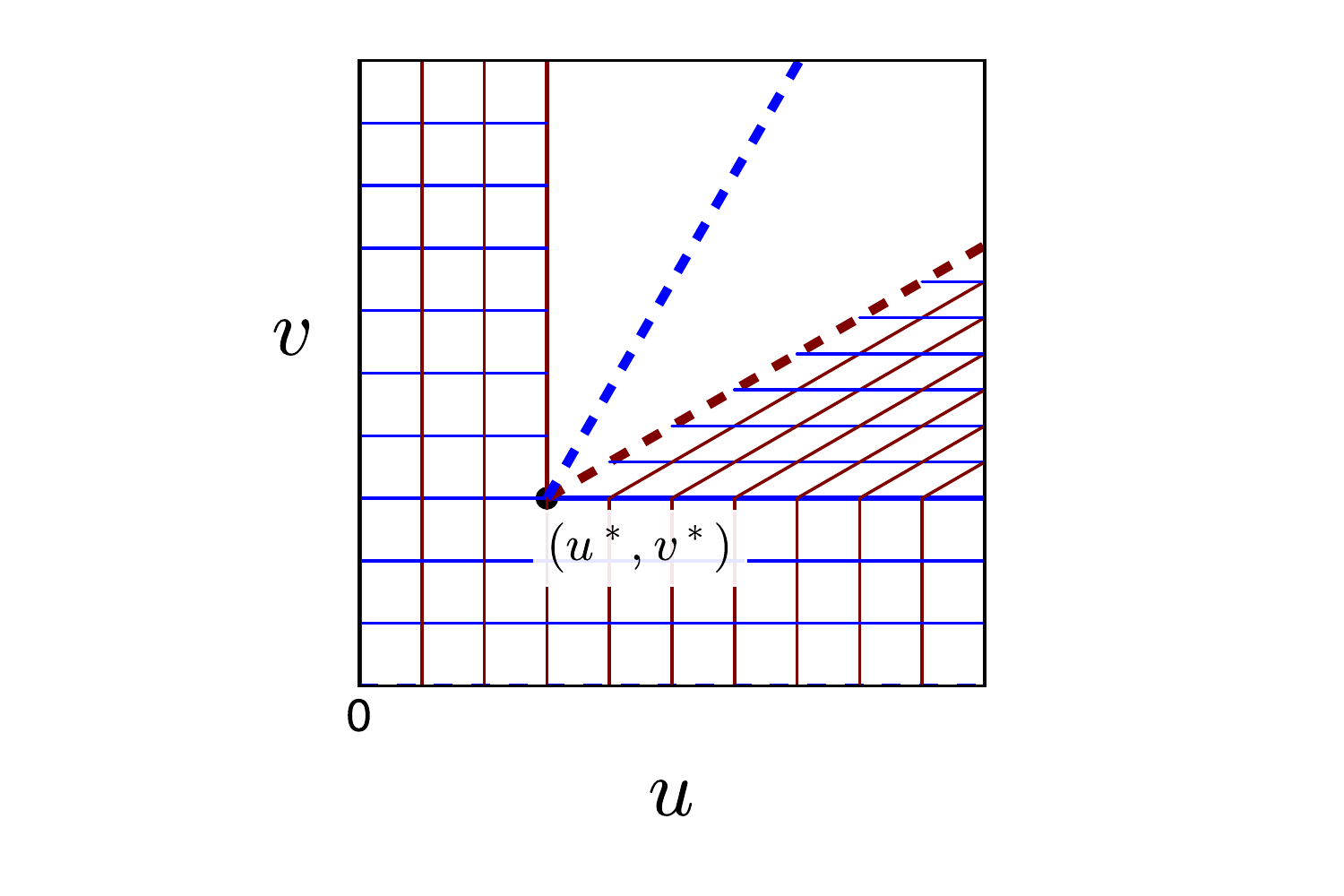}}
                \caption{}
                \label{fig:patching3}
        \end{subfigure}
        \caption{Surgery for asymptotic coordinate patches in $\Omega$. (a) $\Omega$, (b) $\Omega^*$, and (c) $\Omega^* \bigcup \Omega_1$. The normal field along the $u$-line in $\Omega_1$ is obtained by copying the corresponding data from the immersion of $\Omega^*$.}
        \label{fig:asymptoticcoordinatepatches}
\end{figure}

\begin{lemma}[Surgery] Let  $\Omega_0=  [0,u_{\max}] \times [0,v_{\max}]$ and let $r_0:\Omega_0 \to \mathbb{R}^3$ be a $C^{1M}$ PS-front. Given $(u^*,v^*)$ in the interior of $\Omega_0$ and $\tilde{u},\tilde{v} > 0$, let $\Omega^* = [0,u_{\max}] \times [0,v_{\max}] \setminus [u^*,u_{\max}] \times [v^*,v_{\max}], \Omega_1 = [0,u_{\max} -u^*] \times[0,\tilde{v}],  \Omega_2 = [0,\tilde{u}] \times[0,\tilde{v}], \Omega_3 = [0,\tilde{u}] \times  [0,v_{\max} -v^*].$  There exist PS-fronts $r_i : \Omega_i \to \mathbb{R}^3$ and attaching maps $\chi_j$ such that we can glue together $\Omega^*$ with $\Omega_i, i=1,2,3$ and the PS-front $\left. r_0\right|_{\Omega^*}$ with the PS-fronts $r_i,i=1,2,3$ to obtain a branched PS-front with a branch point at $(u^*,v^*) \in \Omega^*$.
\label{lem:surgery}
\end{lemma}
\begin{proof} 
We set $\mathbf{z}_1 = r_0(u^*,v^*),  \mathbf{n}_1 = N_0(u^*,v^*), \mathbf{t}_u = \partial_u r_0(u^*,v^*)$ and  $\mathbf{t}_v = \partial_v r_0(u^*,v^*)$. We define the asymptotic complex $A$ using the attaching maps  
\begin{align}
\chi_1: (u,0) \in \Omega_1 \mapsto (u^*+u,v^*) \in \Omega^*, & \quad  \chi_2: (u,0) \in \Omega_3 \mapsto (u,0) \in \Omega_2  \nonumber \\
\chi_3: (0,v) \in \Omega_3 \mapsto (u^*,v^*+v) \in \Omega^*, & \quad  \chi_4: (0,v) \in \Omega_1 \mapsto (0,v) \in \Omega_3 
\label{eq:attaching}
\end{align}

We construct PS-fronts $r_1,r_2$ and $r_3$ on the rectangles $\Omega_1 = [0,u_{\max}-u^*] \times[0,\tilde{v}],\Omega_2 = [0,\tilde{u}] \times [0,\tilde{v}]$ and $\Omega_3 = [0,\tilde{u}] \times [0,v_{\max}-v^*]$ respectively which are then assembled with the PS-front $r_0$ on $\Omega^*$ as in \S\ref{sec:kminusonemonkeysaddle}. The procedure for gluing the patches is outlined in Fig.~\ref{fig:asymptoticcoordinatepatches}, and the corresponding gluing procedure for the immersions, $r_i$ is illustrated in Fig.~\ref{fig:branchpointintroduction}.

We will take $r_2$ to be an Amsler patch on $\Omega_2$ with boundary conditions given by~\eqref{amsler-frame} with data inherited from $r_0$ by attaching at $(u^*,v^*)$. Specifically, we set 
\begin{equation}
\mathbf{z} = \mathbf{z}_1, \quad \mathbf{n} = \mathbf{n}_1, \quad \mathbf{e}^{(1)}_u = \frac{2  \mathbf{t}_v+  \mathbf{t}_u}{\|2  \mathbf{t}_v+  \mathbf{t}_u\|}, \quad \mathbf{e}^{(1)}_v = \frac{2  \mathbf{t}_u+  \mathbf{t}_v}{\|2  \mathbf{t}_u+  \mathbf{t}_v\|},
\end{equation}
as an approximation to trisecting the angle between the asymptotic curves at the branch point. Solving~\eqref{eq:moutard}  and~\eqref{eq:lelieuvre} gives  $N_2$ and the corresponding  PS-front $r_2$.

\begin{figure}[htbp]
\center
        \begin{subfigure}[b]{0.5\textwidth}
                \centering
                \includegraphics[width=.85\linewidth]{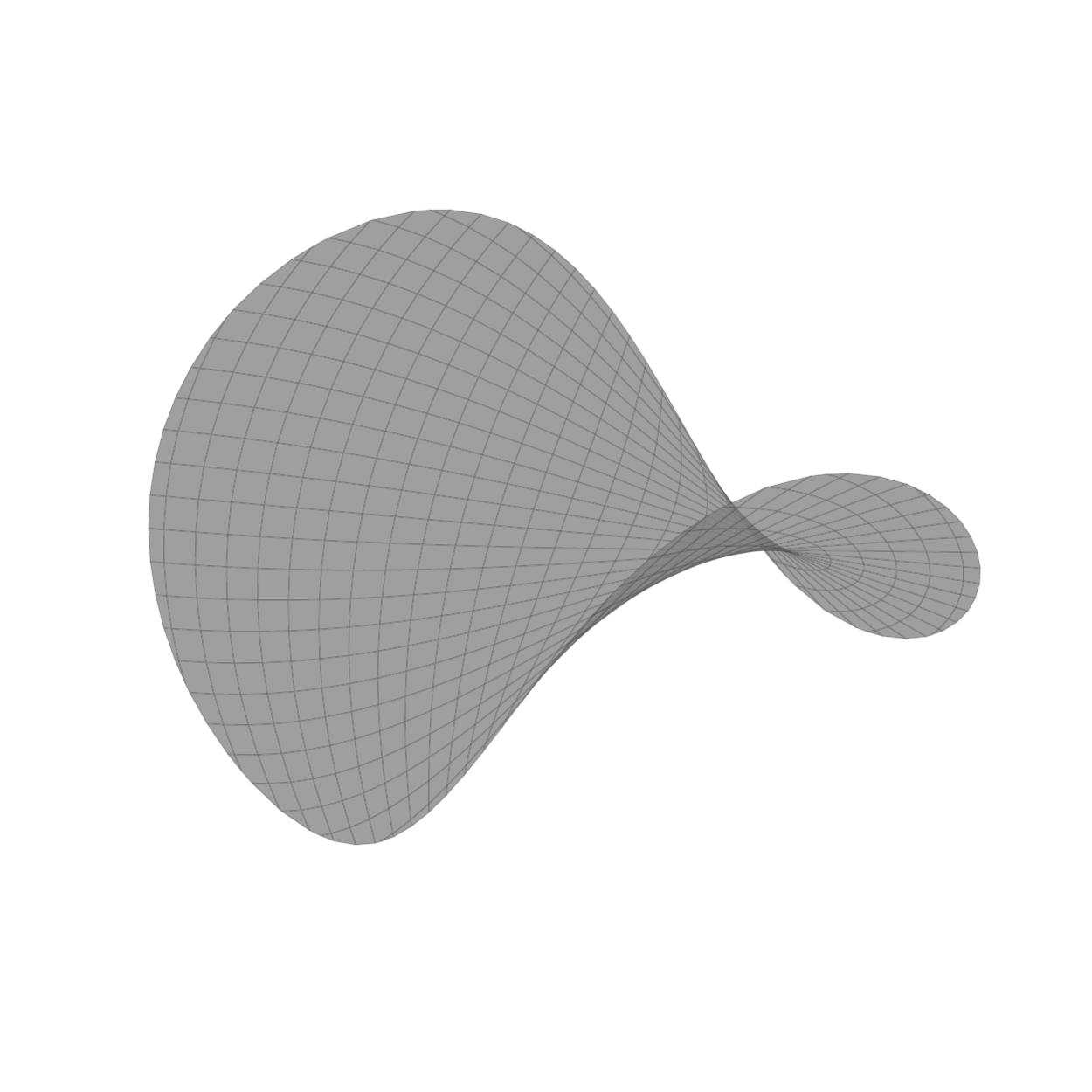}
                \caption{}
                \label{fig:introduce_branch_point_stage0}
        \end{subfigure}%
        \begin{subfigure}[b]{0.5\textwidth}
                \centering
                \includegraphics[width=.85\linewidth]{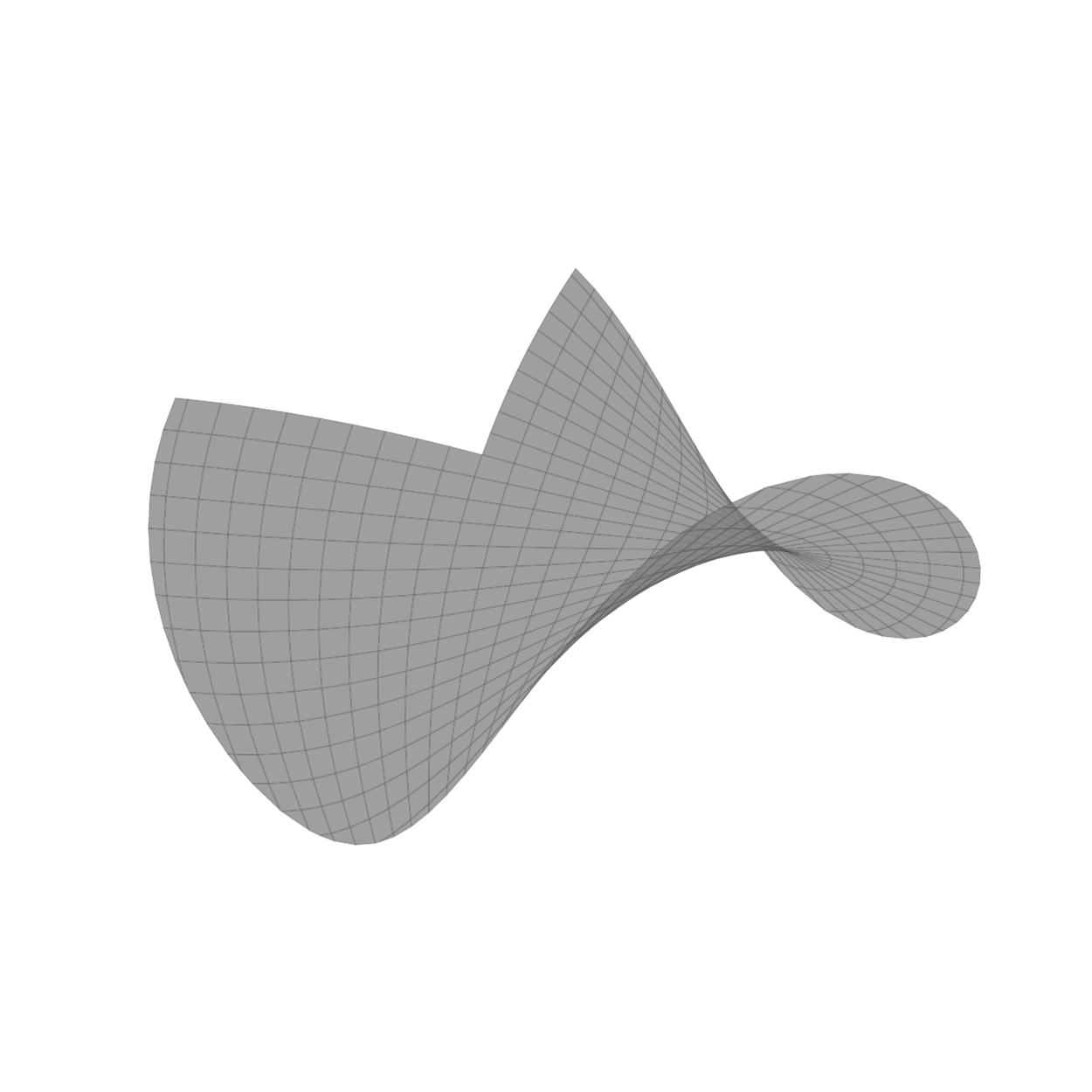}
                \caption{}
                \label{fig:introduce_branch_point_stage1}
        \end{subfigure}

        \begin{subfigure}[b]{0.33\textwidth}
                \centering
                \includegraphics[width=.85\linewidth]{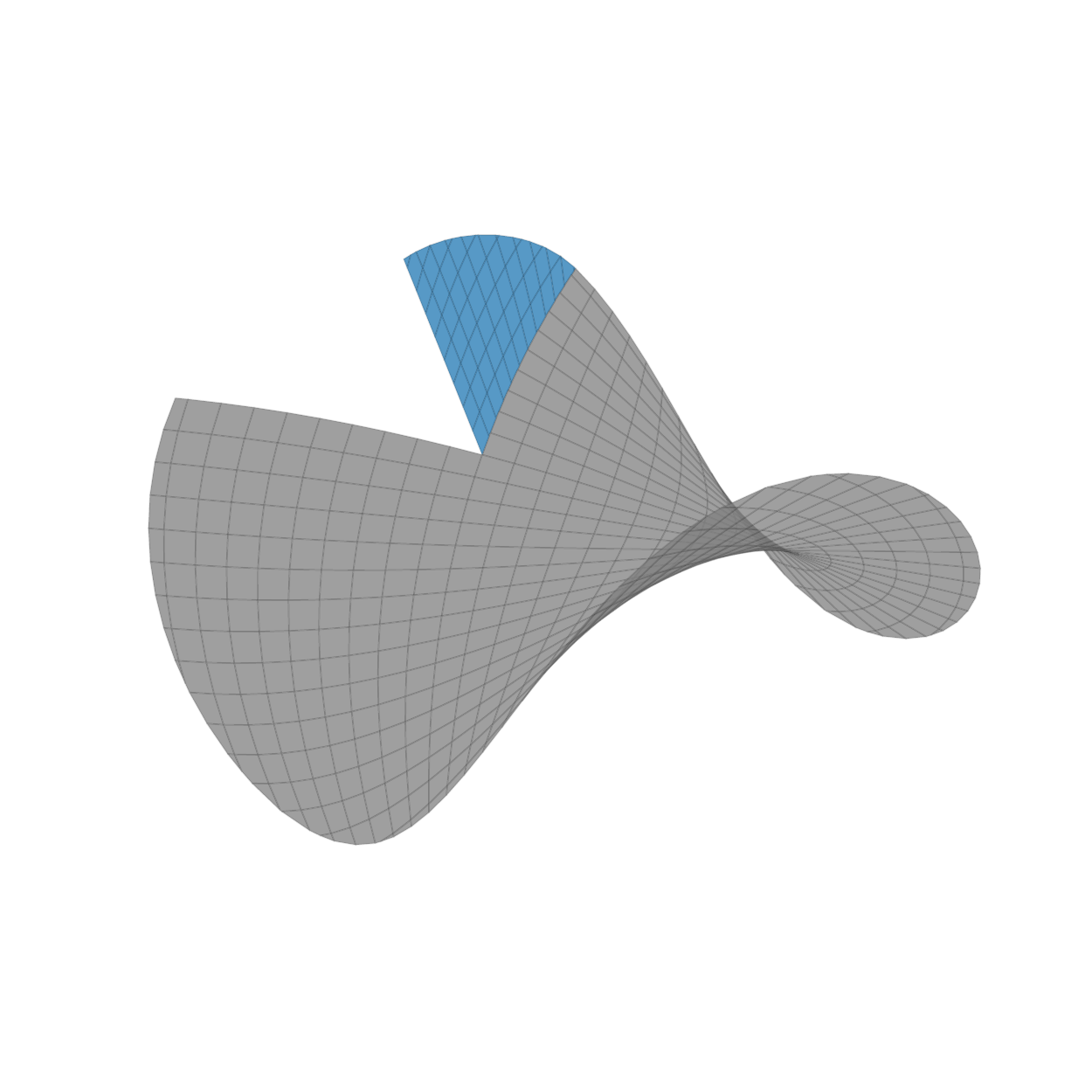}
                \caption{}
                \label{fig:introduce_branch_point_stage2}
        \end{subfigure}%
        \begin{subfigure}[b]{0.33\textwidth}
                \centering
                \includegraphics[width=.85\linewidth]{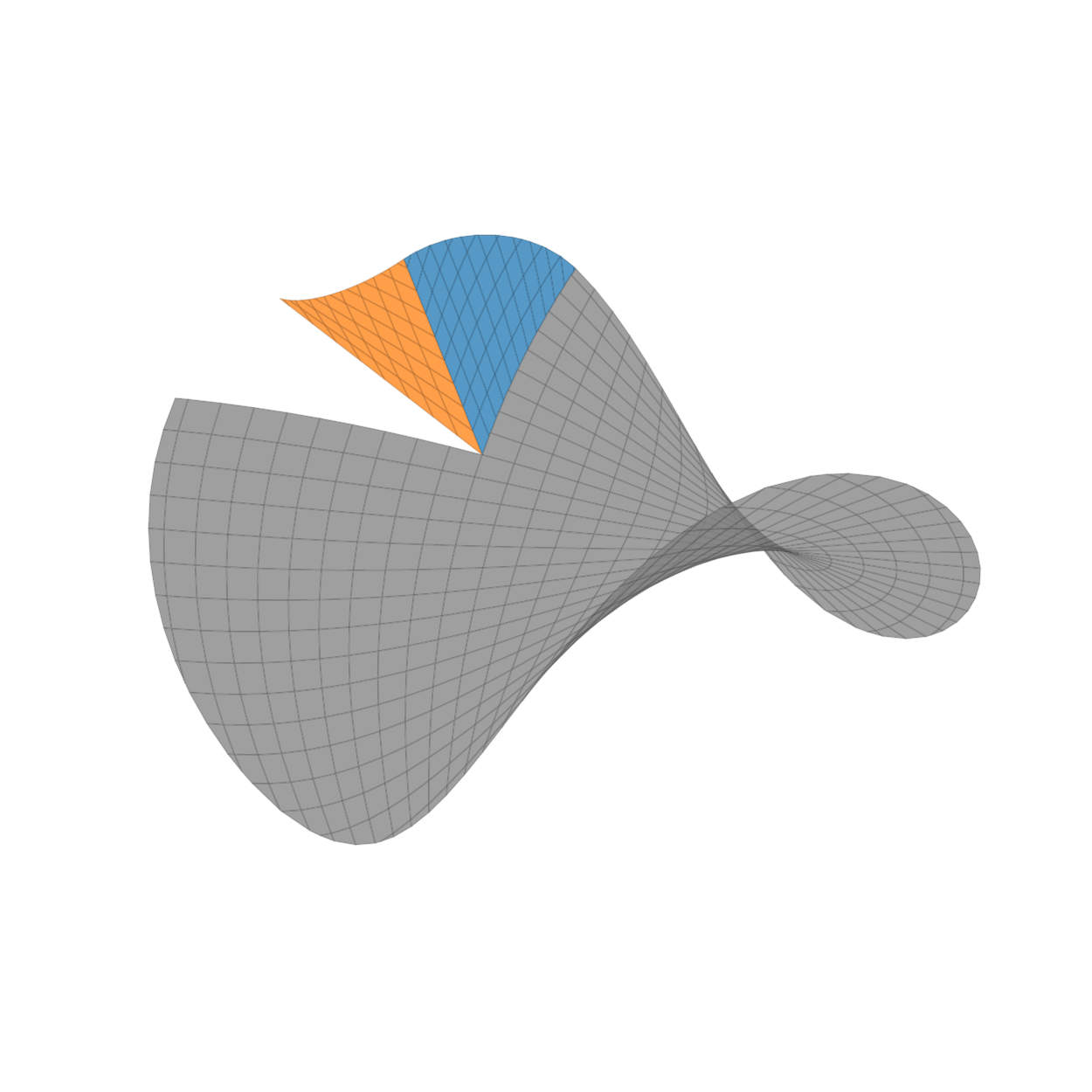}
                \caption{}
                \label{fig:introduce_branch_point_stage3}
        \end{subfigure}%
        \begin{subfigure}[b]{0.33\textwidth}
                \centering
                \includegraphics[width=.85\linewidth]{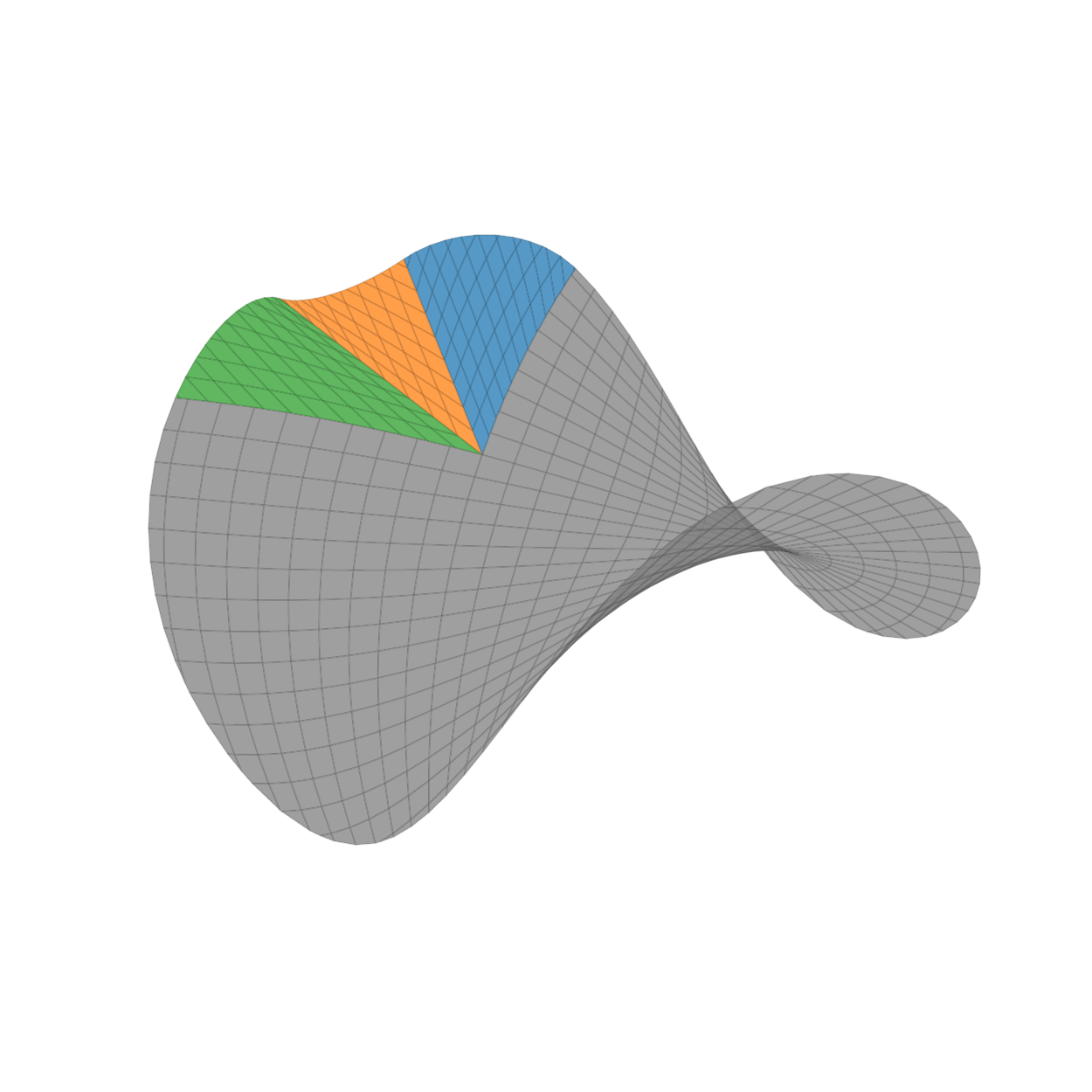}
                \caption{}
                \label{fig:introduce_branch_point_stage4}
        \end{subfigure}%
        \caption{Introducing a branch point into a smooth pseudospherical surface away from the origin. The resulting sectors have curved edges. 
       }
        \label{fig:branchpointintroduction}
\end{figure}

To build the Gauss map, $N_1:\Omega_1\to S^2$, again, we need only prescribe normal data along the axes: $u\geq 0$ and $v\geq 0$, where the coordinates $(u,v)$ are now ``local" to $\Omega_1$. 
We get data along $v = 0$ by copying it from the normal field $N_0$ using the attaching map $\chi_1$: 
\begin{equation}
N_1(u, 0) = N_0(\chi_1(u,0))\textrm{ for } u \in [0, u_{\max}-u^*],
\end{equation}
The data for $N_1$ along $u = 0$ comes from the PS-front $r_2$:
\begin{equation}
N_1(0, v) = N_2(\chi_4(0, v)) = \cos(v) \mathbf{n}_1 - \sin(v) \mathbf{n}_1 \times \mathbf{e}^{(1)}_v \textrm{ for } v \in [0, \tilde{v}],.
\end{equation}
We can now obtain a weakly harmonic normal field $N_1$ by solving the Moutard equation~\eqref{eq:moutard} on the rectangle $\Omega_1$ and then integrating the Lelieuvre equations to obtain the PS-front $r_1$. A similar procedure yields $N_3$ and $r_3$.

\begin{figure}[htbp]
\center
        \begin{subfigure}[b]{0.5\textwidth}
                \centering
                {\includegraphics[trim={1.5cm, 2.5cm, 1cm, 2.5cm}, clip, width=.85\linewidth]{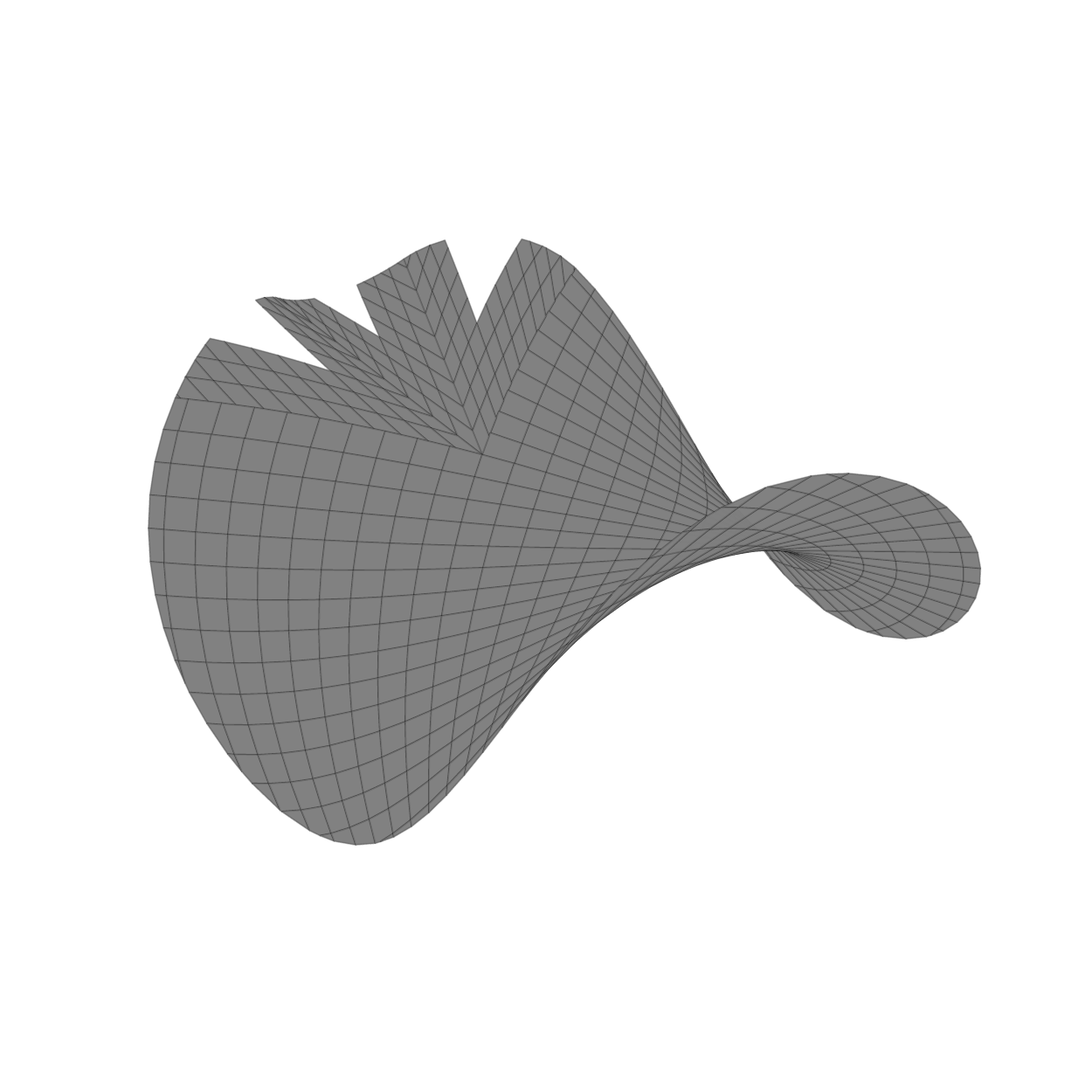}}
                \caption{}
                \label{fig:introduce_branch_point_stage5}
        \end{subfigure}%
        \begin{subfigure}[b]{0.5\textwidth}
                \centering
                {\includegraphics[trim={1.5cm, 2.5cm, 1cm, 2.5cm}, clip, width=.85\linewidth]{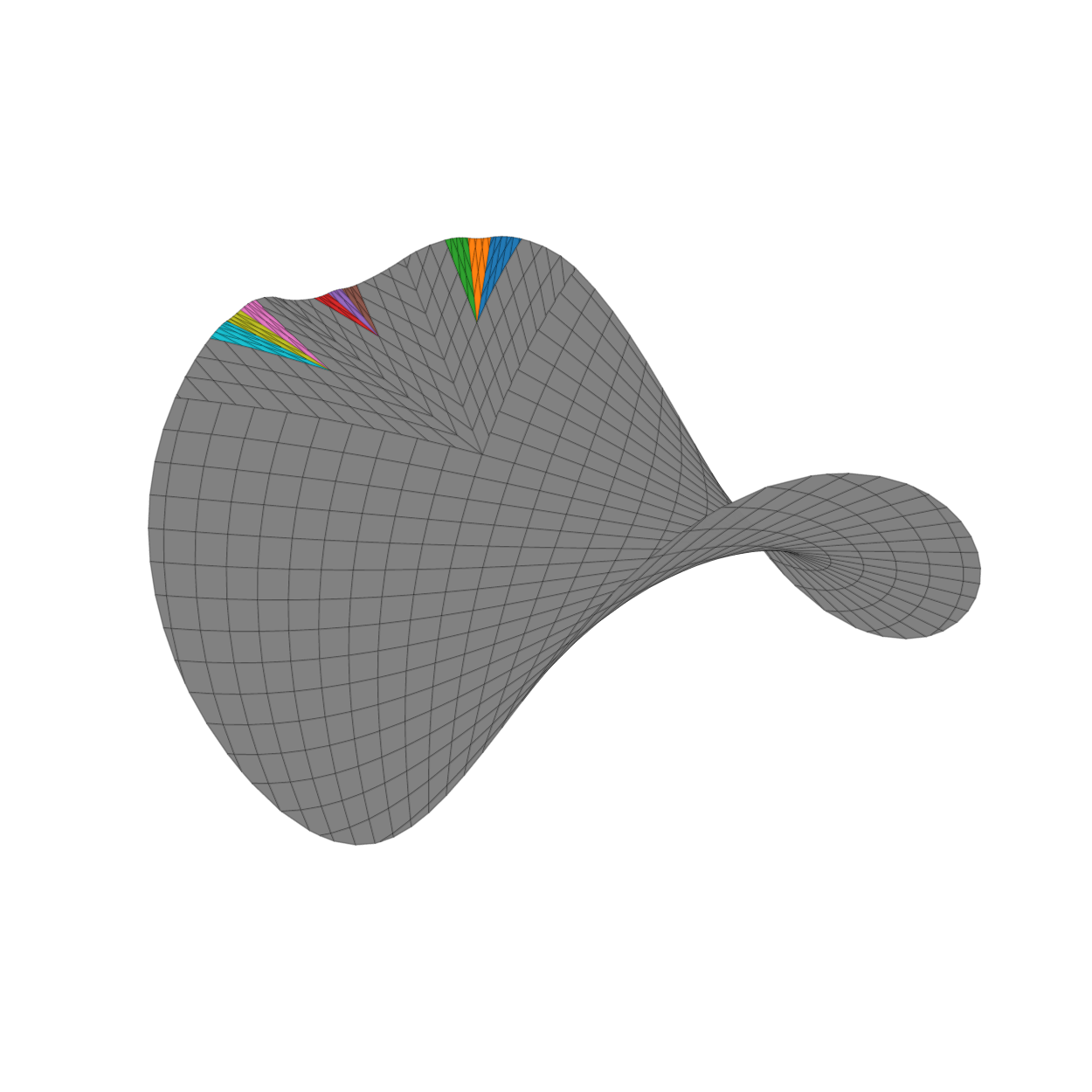}}
                \caption{}
                \label{fig:introduce_branch_point_stage6}
        \end{subfigure}
        \caption{Recursively performing surgery on an initially smooth surface.}
        \label{fig:recursiveconstruction}
\end{figure}

By construction $N_0 = N_1 \circ \chi_1 \Rightarrow r_0 = r_1 \circ \chi_i$ on $\Omega^* \cap \Omega_1$, and similar relations hold on all the edges where asymptotic quadrilaterals intersect. We can therefore assemble the PS-fronts $r_0,r_1,r_2$ and $r_3$ to obtain a branched PS-front $\psi:A \to \mathbb{R}^3$ that agrees with $r_0$ on $\Omega^*$, and on a subdomain such where $N_u \times N_v$ does not vanish, to obtain a $C^{1,1}$ isometric immersion with $K=-1$. 
The topological structure of the asymptotic lines corresponds to  a monkey saddle ($2m = 6$) at the branch point $(u^*, v^*)$ -- there are six asymptotic rays extending from the branch point. 
\end{proof}
It is clear how we can repeat this procedure recursively by picking branch point, cutting out one sector from this branch point, and replacing it with 3 new sectors. We call this procedure {\em surgery} to contrast it with the procedure in \S\ref{sec:kminusonemonkeysaddle}, which we refer to as assembly.  Surfaces with a second generation of branch points are shown in Figure \ref{fig:recursiveconstruction}.

\subsection{The Sine-Gordon equation for surfaces with branch points}
\label{sec:sinegordon}

Let $f:\Omega \to \mathbb{R}^3$ be a smooth pseudospherical immersion, so that the asymptotic curves and the angle function $\varphi(u,v)$ are differentiable. We can define a one form $\alpha = \frac{1}{2}(\varphi_v dv - \varphi_u du)$ and an area 2-from $\beta = \sqrt{\mathrm{det}(g_{ij})} \,du \wedge dv$ where $g = du^2 + 2 \sigma \cos \varphi \,du dv + dv^2$ and the sign of the square root is picked so that the orientation induced by $\beta$ agrees with the orientation induced by $\omega$ or equivalently, by $N^\omega$ (See Eq.~\eqref{eq:metrics}). It is now straightforward to check that $\beta = \sigma \sin \varphi \,du \wedge dv$. On a domain where $\sigma$ does not change sign, the sine-Gordon equation~\eqref{eq:sg} is equivalent to $d \alpha - \beta = 0$. Integrating over an asymptotic quadrilateral $R = \{u_0 \leq u \leq u_1, v_0 \leq v \leq v_1\}$ we obtain the Hazzidakis formula
\begin{equation}
\Delta_R \varphi \equiv \varphi(u_0,v_0) - \varphi(u_0,v_1) + \varphi(u_1,v_1) - \varphi(u_1,v_0) = A(R)
\label{eq:hazzidakis}
\end{equation}
where $\Delta_R \varphi = \sum (-1)^{\ell_i} \varphi_{i}$, $i$ indexes the vertices in the quadrilateral, $\ell_i$ is the modulo 2 length of any path from the vertex $(u_0,v_0)$ to the vertex labelled $i$, and $A$ is the area of (the immersion of) the quadrilateral. In order that $R$ be immersed into $\mathbb{R}^3$, we must have $0 < \varphi(u,v) < \pi$ on $R$, which gives $A(R) < 2 \pi$ for any immersed asymptotic quadrilateral. The Hazzidakis formula~\eqref{eq:hazzidakis} holds even in circumstances where $\varphi$ is not differentiable. For $C^{1M}$ PS-fronts $\varphi$ only needs to be $C^0$ but this formula still holds and the sine-Gordon equation can be interpreted in a distributional sense \cite{dorfmeister2016pseudospherical}.

\begin{definition}[Hamburger polygons] A Hamburger polygon $\gamma$ is a piecewise $C^1$ Jordan curve that bounds a domain, $\gamma = \partial \Gamma$, and consists of arcs that are either $u$ or $v$ asymptotic curves  \cite{hamburger1921},\cite[\S 3.3]{Lorentz_surfaces_Weinstein}.
\end{definition}
Eq.~\eqref{eq:hazzidakis} naturally extends to Hamburger polygons contained in domains where the immersion $r$ is $C^2$. Integrating the sine-Gordon equation $d \alpha - \beta = 0$ on $\Gamma$, we get
\begin{equation}
    \Delta_\Gamma \varphi \equiv \sum_i (-1)^{\ell_i} \varphi_{i} = \oint_\gamma \alpha = \int_\Gamma \beta = A(\Gamma),
    \label{eq:sg-hamburger}
\end{equation}
where $i$ indexes the vertices in the Hamburger polygon and $\ell_i$ is $0 \bmod 2$ at every initial vertex for an arc from the $u$-family (also a terminal vertex for a $v$-arc) and $\ell_i = 1 \bmod 2$ at every terminal vertex of a $u$-arc (resp. initial vertex of a $v$-arc), with respect to the orientation on $\gamma$ that is induced by $\omega$. 

Asymptotic quadrilaterals (Definition~\ref{defn:quad}) and $m$-stars (Definition~\ref{def:m-star}) are bounded by asymptotic curves, so they are examples of Hamburger polygons. However, Eq.~\ref{eq:sg-hamburger} is only guaranteed to apply to $C^2$ asymptotic quadrilaterals, agreeing with the Hazzidakis formula~\eqref{eq:hazzidakis}. Every $m$-star with $m > 2$ contains a branch point, where the immersion is not smooth, so further work is needed to deduce the analog of Eq.~\ref{eq:sg-hamburger} for $m$-stars, or more generally for $C^{1,1}$ branched pseudospherical surfaces. For $C^{1M}$ surfaces, with a continuous $\varphi$, we see that $\Delta_\Gamma \equiv \Delta(\Gamma) \to 0$ as $A(\Gamma) \to 0$, so there is no {\em concentration} for the quantity $\Delta_\Gamma = \oint_{\partial \Gamma} \alpha$ on sets of vanishing area. 

For branched surfaces, $\varphi$ is not always continuous and $\varphi$ necessarily has jumps across the asymptotic curves that are incident on a branch point.  This might potentially result in concentration of $\Delta$ on these ``singular" objects. We can determine the potential concentrations of $\Delta$ on branch points, and along the asymptotic curves that are incident on branch points, by using appropriate Hamburger polygons as illustrated in Fig.~\ref{fig:concentrations}.

\begin{figure}[ht]
\centering
        \begin{subfigure}[t]{0.5\textwidth}
                \centering
                {\includegraphics[height=6cm]{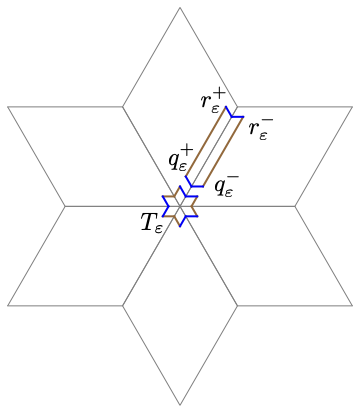}}
                \caption{}
                \label{fig:sinegordonepsrectangle}
        \end{subfigure}%
        \begin{subfigure}[t]{0.5\textwidth}
                \centering
                {\includegraphics[height=6cm]{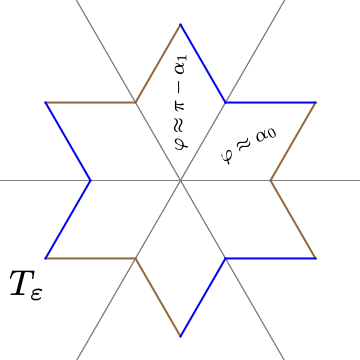}}
                \caption{}
                \label{fig:sinegordonepsstar}
        \end{subfigure}        
        \caption{(a) The Hamburger polygon $T_\varepsilon$ allows us to compute the concentration of $\Delta$ at the central branch points and the rectangle $R_\varepsilon = [q_\varepsilon^- r_\varepsilon^- r_\varepsilon^+ q_\varepsilon^+]$ determines the concentration on an asymptotic curve incident on the branch point. (b) Blowing up the polygon $T_\varepsilon$. The angle $\varphi$ is nearly constant on each sector. $\varphi = \alpha_{2i}$ on the even sectors and $\varphi = \pi - \alpha_{2i+1}$ on the odd sectors, where $\alpha_j$ is the angle between the asymptotic curves bounding the $j$\textsuperscript{th}sector.}
        \label{fig:concentrations}
\end{figure}

\begin{lemma}[Concentration at branch points] Let $T_i$ be an $m$-star that is obtained from $2m_i$ asymptotic quadrilaterals incident on a point $p_i$. Then $\Delta_{T_i}  = A(T_i) - (m_i-2) \pi$.
\label{lem:m-star}
\end{lemma}

\begin{proof}
From Definition~\ref{def:m-star} of an $m$-star, we see that $\partial T_i$ is a $2m_i$ sided Hamburger polygon, as shown in Fig.~\ref{fig:sinegordonepsrectangle}. 
For the $\varepsilon$-thin ``rectangle" $R_\varepsilon$ shown in Fig.~\ref{fig:sinegordonepsrectangle}, we have
$ \Delta_{R_{\varepsilon}} = \varphi(r_{\varepsilon}^{+}) - \varphi(r_{\varepsilon}^{-}) + \varphi(q_{\varepsilon}^{-}) - \varphi(q_{\varepsilon}^{+}).
$
Let us first assume that this rectangle straddles a $u$-curve incident on $p_i$. In this case, we can estimate $\varphi(r_{\varepsilon}^{+}) - \varphi(q_{\varepsilon}^{+}) = \int \partial_u \varphi^+ du + O(\varepsilon)$ noting that the integral is taken entirely inside a sector at $p_i$, so there are no discontinuities along the integration path. Similarly, $\varphi(r_{\varepsilon}^{-}) - \varphi(q_{\varepsilon}^{-}) = \int \partial_u \varphi^- du + O(\varepsilon)$. Although $\varphi^+$ and $\varphi^-$, the limits of the angle $\varphi$ in approaching the boundary $S^+ \cap S^-$ from either side are different, their derivatives $\partial_u \varphi^{\pm} = -\kappa^u$ have to match, since they are both equal to the geodesic curvature of a $u$-curve that is common to both sectors (See Eq.~\eqref{eq:sg}). Consequently, $\Delta_{R_{\varepsilon}} = O(\varepsilon)$. A similar argument also applies to $v$-curves incident on $p_i$. Thus, there is no concentration of $\Delta$ along the asymptotic curves that are incident on branch points. 

We now consider the concentration of $\Delta$ on the branch point $p_i$  with order of saddleness $m_i$ enclosed by a $\varepsilon$-small, $m_i$-star $T_\varepsilon$, comprising of asymptotic rhombi $R_0,R_1,\ldots,R_{2 m_p-1}$ as shown in Figure \ref{fig:sinegordonepsrectangle}. As discussed in Prop.~\ref{prop:assembly}, the local structure is given by alternating sets of $m_i$ $u$-curves and $m_i$ $v$-curves that are incident at $p$ with well defined tangent directions. 
Let $\alpha_j, j=0,1,2,\ldots,2m_p-1$ denote the angle of the rhombus $R_j$ at $p_i$ with respect to the orientation $\omega$ induced by the normal $N(p_i)$. This is consistent with the definitions in \S \ref{sec:kminusonemonkeysaddle}. Clearly $\sum_{j=0}^{2m_i-1}\alpha_j = 2\pi$. From Eq.~\eqref{eq:def_phi}, we see that the angles between the asymptotic directions at $p_i$ are given by comparing the sense of the rotation from $r_u$ to $r_v$, chosen to be directed away from $p$, with the orientation induced by $\omega$:
\begin{equation}
	\varphi_j = \left\{\begin{matrix} \alpha_j &\text{if $r_u$ to $r_v$ is counter-clockwise}  \\ \pi - \alpha_j& \textrm{otherwise} \end{matrix}\right.
\end{equation}

On each rhombus $R_j$, the surface restricts to a $C^2$ (even smooth) PS-front, so it follows that $\varphi$ is continuous. In particular, at the vertex $q_j$, diagonally across from $p$ in $R_j$, we have $\varphi(q_j) = \varphi_j + O(\varepsilon)$. We can now compute,
\begin{equation}
    \Delta_{T_{\varepsilon}} = \sum_j (-1)^i \varphi_j + O(\epsilon)  = -(m_p -2)\pi + O(\epsilon)
\end{equation}
Combining these results, with the contributions of the quadrilaterals that comprise the complement of the $\varepsilon$-thin rectangles and the $\varepsilon$-small $m_i$-star $T_\epsilon$, that are given by the Hazzidakis formula \eqref{eq:hazzidakis}, we get $\Delta_{T_i}  = A(T_i) - (m_i-2) \pi$.
\end{proof}

Note that the same argument also applies at points $p$ with $m_p = 2$. This lemma shows that branch points do indeed concentrate $\Delta$. This concentration, equal to $-(m_p-2) \pi$ at a point $p$, has a definite sign, and is zero at points where the surface is locally a $2$-saddle, as we would expect. 
It is straightforward to ``globalize" the arguments from above to get a generalization of the (integrated form) of the sine-Gordon equation that is valid even for $C^{1,1}$ branched pseudospherical immersions. We record this in the following theorem:
\begin{theorem} 
\label{thm:sg-branched} Let $r:(\Omega,g) \to \mathbb{R}^3$ be a branched pseudospherical immersion, with finitely many isolated branch points $p_i, i =1,2,\ldots,k$. Let $\Gamma \subset \Omega$ be a domain with compact closure in $\Omega$ whose boundary $\gamma = \partial \Gamma$ is a Hamburger polygon with vertices $q_0,q_1,\ldots, q_{2j-1}$ and $q_0$ is an initial vertex for a $u$-arc with respect to an orientation $\omega$ on $\Omega$. Then, we have 
\begin{equation}
  \Delta_\Gamma \equiv \sum_{n=0}^{2j-1} (-1)^n \varphi(q_n) = A(\Gamma) - \sum_{p_i \in \Gamma} (m_i-2) \pi
  \label{sg-global}
\end{equation}
where $\varphi$, the angle between the asymptotic curves, is defined by $\sigma = \mathrm{sign}(\omega(\partial_u r,\partial_v r)), \varphi \in (0,\pi),\sin \varphi = \|\partial_u r \times \partial_v r\|, \cos \varphi = \sigma \partial_u r \cdot \partial_v r$. 
\end{theorem}
\begin{proof}
The domain $\Gamma$ decomposes into a union of finitely many $m$-stars, each enclosing a branch point, and a collection of finitely many asymptotic quadrilaterals. Therefore $\Gamma = \bigcup_{j=1}^N \Gamma_J$ where each $\Gamma_j$ is a Hamburger polygon. Since $\omega$ will induce opposite orientations on a edge that is in $\Gamma_j \bigcap \Gamma_{j'}$ with $j \neq j'$, it is easy to see that $\Delta_\Gamma = \sum_{j=1}^M \Delta_{\Gamma_j}$. The theorem now follows from the additivity of the area $A$, the Hazzidakis formula~\eqref{eq:hazzidakis} and the `concentration at branch points'  lemma~\ref{lem:m-star}.
\end{proof}
\begin{remark}
\label{rmk:distributed}
The principal curvatures of a pseudospherical immersion are given by $\kappa_1 = \tan \frac{\varphi}{2}, \kappa_2 = -\cot \frac{\varphi}{2}$ so $\kappa_1 \kappa_2 = -1$. The Willmore energy is given by a density $\kappa_1^2 +\kappa_2^2$, and the $W^{2,\infty}$ energy is given by $\sup_{x \in \Omega} \max(|\kappa_1(x)|,|\kappa_2(x)|)$. In either case, optimizing the energy demands that we keep $\varphi \approx \frac{\pi}{2}$ everywhere.

If $\varphi$ were identically equal to $\frac{\pi}{2}$, the left hand side of ~\eqref{sg-global} is zero since there are equal number of positive and negative contributions from $(-1)^n\varphi(q_n)$. The right hand side, however, is a difference between two positive quantities, the continuously varying quantity $A(\Gamma)$ and a discrete quantity $\sum_{p_i \in \Gamma} (m_i-2) \pi$. It is therefore impossible to have $\varphi \equiv \frac{\pi}{2}$ everywhere. This underscores the need to distribute branch points on $\Omega$ so there is  ``quasi-local" cancellation between the area form and the branch point contributions, i.e. energy optimal branched pseudospherical immersions will arise from attempting to place, on average, 1 branch point with $m=3$ in every Hamburger polygon $\Gamma$ with area $A(\Gamma) = \pi$. Each such branch point adds an extra undulation to the surface, that persists from the branch point out to the boundary.
\end{remark}

\section{Discrete differential geometry for branched pseudospherical surfaces} \label{sec:ddg}

Our goal is to construct discrete analogs of the geometric notions in \S\ref{sec:branchedsurfaces}. As in Prop.~\ref{prop:assembly}, branched surfaces are realized by patching asymptotic rectangles, with the combinatorics given by the underlying asymptotic complex. Following this approach, we will build discrete PS-fronts by appropriate gluing of discrete $K$-surfaces (see Definition \ref{def:K-surf} below.)

Asymptotic rectangles are discretized by rectangular subsets of $\epsilon \mathbb{Z}^2$ for sufficiently small $\epsilon >0$. Indeed, there is a natural inclusion $\lambda_k:M_k := \{0,\epsilon,2\epsilon,\ldots,i_k \epsilon\} \times \{\epsilon,2\epsilon,\ldots,j_k\epsilon\} \subset \epsilon \mathbb{Z}^2 \to F_k$ given by inverting the bijection $\psi_k:F_k \to [0,u_k] \times [0,v_k]$ (See Definition~\ref{def:A-complex}. WLOG we can assume $u_k,v_k$ are multiples of $\epsilon$ using small perturbations if necessary). The sets $\{(i \epsilon ,j_0 \epsilon) | \,0 \leq i \leq i_k\}$ and $\{(i_0 \epsilon,j\epsilon) |\, 0 \leq j \leq j_k\}$ 
are the `discrete' $u$ and $v$ asymptotic curves.

Rectangular subsets of $\epsilon \mathbb{Z}^2$ have a natural quadgraph structure given by the faces $[i \epsilon, (i+1) \epsilon] \times [j \epsilon, (j+1) \epsilon]$ and the natural attaching maps induced by inclusion into $\mathbb{R}^2$.  This structure, along with the attaching maps defining the asymptotic complex $A$, inherited through the mappings $\lambda_k:M_k \to F_k$, define a quadgraph $Q^\epsilon$, which will be the setting for our numerical constructions of (discrete) branched PS-fronts and pseudospherical surfaces.

As with the `continuous' construction in \S\ref{sec:branchedsurfaces}, we will first construct a discrete Lorentz-harmonic normal field $N^\epsilon: Q^\epsilon \to S^2$, and then determine the corresponding discrete immersion $r^\epsilon:Q^\epsilon \to \mathbb{R}^3$ using an appropriate discretization of the Lelieuvre equations~\eqref{eq:lelieuvre}.

Within each face of the asymptotic complex, 
generating a 
PS-front reduces to solving \eqref{eq:moutard}. 
As we discussed above the discretization of a face uses square grids, i.e. subsets of $\mathbb{Z}^2$, so we denote an arbitrary node by $(i, j)$. We use the following notation, which is standard in DDG \cite[Chap. 2]{bobenko2008bdiscrete}, to denote the discretization of a function $f$ on an elementary quad:
\begin{align}
f_{i, j} = f_0, f_{i+1, j} = f_1, f_{i, j+1} = f_2,\textrm{ and } f_{i+1, j+1} = f_{12}.
\end{align}

\begin{definition} \label{def:K-surf}[Discrete $K$-surface] A map $r : J \subseteq \mathbb{Z}^2 \to \mathbb{R}^3$ is called a {\em discrete $K$-surface} if and only if there exists a discrete map $N : J \to S^2$ such that, on every quad, 
\begin{equation}
r_{1}  = r_{0} + N_{1} \times N_{0},  \quad
r_{2}  = r_{0} - N_{2} \times N_{0}, 
\label{eq:d_lelieuvre}
\end{equation}
\end{definition}
Eqs.~\eqref{eq:d_lelieuvre} are the {\em discrete Lelieuvre equations} ({\em cf.} Eq.~\eqref{eq:lelieuvre}) and go back to the work of Sauer \cite{sauer1950parallelogrammgitter} and Wunderlich \cite{wunderlich1951differenzengeometrie}. The discrete Lelieuvre equations are natural discretizations of the Lelieuvre (differential) equations~\eqref{eq:lelieuvre}. They guarantee that $r_{i\pm1,j} -r_{i,j}$ and $r_{i,j \pm1} -r_{i,j}$ are orthogonal to $N_{i,j}$, i.e. the vertex stars are planar for any solution of~\eqref{eq:d_lelieuvre}. 

Definition~\ref{def:K-surf} only requires us to distinguish $u$-edges (corresponding to $r_1-r_0$) and $v$-edges (giving $r_2-r_0$) and therefore generalize naturally to discrete $K$-surfaces defined on Asymptotic complexes, through the requirement that ~\eqref{eq:d_lelieuvre} hold on each quad with the following labeling of vertices: Give one of the 4 vertices the index 0. Label the neighbor of 0 along a $u$-edge by 1 and the neighbor along a $v$-edge by 2. Finally label the diagonally opposite vertex $12$. On each quad we have two possible definitions of $r_{12}$, either from the path $0 \to 1 \to 12$ or the path $0 \to 2 \to 12$. Compatibility requires that 
\begin{equation}
    N_1 \times N_0 - N_{12} \times N_1 - (-N_2 \times N_0 + N_{12} \times N_2) = (N_1+N_2) \times (N_0 + N_{12}) = 0
\label{eq:d_compatible}
\end{equation}
Directly discretizing the (continuous) compatibility condition Eq.~\eqref{eq:lorentz_harmonic}, yields
\begin{align}
N_{uv} \times N =0  \mapsto 0 &  = \left( N_{12} + N_{0} - (N_{1} + N_{2})\right)\times\frac{\left(N_{0} + N_{1} + N_{2} + N_{12} \right)}{4}  \nonumber \\
					&= \frac{2}{4}\left( N_{0}\times N_{1} - N_{1}\times N_{12} + N_{12}\times N_{2} - N_{2}\times N_{0} \right) \nonumber \\
					&= \frac{1}{2}\left(N_{0} + N_{12}\right)\times\left( N_{1} + N_{2}\right) \label{eqn:discreteconstantcurvaturepde}
\end{align}
This is the same as equation~\eqref{eq:d_compatible}. This discretization therefore has the remarkable property that the discretization of the (continuous) compatibility condition for the Lelieuvre formulae {\em is exactly} the same as the discrete compatibility of the discrete Lelieuvre formulae, rather than, as one might plausibly imagine, a numerical approximation that recovers the exact result in the limit the discretization size $h$ goes to zero. This particular discretization exemplifies a key idea in discrete differential geometry (DDG). Rather than serving merely as numerical discretizations of the ``true" (continuous) differential geometry, DDG is  a complete theory in its own right \cite[p. xiv]{bobenko2008bdiscrete}.

We now give short, self-contained proofs of standard results from DDG for $K$-surfaces $r:\mathbb{Z}^2 \to \mathbb{R}^3$ (See the text \cite{bobenko2008bdiscrete} for further details). We first exhibit solutions for the discrete Goursat problem of specifying $N(i,0)$ and $N(0,j)$ and solving for $N(i,j)$, on a single quad. On an elementary quad, assume  that $N_{12}$ is unknown, while values for $N_0, N_1$ and $N_2$ are known. Then \eqref{eqn:discreteconstantcurvaturepde} requires 
\begin{equation*}
N_{12} = \nu (N_1 + N_2) - N_0,
\end{equation*}
for some $\nu\in\mathbb{R}$, as is the case for a Moutard net \cite[\S 2.3]{bobenko2008bdiscrete}. The condition that $N_{12}$ is a unit vector gives a  quadratic equation for $\nu$:
\begin{align*}
\langle N_{12}, N_{12} \rangle & = \nu^2\langle N_1 + N_{2}, N_1 + N_{2} \rangle - 2\nu\langle N_1 + N_{2}, N_0 \rangle + \langle N_0, N_0\rangle \\
						 & = \nu^2\|N_1 + N_{2}\|^2 - 2\nu\langle N_1 + N_{2}, N_0 \rangle + 1, \\
\intertext{which reduces to}
0 & = \nu\left( \nu\|N_1 + N_{2}\|^2 - 2\langle N_1 + N_{2}, N_0 \rangle \right).
\end{align*}
This implies that $\nu=0$ and $N_{12} = -N_0$ or $\nu = 2\frac{\langle N_1 + N_{2}, N_{0}\rangle}{\langle N_1 + N_{2}, N_1 + N_{2} \rangle}$ and
\begin{equation}
N_{12} = \left[\frac{(N_1 + N_{2})(N_1 + N_{2})^T}{\langle N_1 + N_{2}, N_1 + N_{2}\rangle} - \mathbb{I}\right] N_0.
\label{eq:householder}
\end{equation}
The former being the antipodal point, and the latter being the desired solution. This is the Householder reflection of $N_0$ through the plane generated by $N_1$ and $N_{2}$. Though we  solved for $N_{12} = N_{i+1, j+1}$ above, this approach can be used to solve for the fourth normal vector provided the normal at the three other corners is given (See Figure \ref{fig:chebyshevstwo}).

\begin{lemma}If $\|N_0-N_1\| = \|N_0-N_2\|$, and $N_{12}$ is determined by Householder reflection as in~\eqref{eq:householder}, it follows that $N_0N_1N_{12}N_2$ is a spherical rhombus. 
\label{lem:Chebyshev}
\end{lemma}
\begin{proof} Since the angle $\delta$ between $N_0$ and $N_1$ is the same as the angle between $N_0$ and $N_2$, $\langle N_0,N_1 \rangle = \langle N_0,N_2\rangle = \cos \delta$ and we have 
\begin{align}
    0 & = \langle N_{12} - N_{0},N_{12}+N_0\rangle = \nu(\langle N_{12},N_{1}+N_2\rangle - 2 \cos \delta) \nonumber \\
    0 & = \nu \langle N_{1} + N_{2},N_{1}-N_2\rangle = \nu\langle N_{12},N_{1}-N_2\rangle \nonumber \\
    \implies \cos \delta & = \langle N_{12},N_1 \rangle = \langle N_{12},N_2\rangle,
     \label{eq:d_chebyshev}
\end{align}
proving that $N_0N_1N_{12}N_2$ is a spherical rhombus
\end{proof}
Recursively applying~\eqref{eq:householder} we can solve for the normal field on an asymptotic quadrilateral if it is specified on two of its boundaries, as illustrated in Fig.~\ref{fig:goursat}. In addition, this procedure also determines the normal field on the other two boundaries. Since the $u$ and $v$ asymptotic curves have constant torsions (See~\eqref{eq:frames}) we can discretize these boundaries  so that $\langle N_{i,0},N_{i+1,0} \rangle = \langle N_{0,j},N_{0,j+1}\rangle = \cos \delta$. By~\eqref{eq:d_chebyshev}, we get $\langle N_{i,j},N_{i+1,j} \rangle = \langle N_{i,j},N_{i,j+1}\rangle = \cos \delta$ and $\|r_{i+1,j} - r_{i,j}\|= \|r_{i,j+1} - r_{i,j}\| = \sin \delta$ for all $i,j$, so the discrete surface $r_{ij}$ is a discrete Chebyshev net in $\mathbb{R}^3$ and the corresponding normal field $N_{ij}$ is a discrete Chebyshev net in $ S^2$ as illustrated in Fig.~\ref{fig:chebyshevstwo}. For our purposes, we need to generalize the ideas from above to consider mappings $r:Q \to \mathbb{R}^3$ and $N:Q \to S^2$, where $Q$ is a general asymptotic complex, and not restricted to be a subset of $\mathbb{Z}^2$. This motivates

\begin{figure}[ht]
\centering
\begin{subfigure}[t]{0.5\textwidth}
		\centering
        {\includegraphics[trim={0.cm, 0cm, 0.cm, 0cm}, clip, width=0.8\linewidth]{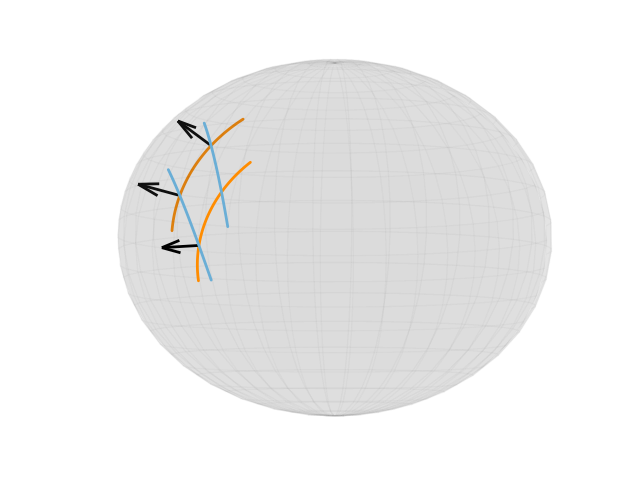}}
        \caption{}
        \label{fig:chebyshevstwo}
    \end{subfigure}%
    \begin{subfigure}[t]{0.5\textwidth}
		\centering
        {\includegraphics[trim={0.cm, 0cm, 0.0cm, 0cm}, clip, width=0.8\linewidth]{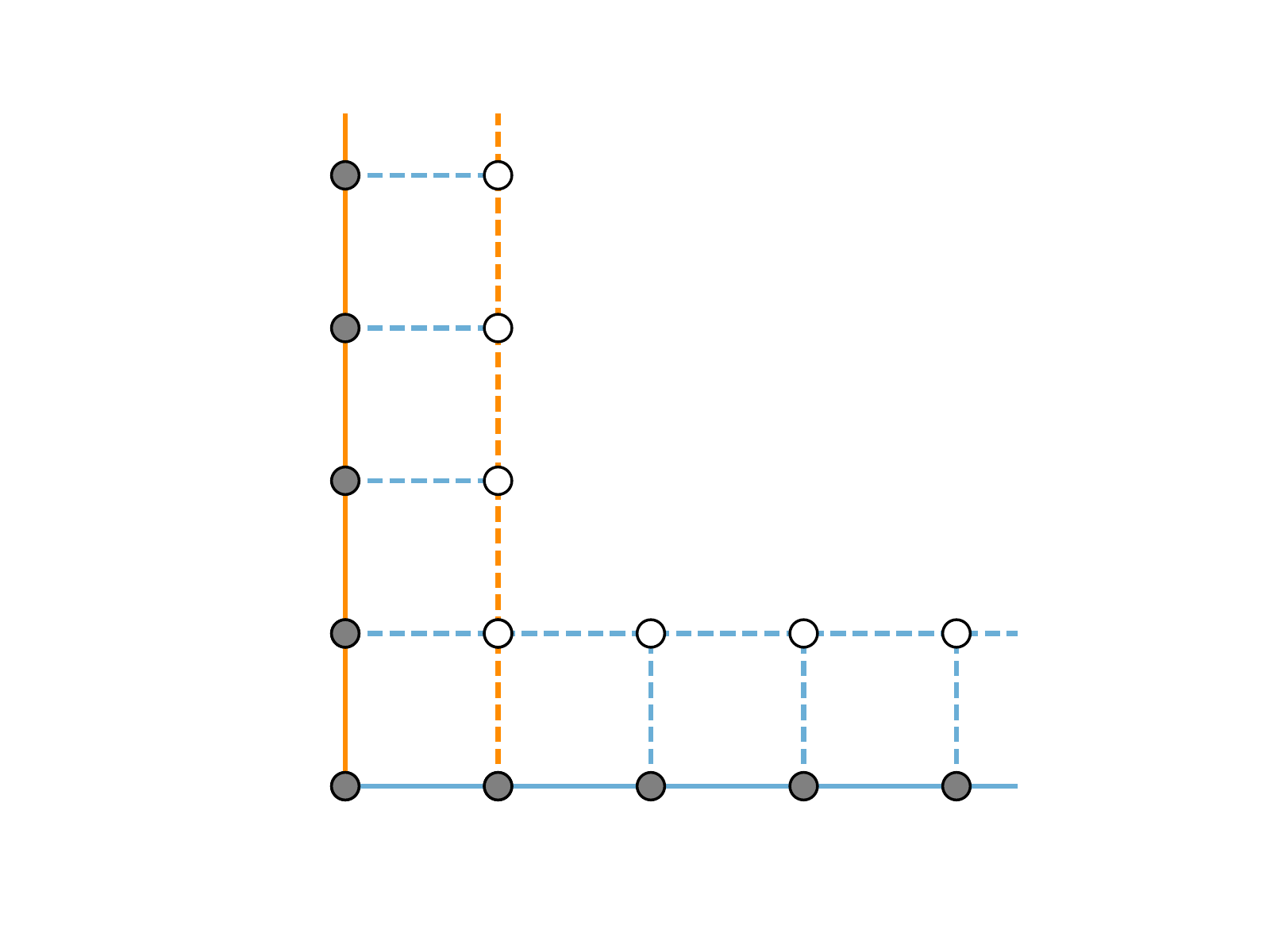}}
        \caption{}
        \label{fig:goursat}
    \end{subfigure}%
    \caption{(a) A single quadrilateral in the induced Chebyshev net on $ S^2$. Given the normal vectors at three vertices, the normal at the fourth vertex is determined. (b) The Goursat-type discretized problem on the asymptotic quadrilateral, $\Omega$. The nodes filled with grey represent provided boundary data, and open nodes are iteratively solved for via the system \eqref{eqn:discreteconstantcurvaturepde}}
    \label{fig:scri}
\end{figure}%

\begin{definition}[Spherical Chebyshev net] A spherical Chebyshev net is a branched embedding $N:Q \to S^2$ of an asymptotic complex $Q \subset \mathbb{R}^2$ into the sphere that (i) maps every quad onto a spherical rhombus, (ii) reverses orientation, and (iii) satisfies 
\begin{equation}
  \begin{cases} \sum_{p \in F_k} \alpha_k = 2 \pi (1-m_p) & p \mbox{ in the interior has degree }2 m_p, \\ \sum_{p \in F_k} \alpha_k \in ((1-d_p)\pi,\min((3-d_p) \pi, 0)) & p \mbox{ on the boundary has degree }d_p, \\  \end{cases}
\label{eq:unramified}
\end{equation}
where the sums are over all the faces $F_k$ incident on $p$, and $\alpha_k$ is the (negative) angle at $p$ for the image $N(F_k)$.
\label{def:sphchebnet}
\end{definition}
Condition~\eqref{eq:unramified} enforces the hypothesis in Lemma~\ref{lem:covering-number} at interior vertices, and allows for ``closing"  an edge (respectively corner) vertex with $d_p$ odd (resp. even), i.e. making it an interior vertex by adding  2 (resp. 3) sectors with angles in $(-\pi,0)$. 
\begin{lemma}
Let $Q^\epsilon$ be an Asymptotic complex (a simply connected, checkerboard colorable quadgraph). For any spherical Chebyshev net $N:Q^\epsilon \to S^2$, the discrete Lelieuvre equations~\eqref{eq:d_lelieuvre} are compatible, and generate generalized $K$-surface(s) $r:Q^\epsilon \to \mathbb{R}^3$.
\label{lem:construct}
\end{lemma} 
\begin{proof}
{We have, $\langle N_0 -N_{12},N_0 + N_1 + N_2 + N_{12} \rangle = |N_0|^2-|N_{12}|^2 + \langle N_0,N_1 + N_2 \rangle  -\langle N_{12}, N_1 + N_2  \rangle = 0$ so $N_0+N_{12}$ and $N_1+N_2$ are both perpendicular to $N_0 - N_{12}$. A similar calculation shows that $N_0+N_{12}$ and $N_1+N_2$ are also perpendicular to $N_1 - N_{2}$. 

Finally, $\langle N_0-N_{12},N_1-N_2 \rangle = \cos \delta - \cos \delta -\cos \delta +\cos \delta = 0$ so $N_1-N_2$ and $N_0-N_{12}$ are not parallel since neither is zero. This implies that $N_0+N_{12}$ and $N_1 + N_2$ are parallel and thus satisfy the compatibility condition~\eqref{eq:d_compatible}. We can therefore ``integrate" the discrete Lelieuvre equations along any path in $Q^\epsilon$, starting from a designated `origin' $o$. Since $Q^\epsilon$ is simply connected, we can find a path from $o$ to every other vertex, and summing~\eqref{eq:d_lelieuvre} over the path gives a consistent definition of $r:Q^\epsilon \to \mathbb{R}^3$. This gives a 3 parameter family of generalized $K$-surfaces determined by the initial (arbitrary) choice of $r(o) \in \mathbb{R}^3$. 

In general, this mapping can be ramified \cite{wissler1972}, but imposing condition~\eqref{eq:unramified} ensures that $r$ is not multi-sheeted, in contrast to $N$. In particular this condition forces $\sum \alpha_k = 2 \pi J_p = 2\pi (1-m_p)$  at all interior vertices, giving consistency with Lemma \ref{lem:covering-number}.}
\end{proof}

The problem of constructing discrete PS-fronts therefore reduces to the problem of constructing spherical Chebyshev nets on Asymptotic complexes. To adapt the assembly and surgery procedures defined for continuous surfaces to the discrete setting we define

\begin{definition}[A corner vertex] A vertex $q$ in an asymptotic complex $Q^\epsilon$ is a corner vertex if a $u$-edge as well as a $v$-edge incident on $q$ are contained in the boundary $\partial Q^\epsilon$.
\label{def:corner}
\end{definition}

\begin{definition}[Boundary segments] A boundary segment  is a curve $\gamma = E_1 \cup E_2 \cdots \cup E_m \subset \partial Q^\epsilon$, where the edges overlap, $E_i \cap E_{i+1} \neq \emptyset$, and are all either $u$ or $v$ edges. 
\end{definition}

\begin{lemma}
A boundary segment $\gamma$ is incident on a corner vertex $q \in Q^\epsilon$ if and only if $q \in \partial \gamma$. Conversely, every corner vertex $q$ determines two maximal boundary segments, $\gamma_u$ consisting of $u$-edges and $\gamma_v$ consisting of $v$-edges.
\label{lem:segments}
\end{lemma}

\begin{proof} Since $Q^\epsilon$ is simply connected  and embeddable in $\mathbb{R}^2$, (see Remark~\ref{remark:simply_connected}) it follows that $\partial Q^\epsilon$ is a Jordan curve consisting of $u$ and $v$ segments.  Definition~\ref{def:corner} implies that $q$ is on one $u$ and one $v$ edge contained in the boundary, a Jordan curve, so $q$ is not on any other edge contained in the boundary. The lemma immediately follows.
\end{proof}

\begin{lemma}[Discrete assembly] Let $2m \geq 4$ be even and $\gamma_i, i=1,2,\ldots 2m$ be mappings $\gamma_i :\{0,1,\ldots,l_i\} \to S^2$ such that (i)  $\gamma_i(0) = (0,0,1)$ for all $i$, (ii) $\alpha_i = \angle \gamma_{i-1}(1) \gamma(0) \gamma_{i}(1)  \in (-\pi,0)$ and $\sum_{i=1}^{2m} \alpha_i = 2 \pi(1-m)$ (here $\gamma_0 = \gamma_{2m}$), and (iii) $\langle\gamma_i(k),\gamma_i(k+1) \rangle = \cos \delta$ for all admissible $i,k$.  This data uniquely determines a maximal asymptotic complex $Q$, a spherical Chebyshev net $N:Q \to S^2$ and an unramified $K$-surface $r:Q \to \mathbb{R}^3$.
\label{lem:discrete_assembly}
\end{lemma}
\begin{proof}
Let $J_i = \{0,1,\ldots,l_{i-1}\} \times \{0,1,\ldots,l_i\}, i=1,2,\ldots,2m$ be $2m$ rectangular domains. We will set $l_0 = l_{2m},J_0 = J_{2m}$.  Defining the attaching maps $\chi_i :(k,0) \in J_i \mapsto (l_{i-1},k) \in J_{i-1}$, we obtain a (discrete) asymptotic complex $Q$. On each of the sets $J_i$, we define $N_i(k,0) = \gamma_{i-1}(k), N_i(0,l) = \gamma_i(l)$  and extend $N_i$ to the rectangle $J_i$ using \eqref{eq:householder}. By construction, the definitions agree along the overlaps, so we can use the attaching maps  to obtain a spherical Chebyshev net $N:Q \to S^2$. Taking the edges $\gamma_i$ with $i$ even as the $u$-edges and $i$ odd as the $v$ edges, we can consistently extended the definition of $u$ and $v$ edges on every rectangle $J_i$. The result now follows from Lemma~\ref{lem:construct}.
\end{proof}

\begin{lemma}[Discrete surgery] Let $Q$ be an asymptotic complex and  $N:Q \to S^2$ be a spherical Chebyshev net (in particular, all vertices satisfy \eqref{eq:unramified}). Given $l_1,l_2 \in \mathbb{N}$ and $q$, a corner vertex for $Q$  we can define an asymptotic complex $Q' \supset Q$ by attaching 3 rectangular domains $J_i,i=1,2,3$ to $Q$ and extending the spherical Chebyshev net $N$ to $N':Q' \to S^2$ such that $q \notin \partial Q'$ and the associated $K$-surface is unramified, i.e. single-sheeted. 
\end{lemma}
\begin{proof} We set $\gamma_0 = \gamma_u$ and $\gamma_3 = \gamma_v$ where $\gamma_u$ and $\gamma_v$ are the boundary $u$ and $v$ boundary segments incident on $q$ whose existence is given by Lemma~\ref{lem:segments}. $q$ satisfies~\eqref{eq:unramified} and this defines $\delta \in (0, 3\pi)$. Since $q$ is a corner vertex, $d_p \geq 2$ is even. Let $\alpha = -\pi + \delta/3 \in (-\pi,0)$. Determine $\gamma_1(1)$ by $\angle \gamma_{0}(1) N(q) \gamma_{1}(1) = \alpha$ and $\gamma_2(1)$ by $\angle \gamma_{1}(1) N(q) \gamma_{2}(1) = \alpha$. Now we set $\gamma_i(k) = N(q) \cos(k \delta)  + \frac{\gamma_i(1)- \cos \delta N(q)}{\sin \delta} \sin(k \delta), k=1,2,\ldots,l_i$ corresponding to equally spaced points on geodesics on the sphere. An argument identical to the proof of Lemma~\ref{lem:discrete_assembly} gives the desired result.

Note that, by adding three spherical sectors with angle $\alpha$ at the boundary point $q$, we ensure that \eqref{eq:unramified} is satisfied at $q$, and of course, we have not introduced further branch points, or modified the solution at existing branch points away from $q$.
 \end{proof}

\begin{remark}
We henceforth consider the discrete mappings $N:Q^\epsilon \to S^2, r:Q^\epsilon \to \mathbb{R}^3$ as our objects of interest. It is also possible to treat them as discrete approximations of the continuous mappings considered in \S\ref{sec:branchedsurfaces}. With finitely many, isolated, branch points the passage to the continuous limit upon refinement of the quadmesh $Q^\epsilon$ follows from a straightforward application of standard arguments that are outlined in~\cite[\S 5.5]{bobenko2008bdiscrete}, applied to one asymptotic rectangle at a time. As a ``fully discrete" alternative we can also build approximations to the branched surface using hyperboloid surface patches since our quadmeshes are checkerboard colorable~\cite{Huhnen-Venedey2014Discretization}.
\end{remark}

\subsection{DDG on the Poincar\'{e} Disk}
\label{sec:poincare-ddg}

Thus far, we have constructed branched pseudospherical surfaces as $K$-surfaces, i.e. mappings $r:Q^\epsilon \to \mathbb{R}^3$ from asymptotic coordinates into $\mathbb{R}^3$. 
However, the primary object of interest in elasticity is the deformation $y:\Omega \to \mathbb{R}^3$, the mapping from the Lagrangian (material) domain $\Omega$ to the Eulerian (lab) frame  $\mathbb{R}^3$. To construct this mapping, we need also to compute the transformation $\zeta: Q^\epsilon \to \Omega$ that allows us to identify the material location corresponding to a point with given asymptotic coordinates so that $y = r \circ \zeta^{-1}$. To this end, we start with a coordinatization of $\Omega$.

Since our interest is in pseudospherical surfaces, we have $\Omega \subset \mathbb{H}^2$, and we can identify $\mathbb{H}^2$ with the Poincar\'{e} disk $(\mathbb{D},g)$ given by $\mathbb{D}  = \{z \, |\, |z| < 1\}$, the unit disk, and $\displaystyle{g = \frac{4 dz d\bar{z}}{1-|z|^2}}$ \cite[Chap. 4]{anderson2005hyperbolic}. $z$ is our Lagrangian or reference coordinate, since it labels material points {\em independently of their particular locations} in $\mathbb{R}^3$, i.e. independent of  the deformation $y:\Omega \to \mathbb{R}^3$. We record a few standard facts about the Poincar\'{e} disk model for $\mathbb{H}^2$:
\begin{enumerate}[leftmargin=*,label={(\Alph*)}]    
    \item The distance between two points $z_1,z_2 \in \mathbb{D}$ is given by 
    \begin{equation*}
d_{\mathbb{H}^2}(z_1, z_2) = \arccosh\left(1 + 2\frac{|z_1 - z_2|^2}{(1 - |z_1|^2)(1 - |z_2|^2)} \right),
\end{equation*}
In particular, if one of the points is the origin, this expression reduces to
\begin{equation}
\label{eq:hyp_radius}
d_{\mathbb{H}^2}(0, z) =2\arctanh(|z|).
\end{equation}
\item The orientation preserving isometries of $\mathbb{H}^2$ are given by (a subgroup of) the M\"{o}bius transformations
\begin{equation}
\label{eq:mobius}
f(z; z_0,\theta) = e^{i\theta} \frac{z + z_0}{1 + z\bar{z_0}}
\end{equation} 
where $|z_0| < 1, \theta \in [0,2\pi)$. For our purposes, it suffices to take $\theta = 0$ and we shall henceforth drop this variable and use $f(z; z_0) = \frac{z + z_0}{1 + z\bar{z_0}}$. It is straightforward to check that $f'(0;z_0)=1-|z_0|^2$ is real and positive, and $f^{-1}(w;z_0) = f(w;-z_0) = \frac{w - z_0}{1 - w\bar{z_0}}$.
\item \label{geodesics}Equally spaced points on the geodesics through $z=0$, are given by $\gamma_\beta(n) = e^{i \beta} \tanh\left(\frac{n \Delta}{2}\right)$, where $\Delta$ is the separation between successive points on the geodesic. Likewise, geodesics through a point $z_0$ are given by $z_n = f(\gamma_\beta(n);z_0)$.
\end{enumerate}

As we argued above, constructing the appropriate DDG for $K=-1$ surfaces is equivalent to constructing discrete Chebyshev nets, i.e. rhombic quadrilaterals in the appropriate space. Constructing such rhombi on $ S^2$, as in~\eqref{eq:householder}, gives us DDG for the Gauss Normal map. As we now show, the same idea also applies to the problem of finding the (discrete) mapping $\zeta:Q^\epsilon \to \Omega \subset \mathbb{H}^2$. Given $\zeta_0, \zeta_1$ and $\zeta_2$ with $d_{\mathbb{H}^2}(\zeta_0,\zeta_1) = d_{\mathbb{H}^2}(\zeta_0,\zeta_2)= 2 \tanh\left(\frac{\Delta}{2}\right)$, we can apply the isometry $f(.,-\zeta_0)$ to these points and obtain 
$$
w_j = f(\zeta_j,-\zeta_0), \quad w_0 =0, \quad w_1 =  \frac{\Delta}{2} e^{i \beta_1}, \quad w_2 =  \frac{\Delta}{2} e^{i \beta_2}.
$$
The fourth vertex $w_{12}$ of the ``normalized" rhombus diagonally across from the vertex $w_0$ at the origin, can be determine by a straightforward computation after setting $d_{\mathbb{H}^2}(w_{12},w_1) = d_{\mathbb{H}^2}(w_{12},w_2)= 2 \tanh\left(\frac{\Delta}{2}\right)$. $\zeta_{12}$ is then obtained by applying the inverse mapping $f(.,\zeta_0)$. Putting everything together, we have
\begin{align}
w_j & = f(\zeta_j,-\zeta_0) \quad i=0,1,2\nonumber \\
w_{12} & = \frac{w_1+w_2}{1+|w_1w_2|}, \nonumber \\
    \zeta_{12} & = f(w_{12},\zeta_0)
\label{poincare-ddg}
\end{align}

\begin{figure}[htbp]
        \begin{subfigure}[t]{0.5\textwidth}
                \centering
                {\includegraphics[trim={0cm, -0.5cm, 0, 0cm}, clip, width=.85\linewidth, height=6.5cm]{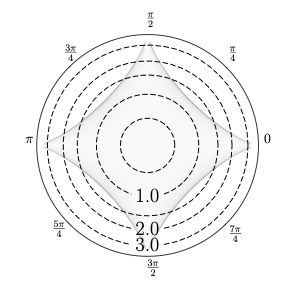}}
                \caption{}
                \label{fig:poincaregeodesics}
        \end{subfigure}%
        \begin{subfigure}[t]{0.5\textwidth}
                \centering
                {\includegraphics[trim={1cm, 0.75cm, 0, 0.5cm}, clip, width=.85\linewidth, height=6cm]{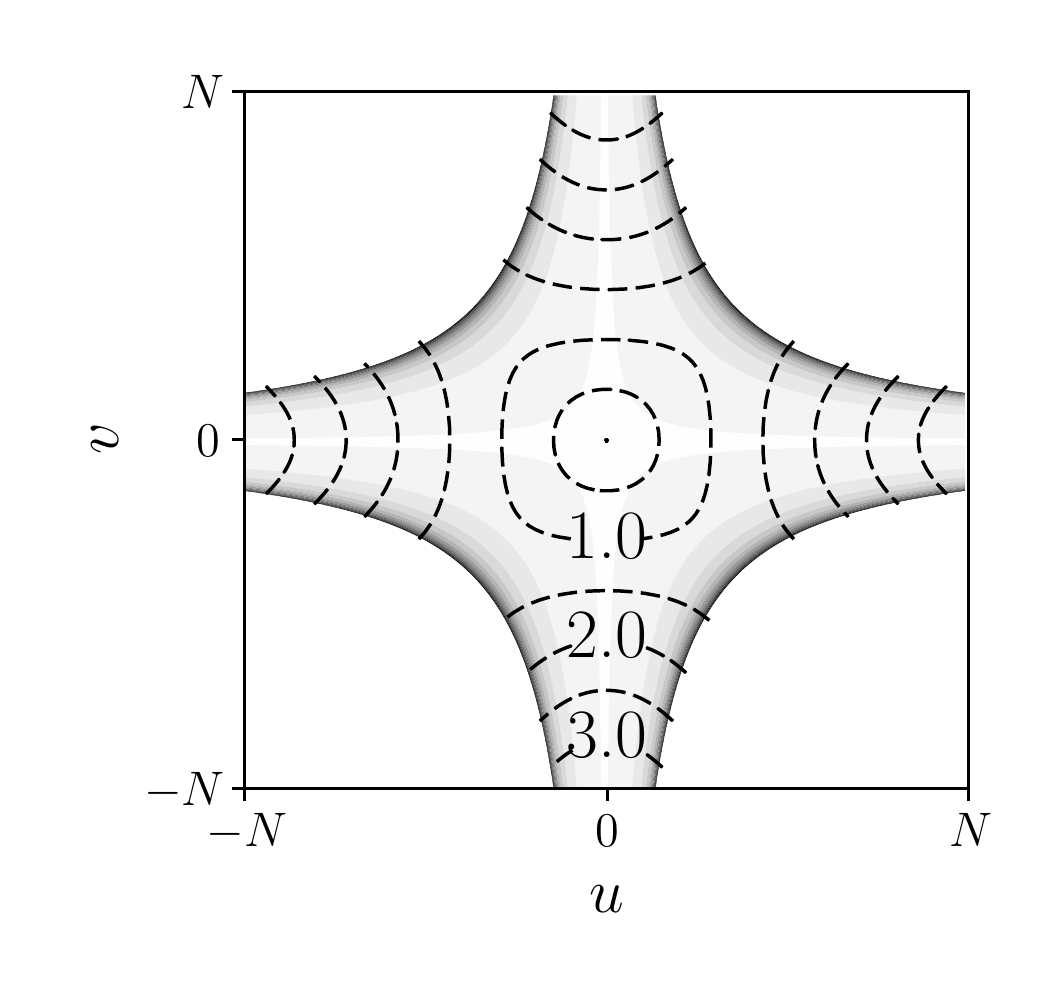}}
                \caption{}
                \label{fig:amslergeodesics}
        \end{subfigure}%
        \caption{Examples of a piece of an Amsler surface represented in (a) asymptotic coordinates $(u, v)$ and (b) in the Poincar\'{e} disk $z$, up to the singular edge, colored by the angle $\phi$ and contoured by geodesic radius with labels.}
        \label{fig:geodesicballs}
\end{figure}

It is now straightforward to
construct (branched) Chebyshev nets in $\mathbb{H}^2$ that inherit their topology from a 
given asymptotic complex. 
More formally, a discrete hyperbolic Chebyshev net is a  quadgraph in $\mathbb{H}^2$ with 
an assignment of $u$ and $v$ labels to the edges such that each face (quad) has two $u$ and two $v$ edges which alternate, and satisfying~\eqref{poincare-ddg} on each quad, where $\zeta_0$ and $\zeta_{12}$ are one set of non-adjacent vertices, and $\zeta_1,\zeta_2$ are the vertices on the other diagonal. A branch point is any interior vertex with degree $2m \geq 6$. From the Chebyshev net in $\mathbb{H}^2$, we can immediately construct the corresponding $K$-surface (discretized surface) in $\mathbb{R}^3$ by requiring that each star (the edges incident on a vertex $r_{j,k}$) be planar with lengths and angles given by the Chebyshev net at the vertex $\zeta_{j,k}$, i.e. the mapping between the Poincar\'e disk and $\mathbb{R}^3$ is a discrete conformal map at each vertex. This mapping between the Poincar\'e disk and $\mathbb{R}^3$ is the desired Lagrangian to Eulerian map. Although differing in details, the idea of conformally mapping the Hyperbolic plane into $\mathbb{R}^3$ has been used to investigate the wavy edges of leaves \cite{Nechaev2017From,nechaev2001plant}, and for energetic and geometric approaches to studying buckling in hyperbolic elastic surfaces \cite{nechaev2015buckling}.

As an illustration of our approach, we construct a discrete hyperbolic Chebyshev net corresponding to an Amsler surface with an angle $\varphi =\pi/2$ between the straight asymptotic lines where they intersect. Since these asymptotic lines are also geodesics in $\mathbb{R}^3$, the same is true for the corresponding curves in the Poincar\'e disk. We pick the origin $z=0$ to correspond to this point of intersection. If the rhombi have a side-length $\Delta$ it follows that the `Amsler-type' boundary data on the Poincar\'e disk are given by $\zeta_{j,0} = \tanh\left(\frac{j \Delta}{2}\right), \zeta_{0,k} = i \tanh\left(\frac{k \Delta}{2}\right)$. We then solve for $\zeta_{j,k}$ with $j\neq 0, k\neq 0$ using~\eqref{poincare-ddg}. The (discretized) angle between the asymptotic lines at node $j,k$ is given by 
\begin{equation}
\varphi_{j,k} = \mathrm{arg}(w_2w_1^*),
\label{eq:angle}
\end{equation}
where the $w_j$ are determined by~\eqref{poincare-ddg} with $\zeta_0 = \zeta_{j,k}, \zeta_1 = \zeta_{(j+1),k}, \zeta_2 = \zeta_{j,(k+1)}$. 

The results are displayed in Figure \ref{fig:geodesicballs}. Fig.~\ref{fig:poincaregeodesics} shows the hyperbolic Chebyshev net $\zeta_{j,k}$ where each node is colored by the angle $\varphi_{j,k}$ up to the contour $\varphi= \pi$ corresponding to the singular edge. The dashed curves are the boundaries of geodesic disks, labelled by radius. It is clear that the Amsler surface with angle $\pi/2$ allows us to smoothly embed a geodesic disks of radius 1 into $\mathbb{R}^3$ but not a disk of radius 1.5 \cite{gemmer2011shape}. Fig.~\ref{fig:amslergeodesics} displays the same information in terms of the discrete indices $j,k$ which are proxies for the asymptotic coordinates $u$ and $v$. Since the geodesic distance to the origin is easily computed in the Poincar\'e disk by~\eqref{eq:hyp_radius}, we have an efficient method to determine geodesic radii on pseudospherical surfaces without explictly integrating the arclength \cite{gemmer2011shape} or solving an eikonal equation on the surface. Fig.~\ref{fig:amslerspace} shows the corresponding $K$-surface in $\mathbb{R}^3$, a discretization of the Amsler surface with angle $\frac{\pi}{2}$ between the generators. Multiple singular edges are discernible by their characteristic cuspidal form ({\em cf.} Fig.~\ref{fig:bobbin}). 

\begin{figure}[htbp]
                \centering
                {\includegraphics[trim={1cm, 4cm, 1cm, 4cm}, clip,width=.6\linewidth]{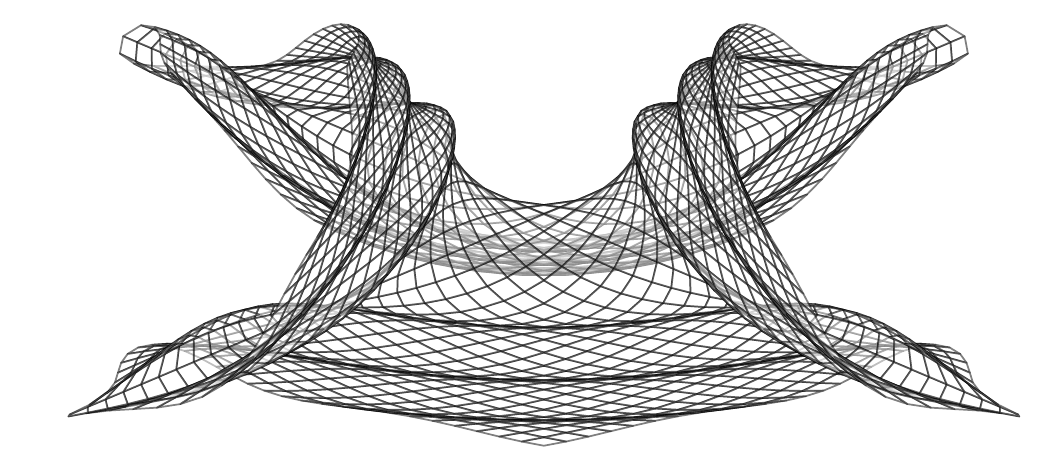}}
                \caption{The Amsler surface with $\varphi_0 = \frac{\pi}{2}$.}
                \label{fig:amslerspace}
\end{figure}

The last notion  we need to introduce 
is that of a {\em reversal}. We know that, in general, a pseudospherical parametrization $r(u,v)$ does not correspond to an immersed surface, and the failure of (local) injectivity is associated with the locus where $\partial_u r \times \partial_v r = 0$. The notion of reversal captures this idea in a discrete setting. Let $\omega$ be an orientation on $\mathbb{H}^2$. If $\zeta_{j,k}$ is a regular point, it is incident on 4 quads given by $\zeta_{j+p,k+q}$, where $p,q \in \{-1,0,1\}$. We say that there is a reversal at $\zeta_{j,k}$ if 
\begin{equation}
\prod_{p \in \{-1,1\},q \in \{-1,1\}} \omega(f(\zeta_{j+p,k},-\zeta_{j,k}),f(\zeta_{j,k+q},-\zeta_{j,k})) \leq 0.
\label{reversal}
\end{equation}
This condition  is invariant under M\"obius transformations and also under reversal of the orientation $\omega \to - \omega$ being a product of 4 terms. The import of this condition is that, at a reversal one of the quads that is incident on $\zeta_{j,k}$ is flipped relative to the other three, 
so the Chebyshev net is folding over itself. The Amsler surface in Fig.~\ref{fig:amslerspace} corresponds to three reversals of the associated hyperbolic Chebyshev net.

\begin{figure}[ht]
        \begin{subfigure}[t]{0.33\textwidth}
                \centering
                \includegraphics[trim={3.5cm 0 3.5cm 0}, clip, width=.85\linewidth]{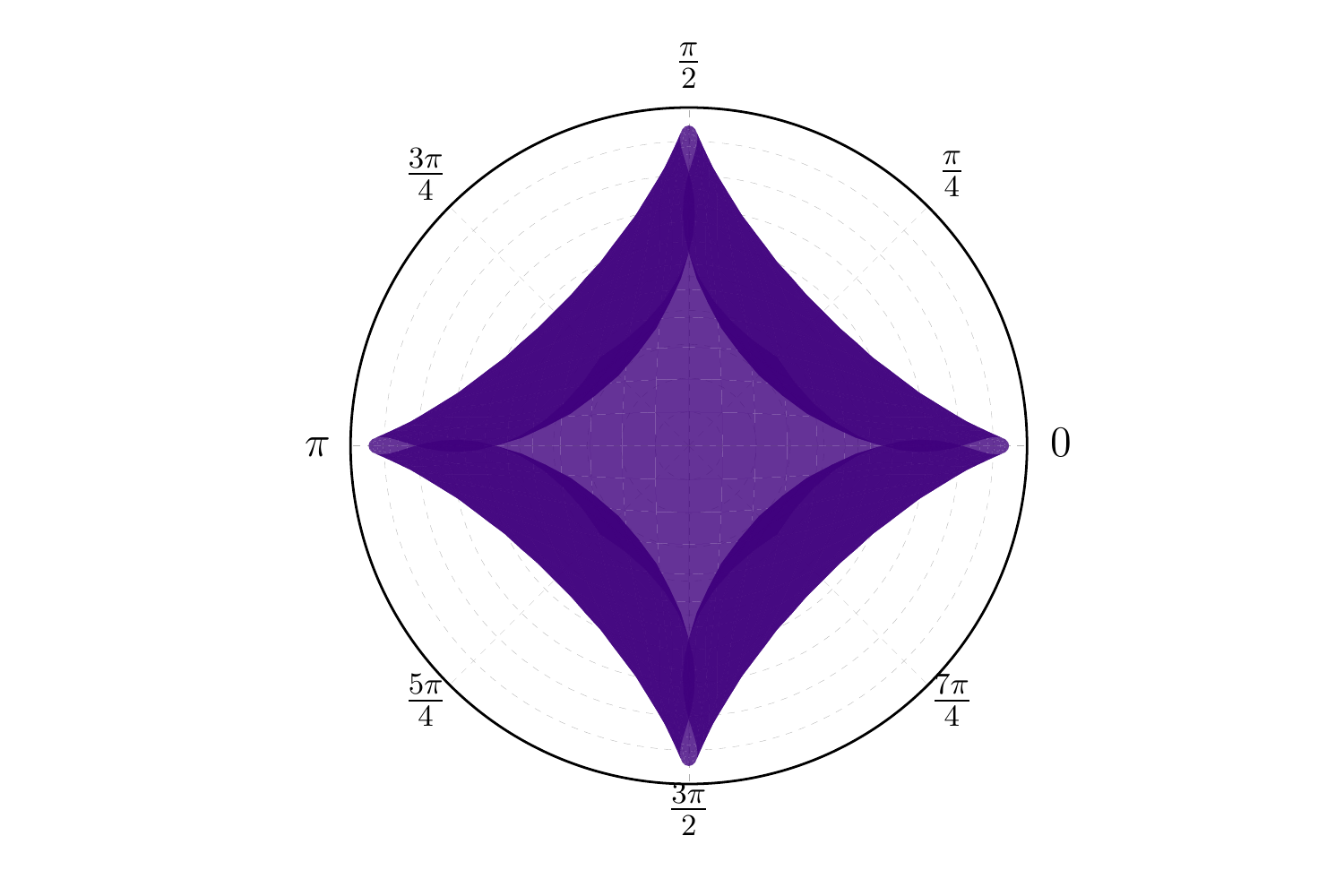}
                \caption{}
                \label{fig:poincaresmooth}
        \end{subfigure}%
        \begin{subfigure}[t]{0.33\textwidth}
                \centering
                \includegraphics[trim={3.5cm 0 3.5cm 0}, clip, width=.85\linewidth]{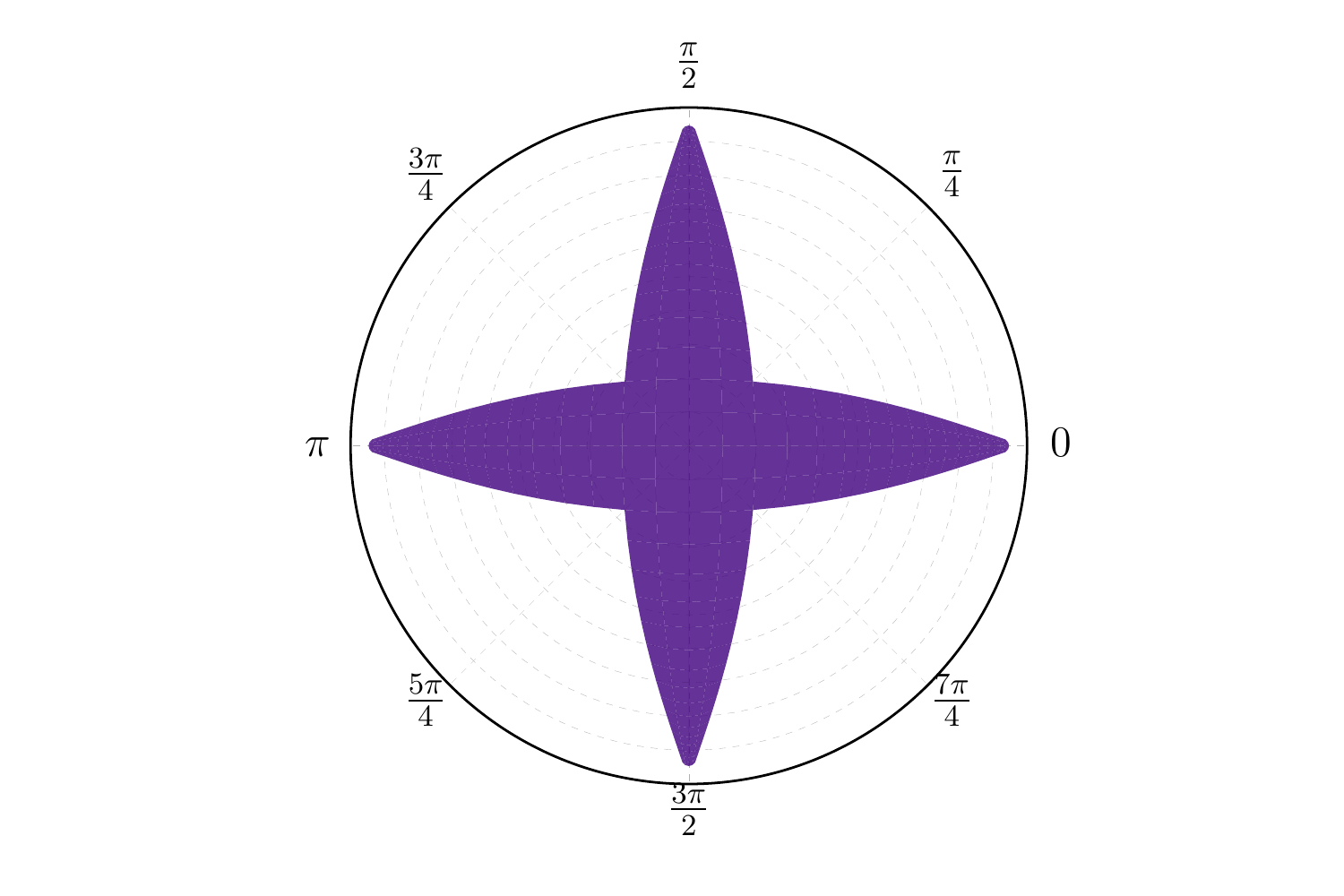}
                \caption{}
                \label{fig:poincarecropped}
        \end{subfigure}
        \begin{subfigure}[t]{0.33\textwidth}
                \centering
                \includegraphics[trim={3.5cm 0 3.5cm 0}, clip, width=.85\linewidth]{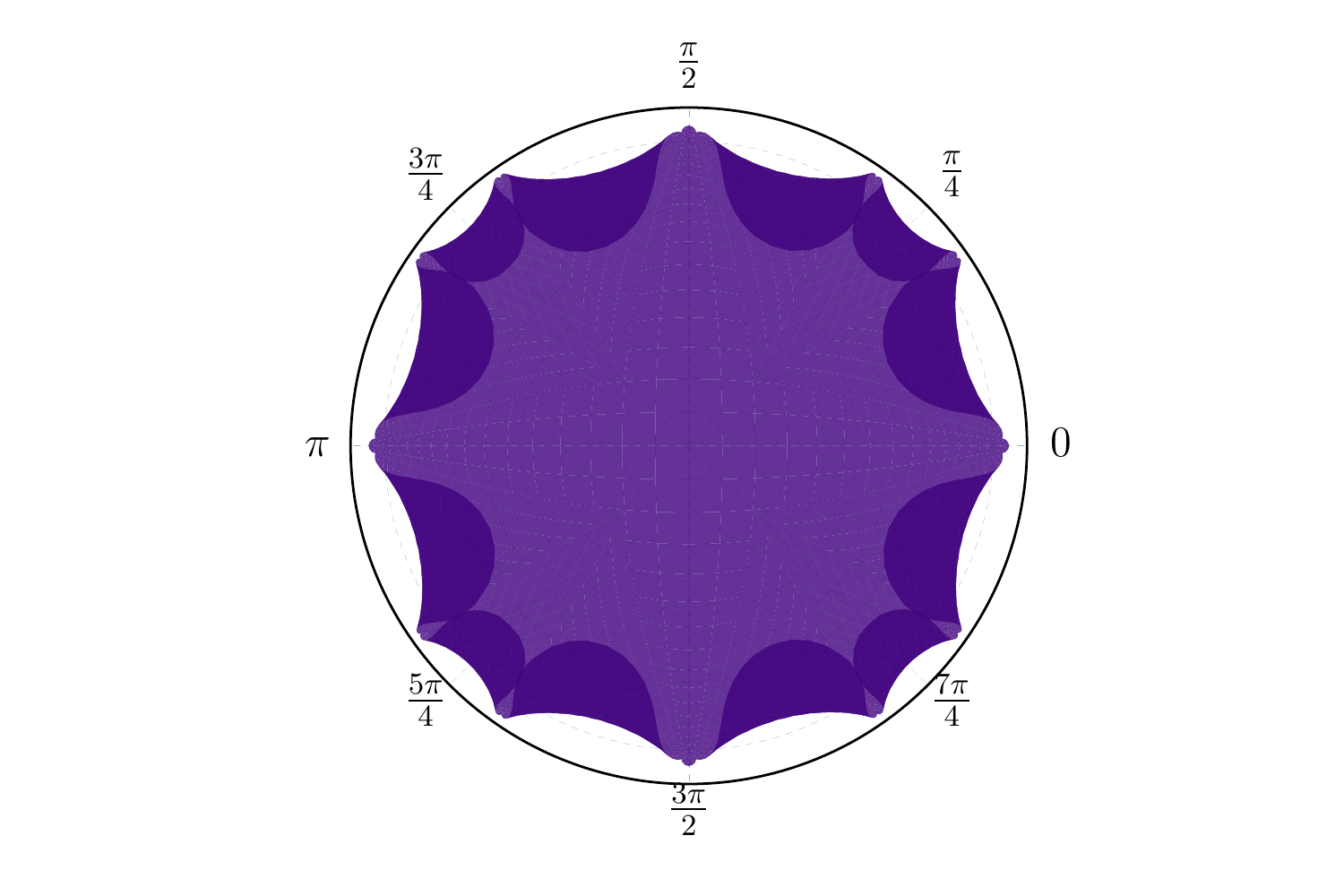}
                \caption{}
                \label{fig:poincareglued}
        \end{subfigure}        
        \caption{Introduction of branch points into the Poincar\'{e} disk via the surgical process. In (a) we see a smooth immersion, the singular edge inhibiting the ability to immerse a large portion of $\mathbb{H}^2$; (b) a cropped version and finally (c) the glued $C^{1,1}$ Poincar\'{e} disk. Overlapping ``sheets'' of the immersion appear significantly darker and provide a signature for the singular edge.}
        \label{fig:poincaresurgery}
\end{figure}

Fig.~\ref{fig:poincaresmooth} shows the discrete hyperbolic Chebyshev net for the Amsler surface with angle $\pi/2$ `extended' beyond the singular edge, where the Chebyshev net $\zeta_{j,k}$ appears to fold back upon itself, as expected. This is evident in Figure \ref{fig:poincaresmooth}. The rhombi in the Chebyshev net are colored with an opacity of eighty percent. As a result, overlapping ``sheets'' of the immersion appear significantly darker. 
Since our procedure gives a (discrete) isometry from the hyperbolic Chebyshev net to the corresponding $K$-surface in $\mathbb{R}^3$, a reversal in the hyperbolic Chebyshev net indicates that $\delta_u \zeta \equiv \zeta_{j+1,k}-\zeta_{j,k}$ and $\delta_v \zeta \equiv \zeta_{j,k+1}-\zeta_{j,k}$ have passed through collinearity. This corresponds to the angle $\varphi$ between the asymptotic curves becoming $0$ or $\pi$, indicating the occurrence of a singular edge.
%

\begin{figure}[htbp]
		 \begin{subfigure}[t]{0.45\textwidth}
              \includegraphics[width=.95\linewidth]{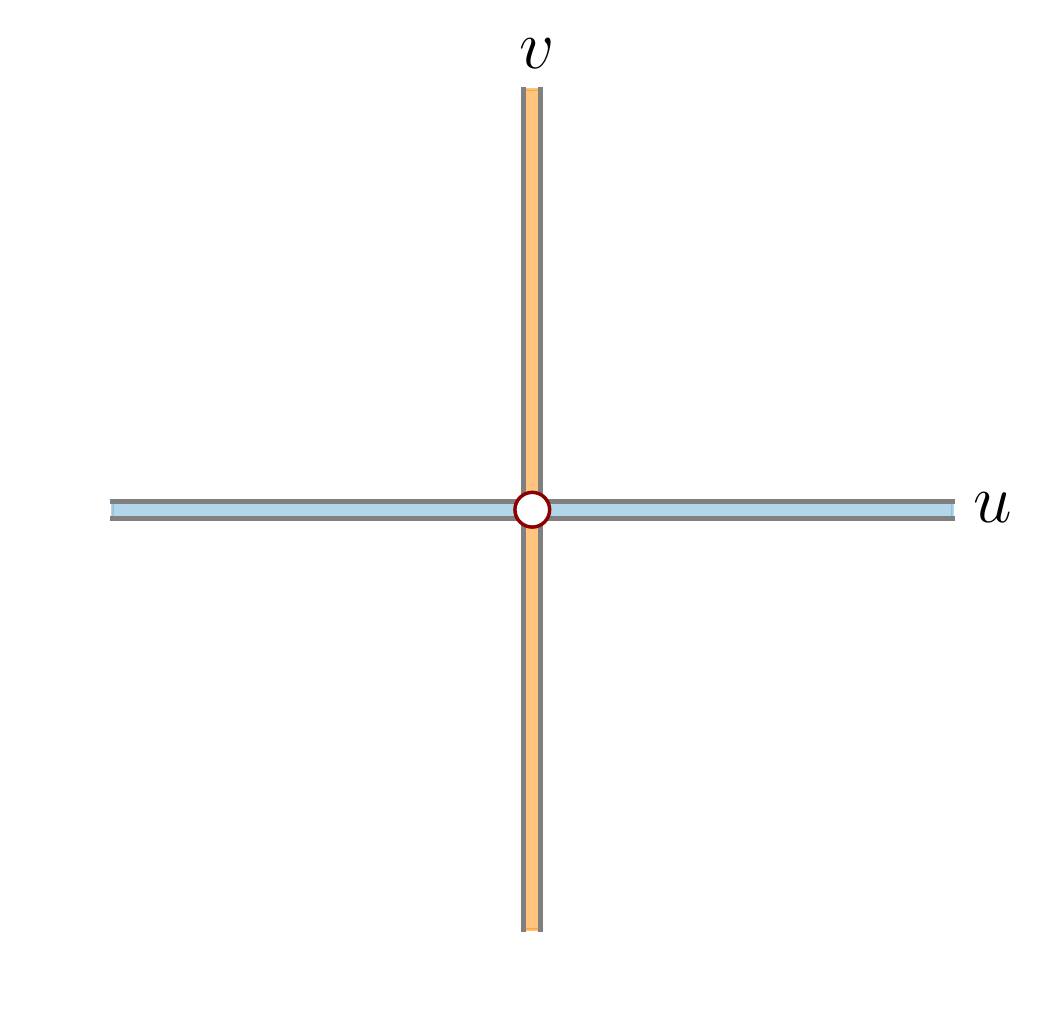}
         \caption{Asymptotic complex at stage $n$}
          \label{fig:TCD_schematic_stage1}
  \end{subfigure}%
        \begin{subfigure}[t]{0.45\textwidth}
                \includegraphics[width=0.95\linewidth]{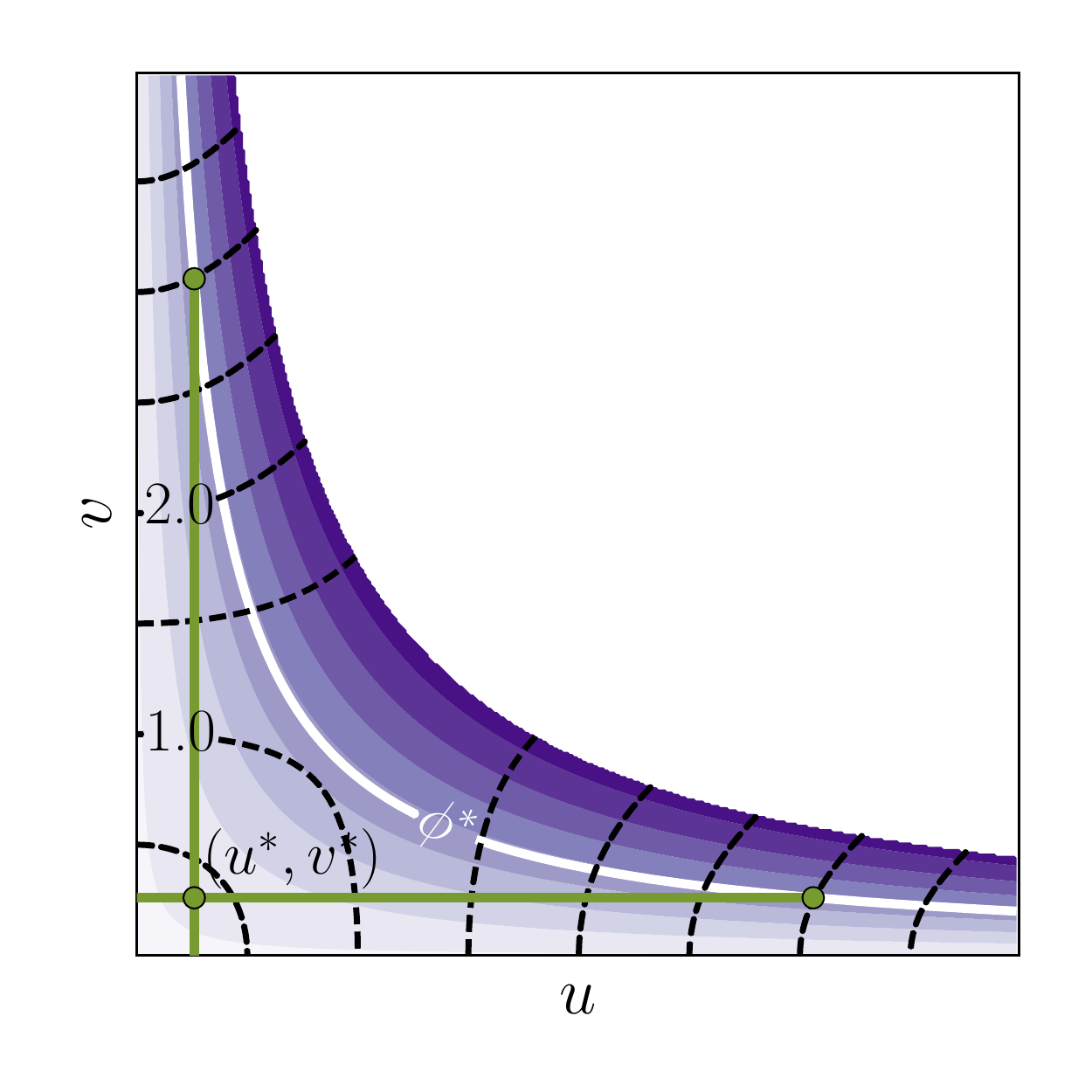}
                \caption{Determining the cut location}
                \label{fig:amslercutschematic}
        \end{subfigure}%
          \label{fig:shapeprocedure1}
  \begin{subfigure}[t]{0.45\textwidth}
          \includegraphics[width=.95\linewidth]{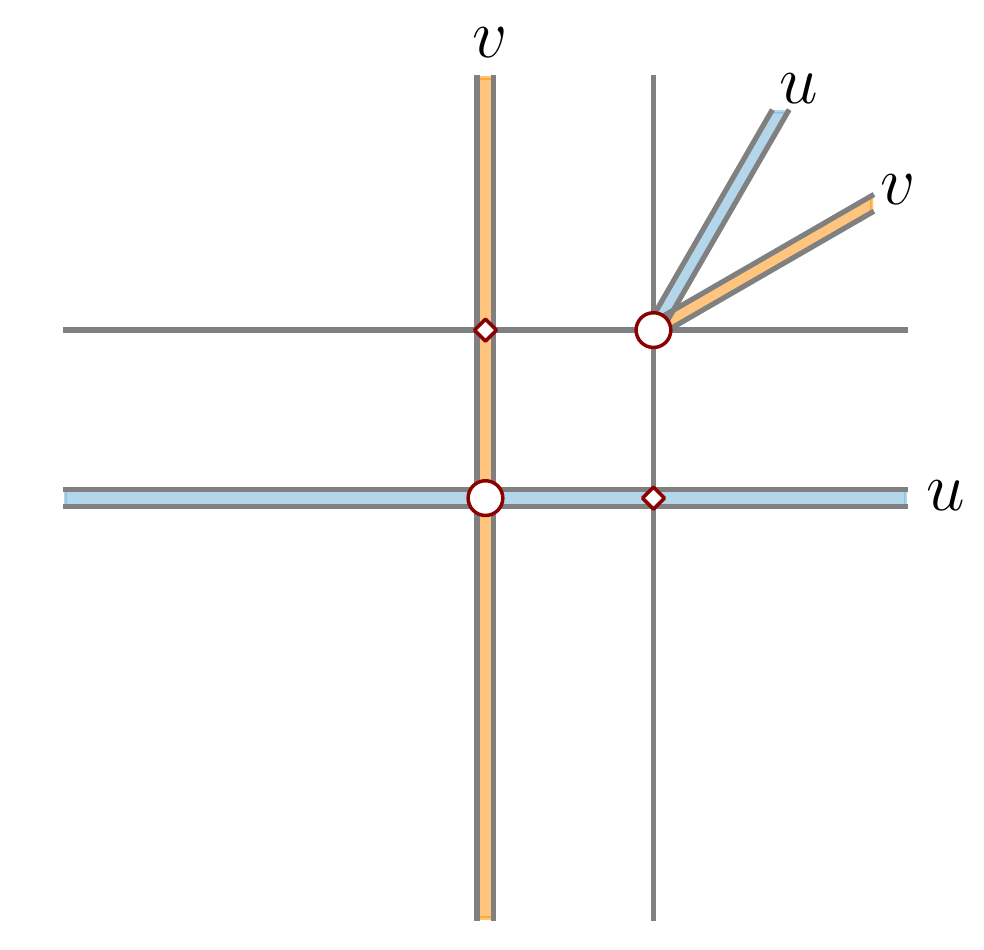}
         \caption{Asymptotic complex at stage $n+1$}
          \label{fig:TCD_schematic_stage2}
  \end{subfigure}%
        \begin{subfigure}[t]{0.45\textwidth}
                \includegraphics[width=0.95\linewidth]{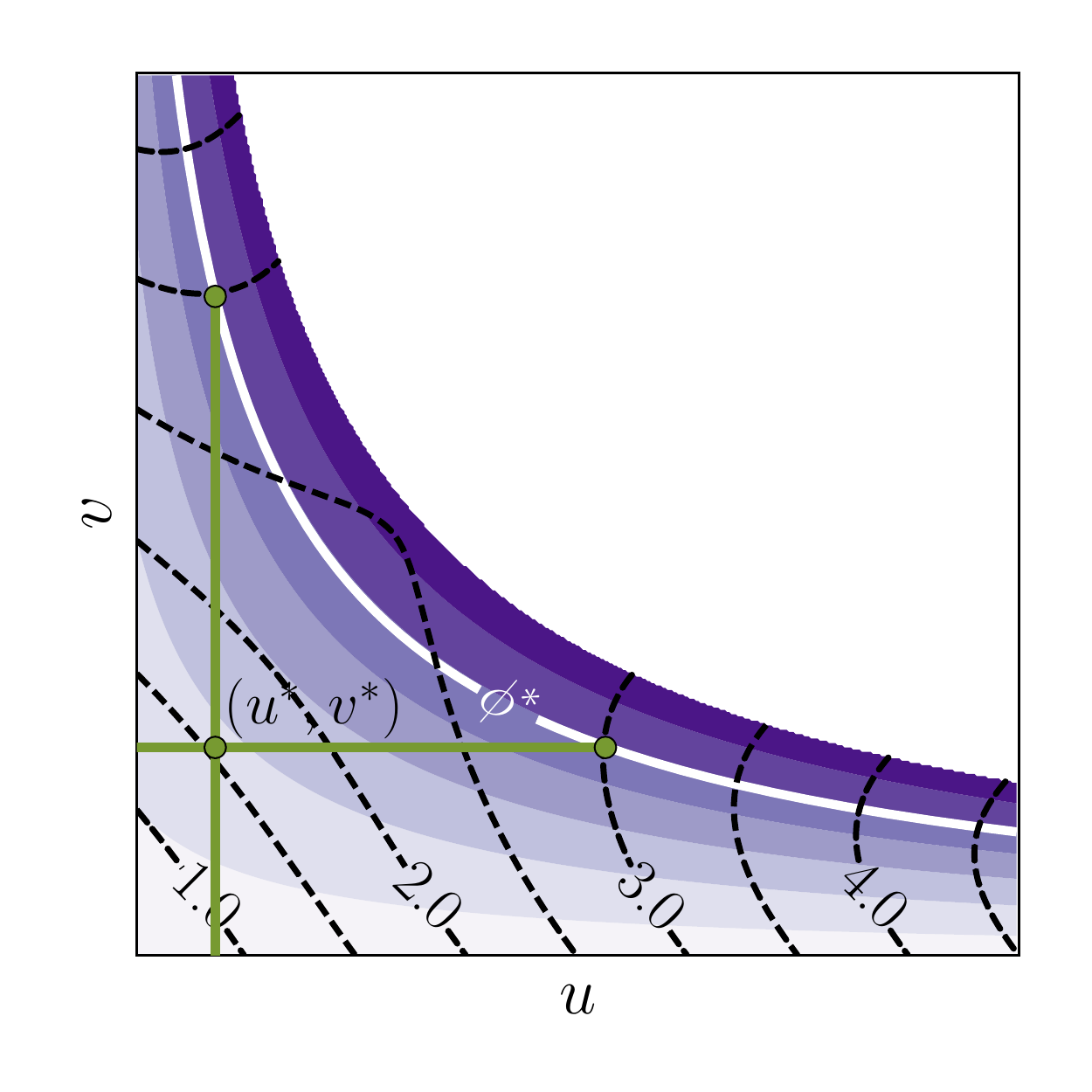}
                \caption{Filling the cut and $n \to n+1$.}
                \label{fig:pseudoamslercutschematic}
        \end{subfigure}%
        \caption{Illustration of Algorithm~\protect{\ref{alg:greedy-cuts}}. 
        Desired $R=3$. (a) The initial asymptotic curves on which we prescribe Amsler data (equally spaced points on geodesics) (b) Filling in the discrete hyperbolic Chebyshev net and identifying the first cut location $(u^*,v^*)$. (c) Introducing new asymptotic curves from the branch point on which we again prescribe Amsler data. (d) General sector having non-constant $\varphi$ (non-Amsler data) along the $v$-axis. In (b),(d), the figures are colored by the value of $\varphi$ with black-dashed contours representing geodesic radius, increments of $0.5$. The solid green lines represent the edges of the $L$-cut, and their intersection the location of the branch point, $(u^*, v^*)$.}
        \label{fig:shapeprocedure2}
\end{figure}

In Fig.~\ref{fig:shapeprocedure2} we show the steps for the particular example of starting with an Amsler surface with angle $\pi/2$ and building a (branched) immersion to $R=3$, a radius beyond the initial singular edge. To stave off the singular edge we first pick a threshold angle $\phi^* < \pi$. For the illustration in Fig.~\ref{fig:shapeprocedure2} we take $\phi^* = 3 \pi /4$. We then excise the region $u \geq u^*, v \geq v^*$, where $u^*, v^*$ are determined by the intersection of the geodesic circle with radius $R=3$ with the contour $\varphi(u,v) = \phi^*$. Note that, at this point $\varphi(u^*,v^*)  < \phi^* < \pi$ and $R < 0.5$, so the cut is {\em significantly inside} the singular edge of the initial Amsler surface. 

We now perform {\em surgery} to replace the removed sector by 3 new sectors. This needs the introduction of two more asymptotic curves, indicated in Fig.~\ref{fig:TCD_schematic_stage2}, along which we are free to prescribe data. We prescribe this data in the Poincar\'e disk by picking equally spaced point on the two geodesics through the point $\zeta_0 = \zeta(u^*,v^*)$ obtained by trisecting the angle left behind by the sector that is removed. In more detail, if  $w_1$ and $w_2$ are the ``edges" of the excised sector, moved to the origin by a M\"obius transformation (See~\eqref{poincare-ddg}), we define 
\begin{equation}
\varphi_1 = \mathrm{arg}(w_2 w_1^2)/3, \quad \varphi_2 = \mathrm{arg}(w_2^2 w_1)/3.
\label{eq:trisect}
\end{equation}
 Then, the appropriate geodesics are given by undoing the M\"obius transformation,
\begin{equation}
    \zeta_{0,k} = f\left(e^{i \varphi_1} \tanh\left( \frac{k \Delta}{2}\right) ,\zeta_0\right), \quad \zeta_{j,0} = f\left(e^{i \varphi_2} \tanh \left( \frac{j \Delta}{2}\right) ,\zeta_0\right),
    \label{eq:poincare-surgery}
\end{equation}
where $\Delta$ is the side-length of the rhombi in the Chebyshev net. We can determine $\zeta_{j,k}$ in the interiors of the three new sectors  using~\eqref{poincare-ddg}. We can do this in each of the 4 sectors (quadrants) that constitute the initial Amsler surface and the resulting Chebyshev net in the Poincare disk is illustrated in Fig.~\ref{fig:poincareglued}. 
The result is a discrete Chebyshev net with 4 vertices of degree 6, one in each quadrant, corresponding to the branch points. 
We thus have implemented surgery, as introduced in \S\ref{sec:introducingabranchpoint}, directly in the Poincar\'e disk.

\begin{remark}[Ramification]Unlike for DDG based on spherical Chebyshev nets~\eqref{eq:d_lelieuvre}, where we need condition~\eqref{eq:unramified} to guarantee that the resulting $K$-surface is unramified, DDG based on \eqref{poincare-ddg} gives a discrete conformal map between the net in the Poincare disk and the resulting $K$-surface, so any hyperbolic Chebyshev net where the angles add up to $2 \pi$ at interior nodes, and to less than $2 \pi$ at boundary nodes will give a $K$-surface with no ramification ({\em cf}. Wissler \cite{wissler1972}). In particular, our algorithm~\ref{alg:greedy-cuts} guarantees this. Of course, the normal map is ramified at branch points.
\label{remark:ramification}
\end{remark}

Fig.~\ref{fig:pseudoamslercutschematic} shows one of the resulting sectors in the 2\textsuperscript{nd} generation, i.e. the asymptotic curves defining the boundaries of the sector are incident on the branch point $(u^*,v^*)$ from the first cut. Note that the singular edge again intersects the geodesic circle $R=3$ so we have to repeat the entire procedure to obtain the 2\textsuperscript{nd} generation branch points and 3\textsuperscript{rd} generation sectors. Note also that the new branch point is at $R \approx 1.5$ and thus the first and second generations sectors, taken together, are closer to covering the desired domain $R \leq 3$, and do so while maintaining $\varphi \leq \phi^*$ everywhere.

This {\em surgery} procedure can be repeated recursively to construct branched isometric immersions of arbitrarily large disks. We list the steps involved in Algorithm~\ref{alg:greedy-cuts}. This is a `greedy' algorithm for constructing branched immersions since it is based on picking the cut locations using information local to a particular sector, and attempts to maximize the size of the sector in the current generation, rather than pick the cut location in a more globally optimal fashion. By construction, the algorithm generates distributed branch points in a  recursive and self-similar manner, We discuss this further in~\S\ref{sec:amslerrecursion}, where we estimate the number of recursion steps needed before the algorithm terminates when applied to a Pseudospherical disk with (geodesic) radius $R$. 
\begin{algorithm}
\caption{A greedy algorithm for building large branched surfaces recursively.  \label{algorithm1}}
\begin{algorithmic}[1]

\State \textbf{Parameters:}  $R \gets $ radius of disk to be embedded,  $M \gets 2m \geq 4$ number of initial sectors, $\phi^* \in (\pi/m, \pi) \gets$ cutoff angle, $\Delta \gets$ discretization size, $N \gets \lceil 2R/\Delta \rceil$. 
\State \textbf{Initialize:} List of  {\bf Sectors} $= \{\Omega_1,\Omega_2,\ldots,\Omega_{M}\}$, Each sector $\Omega_n = \emptyset$. 
\For{$n \in \{1,2,\ldots, M\}$}
\State $\zeta^n_{j,0} \gets e^{i \pi (n-1)/m} \tanh\left(\frac{j \Delta}{2}\right), \zeta^n_{0,k} \gets e^{i \pi n/m} \tanh\left(\frac{k \Delta}{2}\right)$ for $j,k=0,1,2,\ldots,N$.
\State Determine $\zeta^n_{j,k}$ recursively for $1 \leq j,k \leq n$ from~\eqref{poincare-ddg}.
 \State  Discard $\zeta^n_{j,k}$ if both $\zeta^n_{j-1,k}$ and $\zeta^n_{j,k-1}$ are outside the geodesic disk of radius $R$.
\State $\Omega_n \gets \{\zeta^n_{j,k} \, \, \mbox{ all } j,k \leq N \mbox{ not discarded}\}$.
\State Determine $\varphi^n_{j,k}$ using~\eqref{eq:angle} at nodes where $\zeta^n_{j+1,k}$ and  $\zeta^n_{j,k+1}$ are in $\Omega_n$.
\EndFor
\Repeat
\State Identify a sector $\Omega_n$ containing points  with $\varphi^n_{j,k} > \phi^*$.
\State $j^* \gets  \max\{j \, | \, \varphi_{\ell,k} \leq \phi^* \, \,\forall \, \zeta_{\ell,k} \in \Omega_n, \ell \leq j\}$.
\State $k^* \gets  \max\{k \, | \, \varphi_{j,\ell} \leq \phi^* \, \, \forall \,\zeta_{j,\ell} \in \Omega_n, \ell \leq k\}$.
\State $\Omega_n \gets  \Omega_n \setminus \{\zeta^n_{j,k} \, \, | \, \, j > j^*, k> k^*\}$.
\State $\mathrm{\bf Sectors} \gets  \mathrm{\bf Sectors} \bigcup \{\Omega_{M+1},\Omega_{M+2},\Omega_{M+3}\}$.
\State  $\zeta^{M+2}_{j,0} , \zeta^{M+2}_{0,k}, 0 \leq j,k \leq N$ are determined using~\eqref{eq:trisect}~and~\eqref{eq:poincare-surgery}.
\State $\zeta^{M+1}_{j,0} \gets \zeta^n_{j+j^*,k^*}, \quad \zeta^{M+1}_{0,k} \gets \zeta^{M+2}_{0,k}, \quad \zeta^{M+3}_{0,k} \gets \zeta^n_{j^*,k+k^*}, \quad \zeta^{M+3}_{j,0} \gets \zeta^{M+2}_{j,0} $.
\For{$p \in \{1,2,3\}$}
\State Determine $\zeta^{M+p}_{j,k}$ recursively using~\eqref{poincare-ddg}.
\State  Discard $\zeta^{M+p}_{j,k}$ if both $\zeta^{M+p}_{j-1,k}$ and $\zeta^{M+p}_{j,k-1}$ are outside the geodesic disk of radius $R$.
\State $\Omega_{M+p} \gets \{\zeta^{M+p}_{j,k} \, \, \mbox{ all } j,k \leq N \mbox{ not discarded}\}$.
\State Determine $\varphi^{M+p}_{j,k}$ using~\eqref{eq:angle} at nodes where $\zeta^{M+p}_{j+1,k}$ and  $\zeta^{M+p}_{j,k+1}$ are in $\Omega_{M+p}$.
\EndFor
\State $M \gets M+3$.
\Until no sector contains points with $\varphi_{j,k} > \phi^*$.
\State $Q \gets $ quadgraph given by the hyperbolic Chebyshev net $\bigcup_n \zeta^n_{j,k}$.
\State Construct an $K$-surface $r:Q \to \mathbb{R}^3$ using the side-lengths and angles given by $\bigcup_n \zeta^n_{j,k}$.
\end{algorithmic}
\label{alg:greedy-cuts}
\end{algorithm}
Fig.~\ref{fig:surgery} illustrates the final step in Algorithm~\ref{alg:greedy-cuts}, showing discrete surfaces constructed from mapping the rhombi in hyperbolic Chebyshev nets to skew rhombi in $\mathbb{R}^3$.

\begin{figure}[htbp]
   \centering
   \begin{subfigure}[t]{0.25\textwidth}
     \centering
     \includegraphics[trim={1.cm, 1cm, 1cm, 5cm}, clip, width=\textwidth]{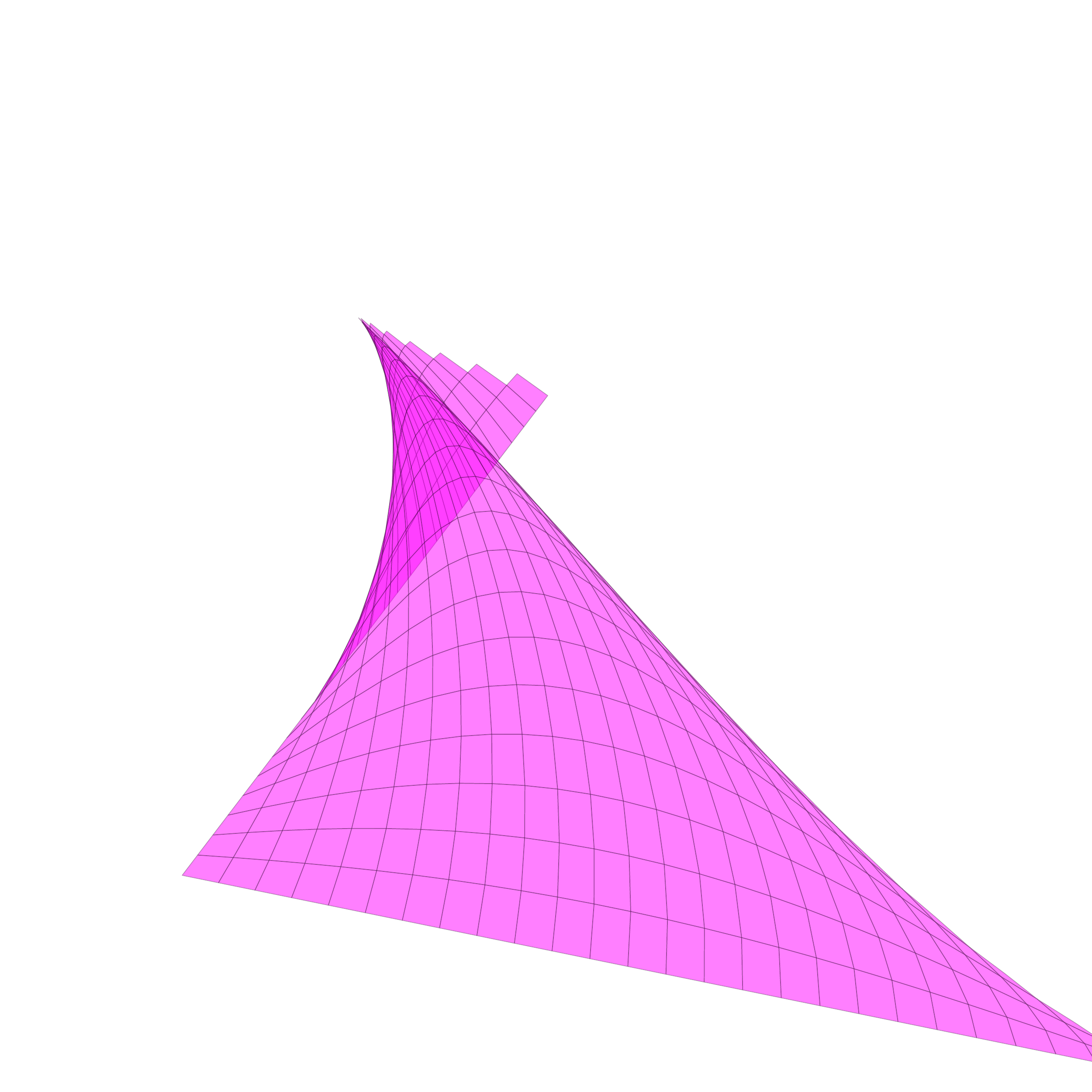}
   \end{subfigure}
   ~
   \begin{subfigure}[t]{0.25\textwidth}
     \centering
     \includegraphics[trim={1.cm, 1cm, 1cm, 5cm}, clip, width=\textwidth]{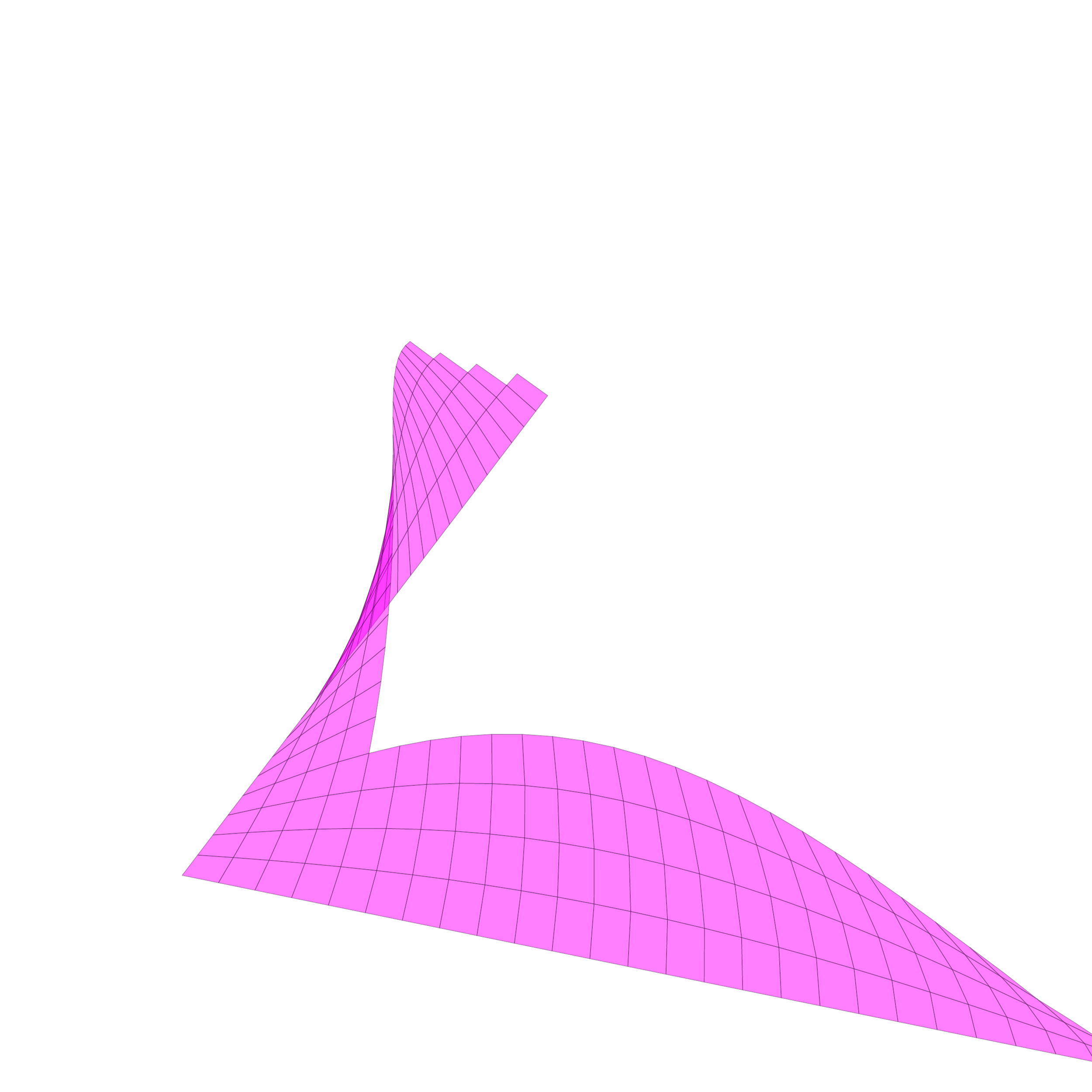}
   \end{subfigure}
   ~
   \begin{subfigure}[t]{0.25\textwidth}
     \centering
     \includegraphics[trim={1.cm, 1cm, 1cm, 5cm}, clip, width=\textwidth]{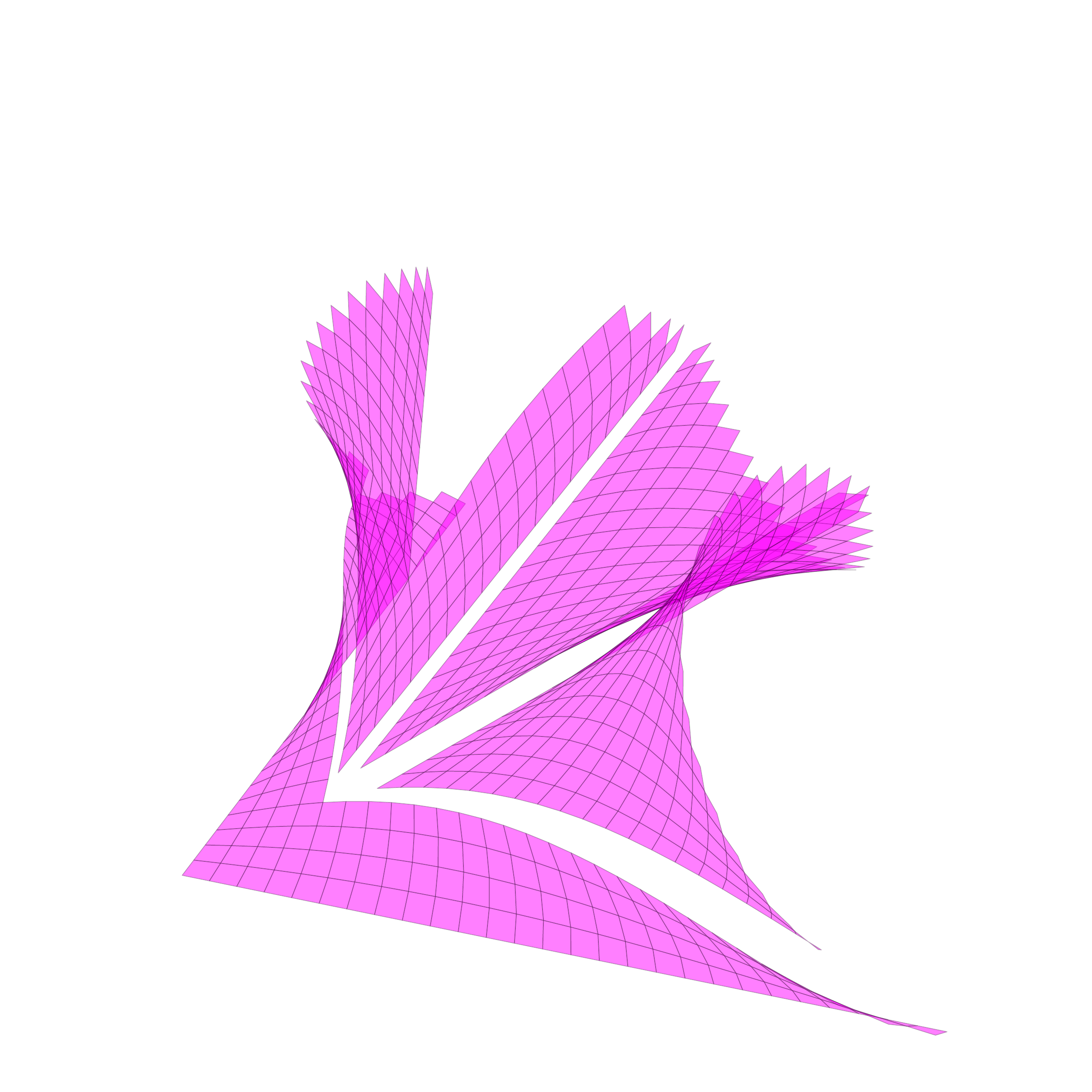}
   \end{subfigure}
   ~
   \begin{subfigure}[t]{0.25\textwidth}
     \centering
     \includegraphics[trim={1.cm, 1cm, 1cm, 5cm}, clip, width=\textwidth]{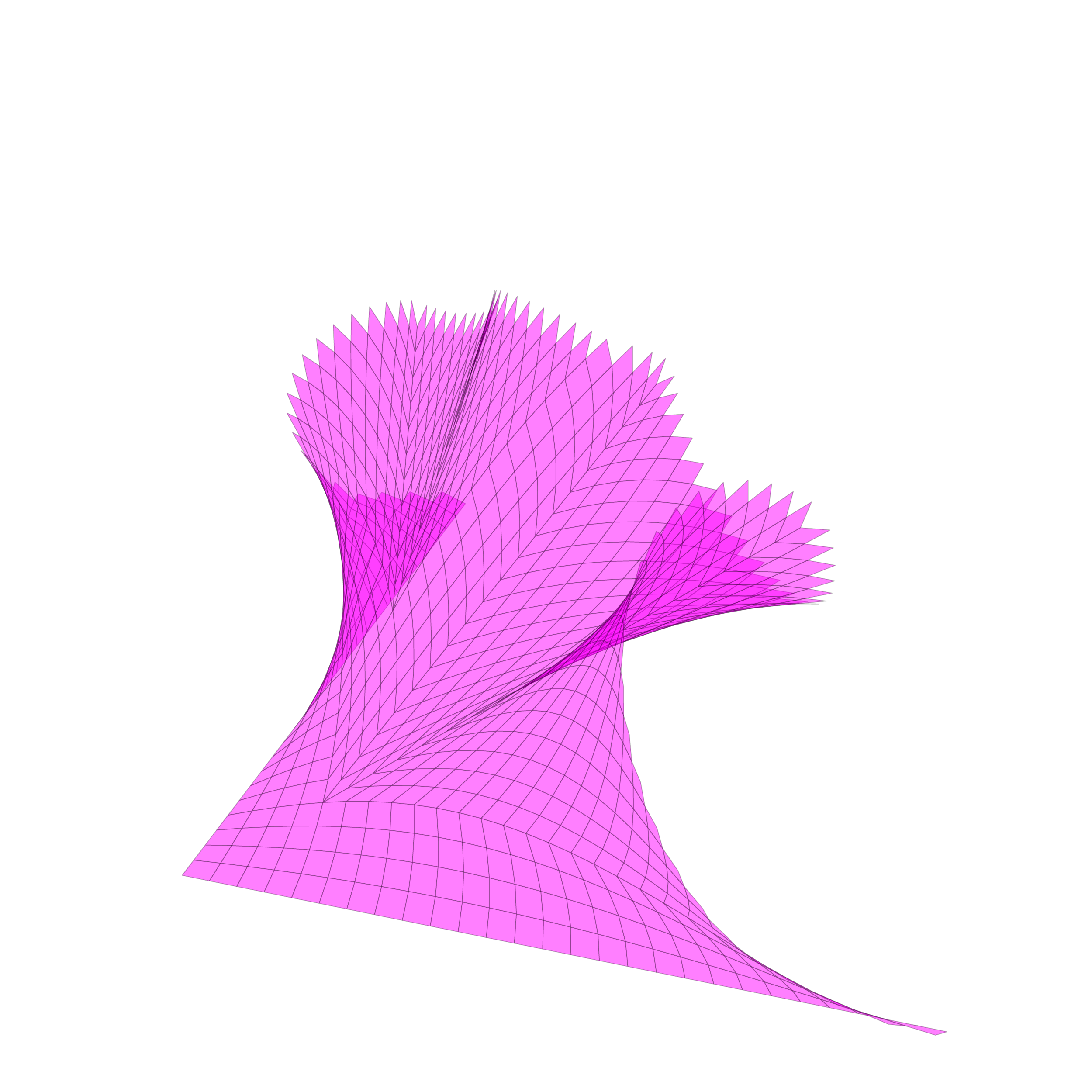}
   \end{subfigure}
   \caption{The process of constructing a discrete isometric immersion recursively by surgery. These figures illustrate the  generation an $K$-surface in $\mathbb{R}^3$ from a discrete Chebyshev net in the Poincar\'e disk.}
   \label{fig:surgery}
\end{figure}

We will present a full analysis of Alg.~\ref{alg:greedy-cuts} elsewhere. We note that every branch point $p_i$ has a non-empty open neighborhood, the interior of $\Omega_n \bigcup \Omega_{i+1} \bigcup \Omega_{i+2}\bigcup \Omega_{i+3}$, given by the parent sector $p_n$ and the 3 sectors at $p_i$. Compactness of the closed geodesic disk implies we only have finitely many branch points if we can show that the sectors cover the disk. 

We can do this, and more, by exploiting a ``dual" view point of the algorithm starting from the alternative, ``non-recursive", construction for isometric immersions of disks into $\mathbb{R}^3$. This immersion is achieved through patching sufficiently narrow Amsler sectors, whose singular edges are further away from the origin than the radius $R$, meeting at a single branch point of sufficiently high index at the origin \cite{gemmer2011shape,Gemmer2013Shape} (See also \S \ref{sec:kminusonemonkeysaddle} and Fig.~\ref{fig:monkeysaddleconstruction}). The comparison between the two methods is shown in Fig.~\ref{fig:branchedpoincareimmersions}. The figures show the discrete Chebyshev net in $\mathbb{H}^2$ corresponding to the recursive and single branch point isometries of disks of radii $2,3$ and $4$ respectively. The quads in the Chebyshev nets are colored by the $\kappa_{\max}$, the larger principal curvature. The figures suggest that the energies of both types of embeddings grow with $R$, the radius of the disk, but the energy of recursive embeddings grows  slowly compared to the energy of single branch point `periodic-Amsler' embeddings.

\begin{figure}[htbp]
	\begin{minipage}[t]{.46\textwidth}
	\centering
        \begin{subfigure}[t]{0.23\textheight}
              	{\includegraphics[trim={3.25cm, .5cm, 3cm, .5cm}, clip, width=\linewidth]{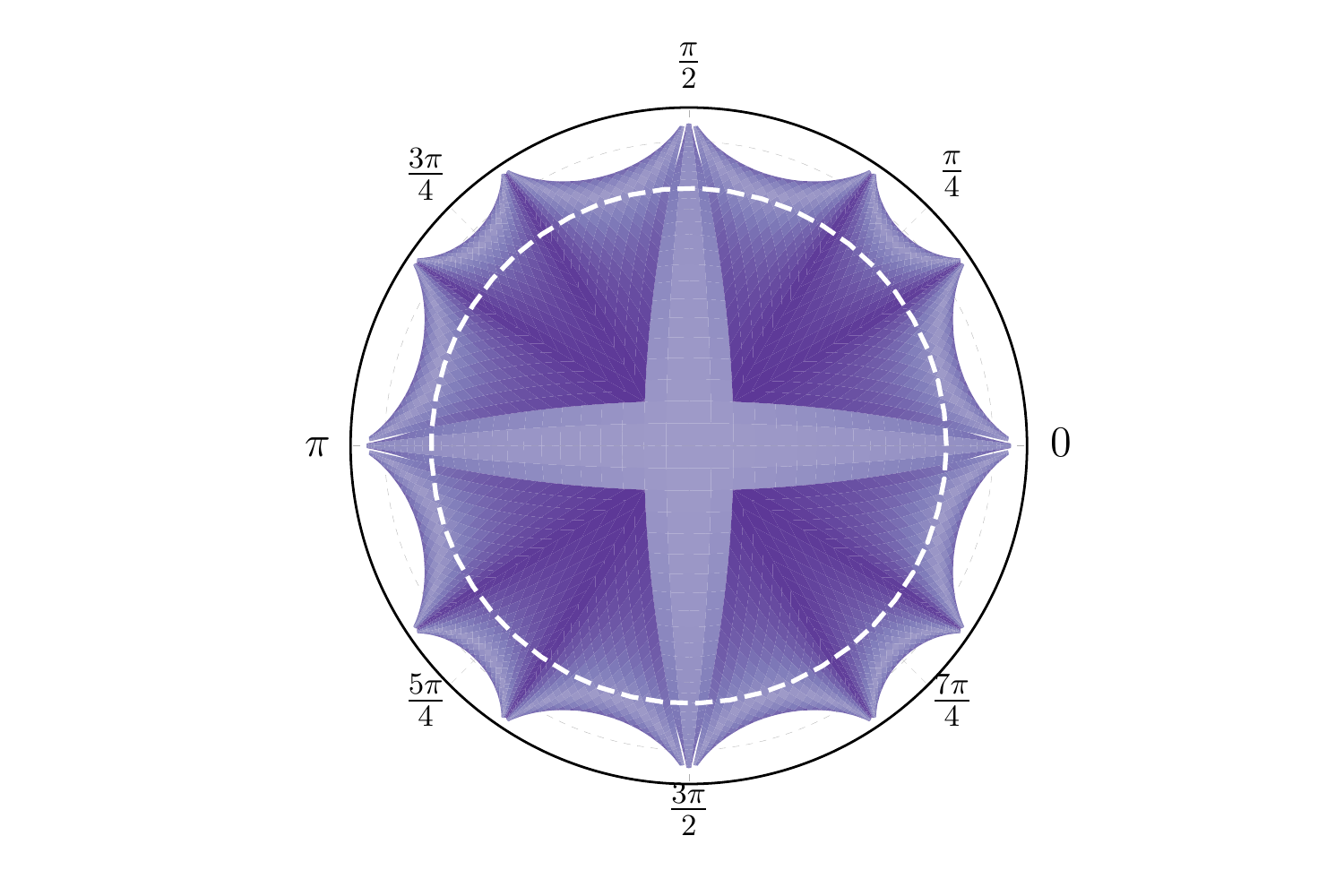}}
                \label{fig:branchsurfacepoinare2}
                \vspace{-\baselineskip}
        \end{subfigure}
        \begin{subfigure}[t]{0.23\textheight}
                {\includegraphics[trim={3.25cm, .5cm, 3cm, .5cm}, clip, width=\linewidth]{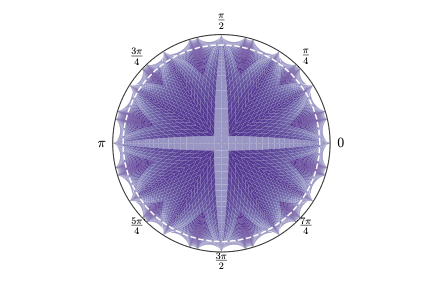}}
                \label{fig:branchsurfacepoincare3}
                \vspace{-\baselineskip}
        \end{subfigure}
        \begin{subfigure}[t]{0.23\textheight}
                {\includegraphics[trim={3.25cm, .5cm, 3cm, .5cm}, clip, width=\linewidth]{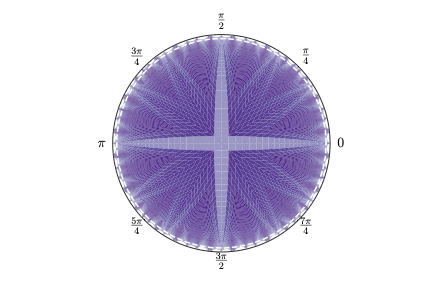}}
                \label{fig:branchsurfacepoinare4}
                \vspace{-\baselineskip}
        \end{subfigure}
     \end{minipage}%
	  \begin{minipage}[t]{0.46\textwidth}
	  \centering
        \begin{subfigure}[t]{0.23\textheight}
              	{\includegraphics[trim={3.25cm, .5cm, 3cm, .5cm}, clip, width=\linewidth]{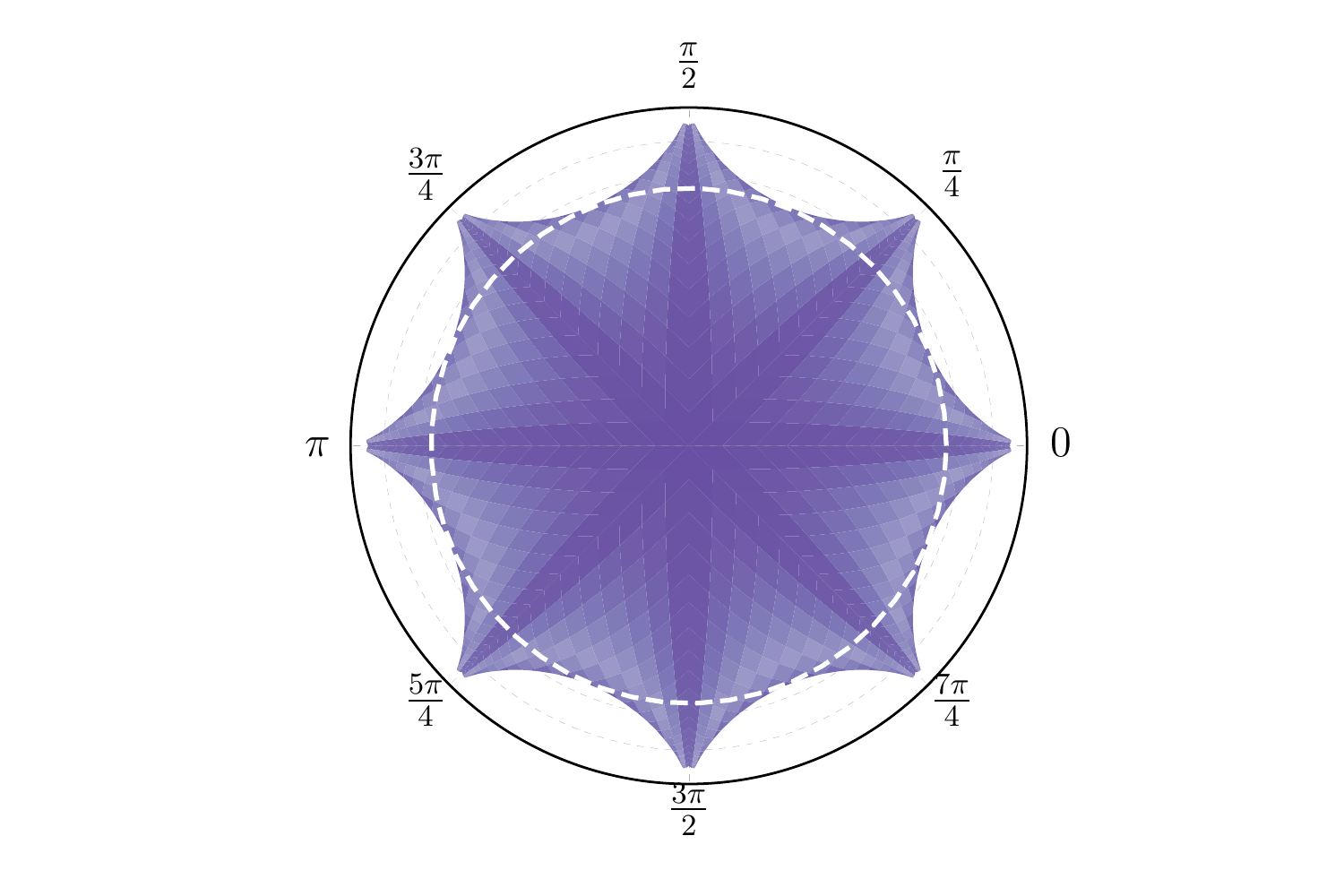}}
                \label{fig:periodicamslerpoincare2}
                \vspace{-\baselineskip}
        \end{subfigure}
        \begin{subfigure}[t]{0.23\textheight}
                {\includegraphics[trim={3.25cm, .5cm, 3cm, .5cm}, clip, width=\linewidth]{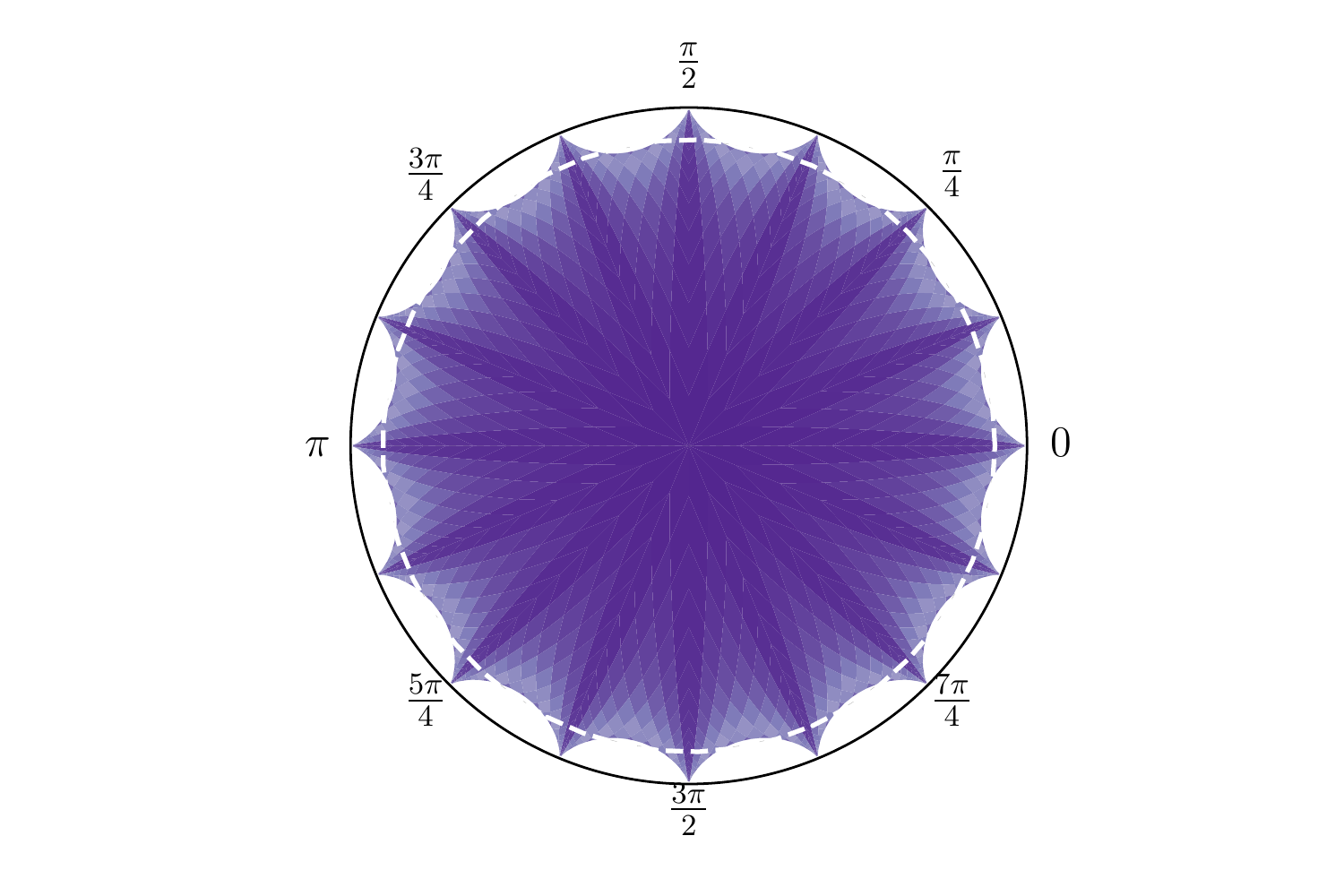}}
                \label{fig:periodicamslerpoincare3}
                \vspace{-\baselineskip}
        \end{subfigure}
        \begin{subfigure}[t]{0.23\textheight}
                {\includegraphics[trim={3.25cm, .5cm, 3cm, .5cm}, clip, width=\linewidth]{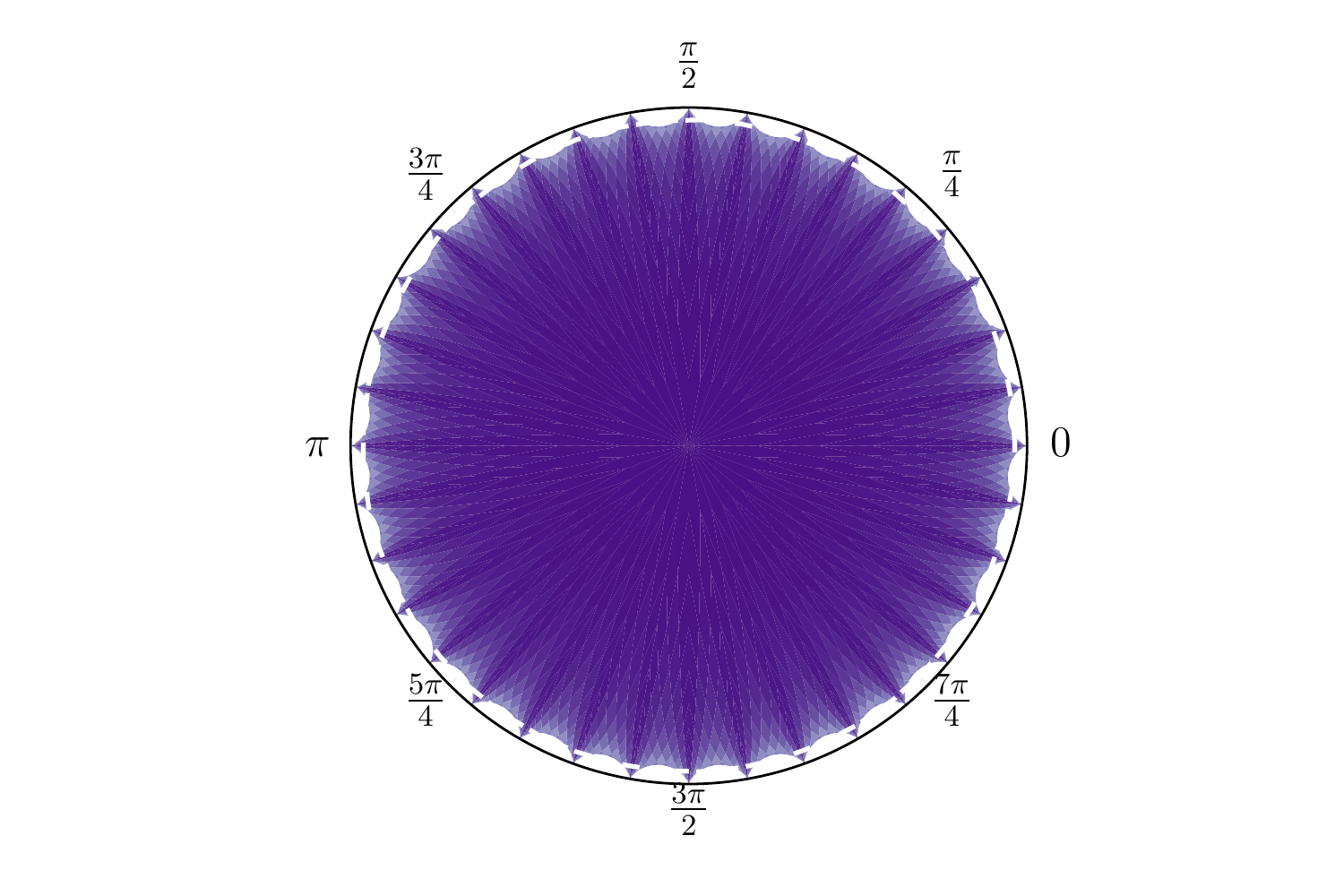}}
                \label{fig:periodicamslerpoincare4}
                \vspace{-\baselineskip}
        \end{subfigure}
     \end{minipage}%
     \caption{A comparison of isometric immersions of $\mathbb{H}^2$ via recursively constructed branched surfaces (left) and by a single branch point at the origin with a large index (right) as represented in the Poincar\'{e} disk. The figures show immersions with geodesic radii $R=2, 3$ and $4$ represented by the dashed line. The surfaces are colored by the max of the absolute principal curvatures: darker representing higher energy.}
     \label{fig:branchedpoincareimmersions}
\end{figure}

For the single branch point at the origin, the order of saddleness can be estimated $m_0 \sim C e^R$ for an $O(1)$ constant $C$. Algorithm~\ref{alg:greedy-cuts} has a dual interpretation as follows:
{
\begin{enumerate}
\item Start with a single branch point at the origin with $m' = 3^g m$ where $m$ is as defined in the algorithm and $g$ is determined by $3^{g-1} m < m_0 \leq 3^g m$. $g$ is the expected number of `generations' of branch points (see also `cut depth' in \S \ref{sec:amslerrecursion}). 
\item In the first step, retain a branch point with index $m$ at the origin and move $2m$ daughter branch points, each with degree $3^g$ outward in their respective sectors, until the maximum angle $\varphi$ over points in (each of) the ``growing" sectors at 0 equals the cutoff angle $\phi^*$. 
More precisely, this is equivalent to finding the locations $j^*,k^*$ in each of the $2 m$ initial sectors (Steps 12 and 13 in the algorithm) and this determines the locations of the branch points of the first generation.
\item Recursively, at the $k$-th stage, move $2\cdot 3^{k-1}\cdot m$ daughter branch points, each with degree $3^{g-k}$, outward until the max angle $\varphi$ in the sectors at the branch points in the $(k-1)$\textsuperscript{th} generation equals $\phi^*$.
\item At every stage, the union of the sectors cover the entire disk. 
\item This process is `monotonic', for $\phi^* \leq \pi/2$, because we have the following comparison principle.  Let $J = [0,u_0] \times [0,v_0]$. Let $\varphi_i, i=1,2$ denote solutions of the sine-Gordon equation $\partial_{uv} \varphi_i(u,v) = \sin \varphi_i(u,v)$ satisfying $0 < \varphi_i(u,v) < \pi/2$ on $J$. If $\phi_1(u,0) \leq \phi_2(u,0)$ for $0 \leq u \leq u_0$ and $\phi_1(0,v) \leq \phi_2(0,v)$ for $0 \leq v \leq v_0$ then it follows that $\phi_1 \leq \phi_2$ on $J$. 
\item This monotonicity implies that, for $m \geq 6, \phi^* \leq \pi/2$, the result of the greedy algorithm is obtained by starting with the appropriate periodic Amsler surface on the disk of radius $R$ and moving branch points outwards, {\em a process that increases} $\varphi$. We can discard a branch point and all of its sectors if it ever reaches the boundary of the disk, and {\em no new branch points ever enter the disk}. Formalizing this argument proves that Algorithm~\ref{alg:greedy-cuts} terminates, and further, obtains an apriori bound on the number of sectors $M \leq 2 \cdot 3^g \cdot m< 6 m_0$ and the minimum angle $\varphi \geq 3^{-g} \frac{\pi}{m}$ so $\mathcal{E}_\infty < C 3^g  m  < C' e^R$ for some constant $C'$.
\end{enumerate}
}
Numerically, we find that Algorithm~\ref{alg:greedy-cuts} terminates, even for $\phi^* > \pi/2$. 

\section{Distributed branch points and curvature energy} \label{sec:applications}

We now investigate the energies of the various classes of pseudospherical immersions. The principal curvatures are determined by the angle $\varphi(u,v)$ between the asymptotic directions as $\kappa_1 = \pm\tan\frac{\varphi}{2}, \kappa_2 = \mp\cot\frac{\varphi}{2}$. Consequently, the  bending energy (both $W^{2,\infty}$ and $W^{2,2}$) diverge if the singular edge $\varphi = 0$ or $\varphi = \pi$ encroaches the domain $\Omega \subset \mathbb{H}^2$. Our goal therefore is to construct immersions of $\Omega$ such that the angle $\varphi$ between the asymptotic lines satisfies $0 < \delta \leq \varphi \leq \pi-\delta < \pi$, where $\delta = \delta(\Omega) > 0$, and gives a quantitative measure of how ``non-singular" we can make an isometric immersion $\Omega \subset \mathbb{H}^2 \to \mathbb{R}^3$. $\delta$ is related to the max curvature energy by $\mathcal{E}_\infty = \cot(\delta)$.

Earlier analyses suggest that the energy optimal $C^2$ pseudospherical immersions of a geodesic disks are given by subsets of the universal cover of Minding's bobbin \cite{gemmer2011shape} (See also Example~\ref{ex:bobbin}, Eq.~\eqref{eq:c2energy}) giving 
\begin{equation}
\log\inf_{r \in C^2} \mathcal{E}_\infty[r] \sim R
\label{e2}
\end{equation}
where by $a \sim b$, we are conjecturing the existence of a constant $1 < C < \infty$ such that $C^{-1}b \leq a \leq Cb$ for all $R$. Alternative low energy immersions of disks are in the form of $C^{1,1}$ periodic Amsler surfaces~\cite{gemmer2011shape,Gemmer2013Shape} which introduce a single branch point at the origin. Even with the introduction of this branch point, there are still ``large" sets, in particular, disks with radius $R/2$ which are free of branch points and where the immersion is smooth (or can be approximated by smooth isometries as discussed in Remark~\ref{rem:approximable}), so Eq.~\eqref{e2} implies that, even for these periodic Amsler surfaces, $\log \mathcal{E}_\infty \sim R$.

\begin{figure}[hbpt]
        \begin{subfigure}[t]{0.5\textwidth}
    \centering
    {\includegraphics[height=6cm, trim={0.5cm 0.5cm 0.5cm 0.5cm}, clip]{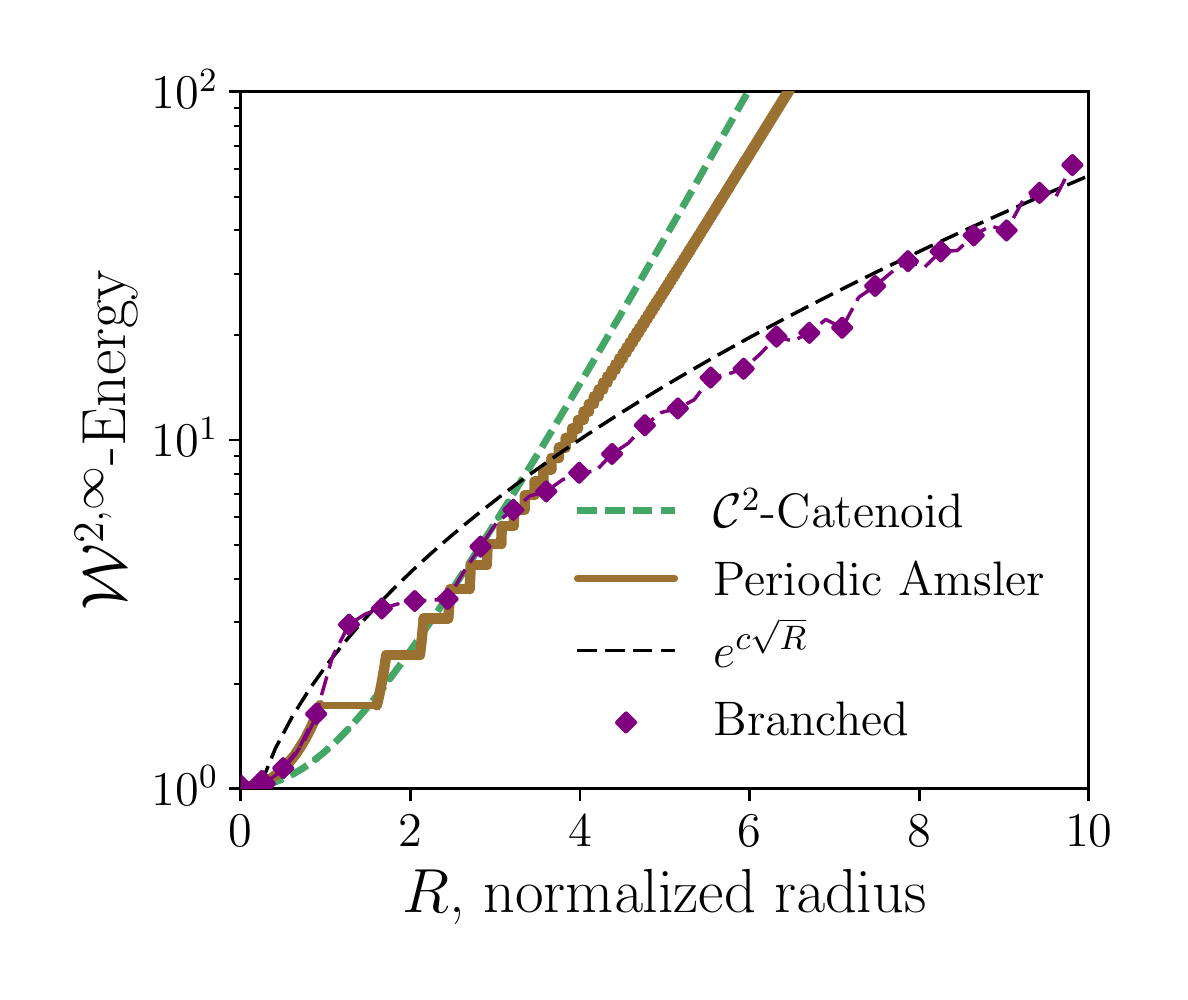}}
    \caption{}
    \label{fig:energywithradius}
\end{subfigure}%
\begin{subfigure}[t]{0.4\textwidth}
    \centering
    {\includegraphics[height=6cm, trim={0.5cm 0.5cm 0.5cm 0.5cm}, clip]{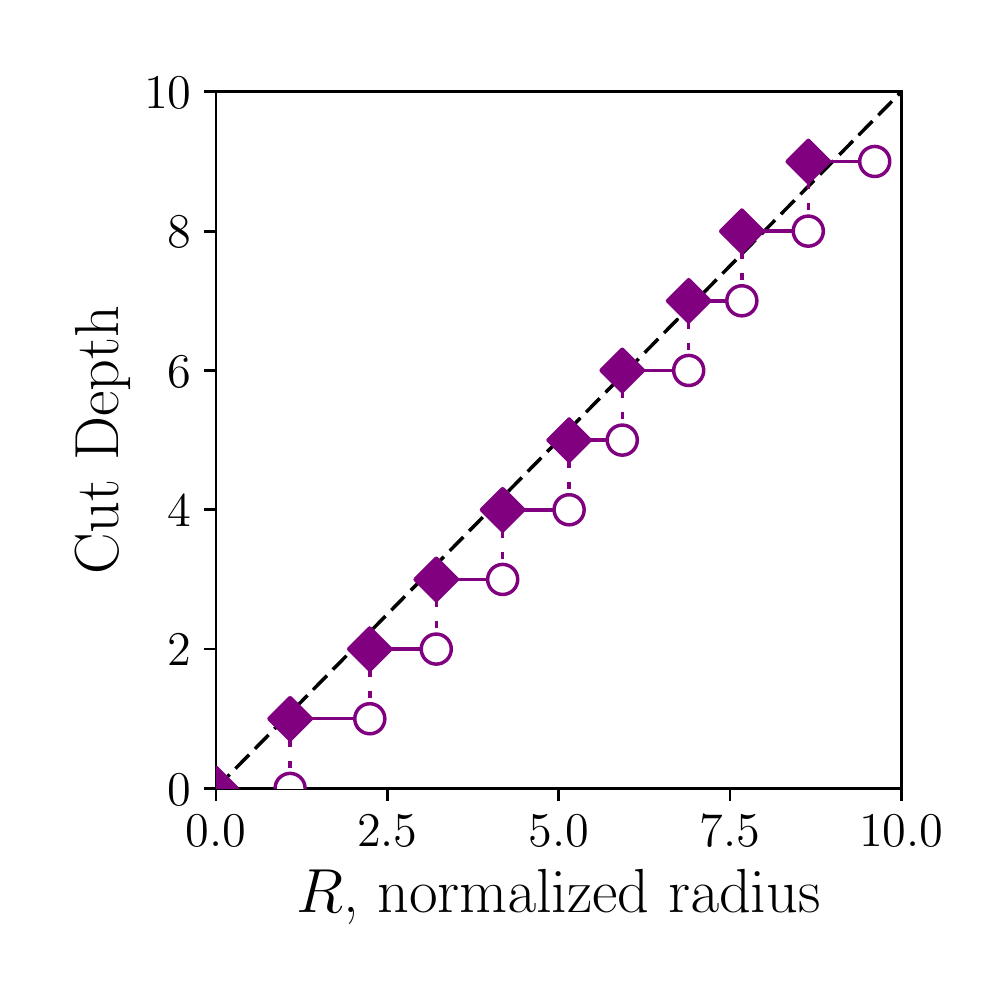}}
    \caption{}
    \label{fig:cutdepthwithradius}
    \end{subfigure}
     \begin{subfigure}[t]{0.8\textwidth}
    \centering
    {\includegraphics[height=7.5cm, trim={0.5cm 0.5cm 0.5cm 0.5cm}, clip]{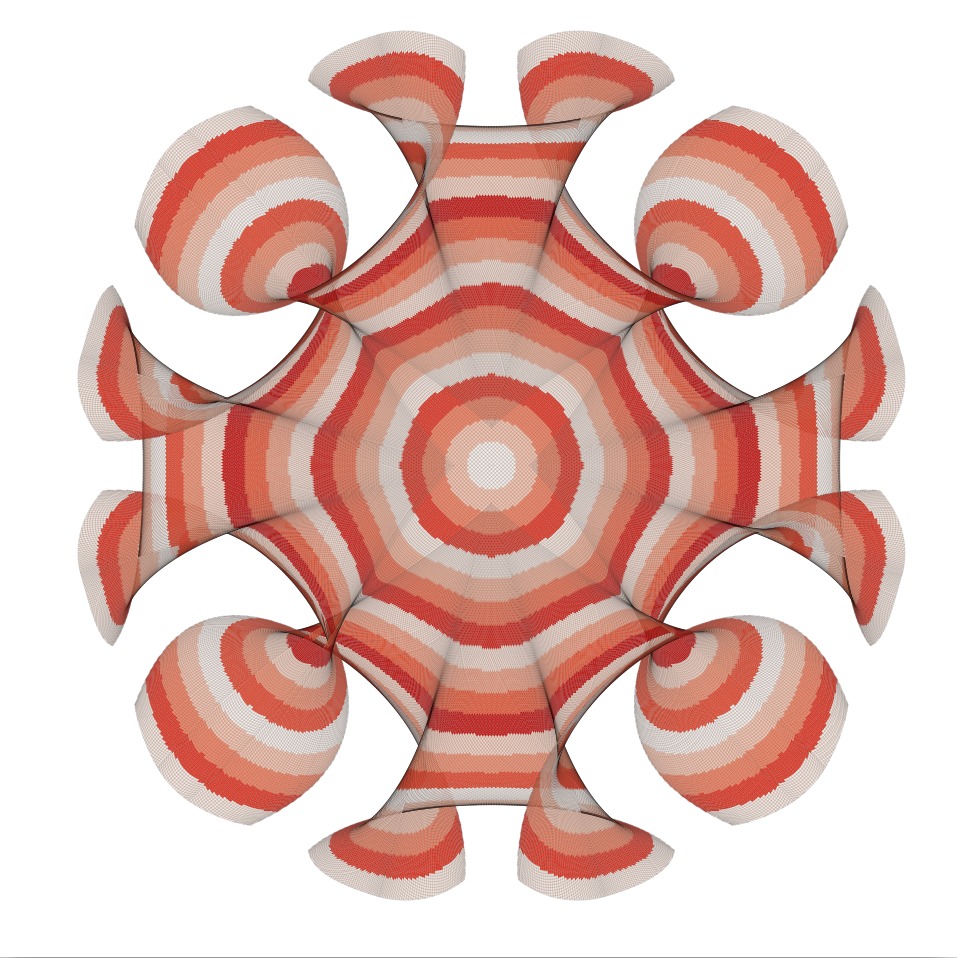}}
    \caption{}
    \label{fig:numerical_crochet}
    \end{subfigure}
    \caption{(a) The $\mathcal{E}_{\infty}$ energy for three types of immersions: Minding's bobbins ($C^2$-catenoid, thick-dashed), $C^{1,1}$ periodic-Amsler surfaces (solid) and $C^{1,1}$ branched surfaces (dashed-diamond). (b) The maximum recursion depth $n$  as a function of the geodesic radius $R$. (c) A numerically generated `hyperbolic crochet' obtained using Alg.~\ref{alg:greedy-cuts} on a disk of radius $R=3$. }
\end{figure}

Our construction (Algorithm~\ref{alg:greedy-cuts}) introduces distributed branch points, which appear ``as needed". In this case, as we argue below, we stave off the singular edge and obtain 
\begin{equation}
\log\inf_{r \in C^{1,1}} \mathcal{E}_\infty[r] \sim \sqrt{R}
\label{ebs}
\end{equation}
achieving an improvement in the scaling of the {\em logarithm of elastic (bending) energy}. The separation between the energy scales of the smooth and branched isometries is therefore enormous for large $R$. The number of generations of branch points, which we also call the {\em cut depth}, grows linearly with $R$. 

While we do not have rigorous proofs for these claims yet, we give arguments that illustrate the intuition behind these relations in \S\ref{sec:amslerrecursion}. We also have numerical evidence for the energy and cut depth scaling obtained from Algorithm~\ref{alg:greedy-cuts} applied to disks of radius up to 10. Figure \ref{fig:energywithradius} shows the analytically derived energy scaling for Minding's bobbin, conjectured as the minimizer of the elastic energy over the class of all $C^2$ isometric immersions, as in~\eqref{e2}. Periodic-Amsler surfaces exhibit a similar $\exp(R)$ scaling, though with an improved constant \cite{gemmer2011shape}. The energetic benefits of introducing distributed branch points is clear, with an apparent energy scaling $\exp(c \sqrt{R})$. The cut depth scales linearly with $R$ as shown in Figure \ref{fig:cutdepthwithradius}.
 Fig.~\ref{fig:numerical_crochet} shows an immersed pseudospherical surfaces with distributed branch points, a mathematical `hyperbolic crochet' with $R=3$.

\subsection{Recursion on Amsler type surfaces}
\label{sec:amslerrecursion}

We have implemented Algorithm~\ref{alg:greedy-cuts} on  disks of radii $R \leq 10$ and for various choices of the initial angle $\phi_0$ and the cutoff angle $\phi^*$. Fig.~\ref{fig:branch-tree} shows the branch points in a disk of radius 4 with $\phi_0 = \frac{\pi}{2}, \phi^* = \frac{3 \pi}{4}$. The solid lines indicate the parent-daughter relations among the branch points. The branch points form a tree since every branch point has a unique parent. We observe that every branch point (other than the origin) has 3 or fewer daughters, and the leaves of the tree are at different depths. The `Amsler nodes' along the diagonal are (typically) farther apart than the `pseudo-Amsler' off-diagonal nodes. 
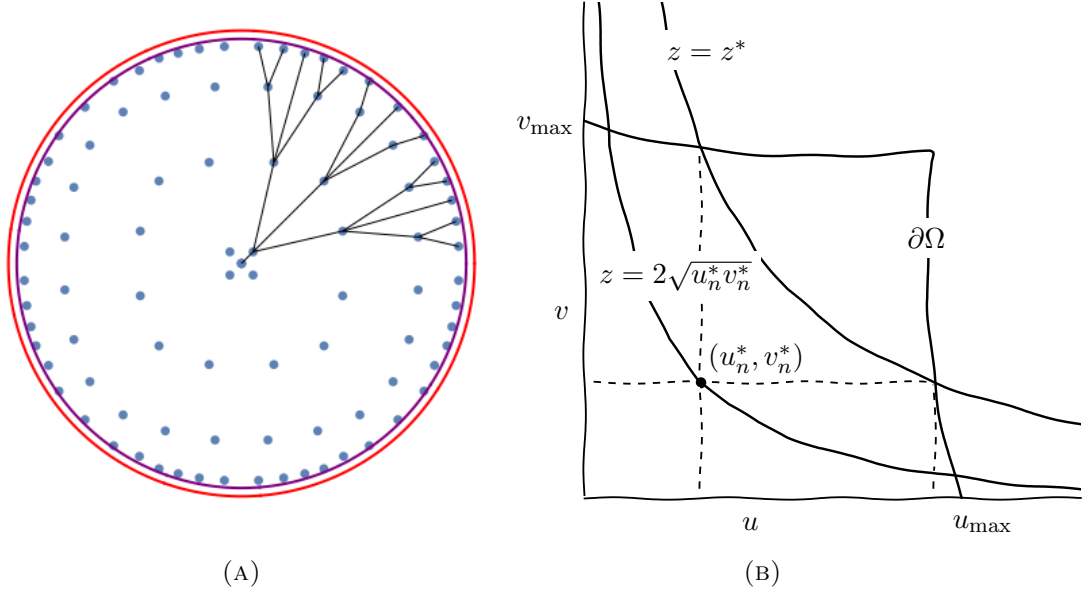
\begin{figure}[bpht]
\begin{subfigure}[t]{0.45\textwidth}
                \centering
                \includegraphics[width=.95\linewidth, trim={0cm -0.5cm 0cm 0cm}, clip]{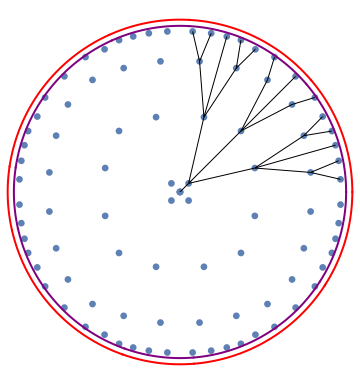}
                \caption{}
                \label{fig:branch-tree}
        \end{subfigure}%
        \label{fig:recursion}
 \begin{subfigure}[t]{0.45\textwidth}
{\begin{tikzpicture}
    \node[anchor=south west,inner sep=0] at (0,0) {\includegraphics[width=0.98\linewidth, trim={3.cm, -0.6cm, 3cm, 0cm}, clip]{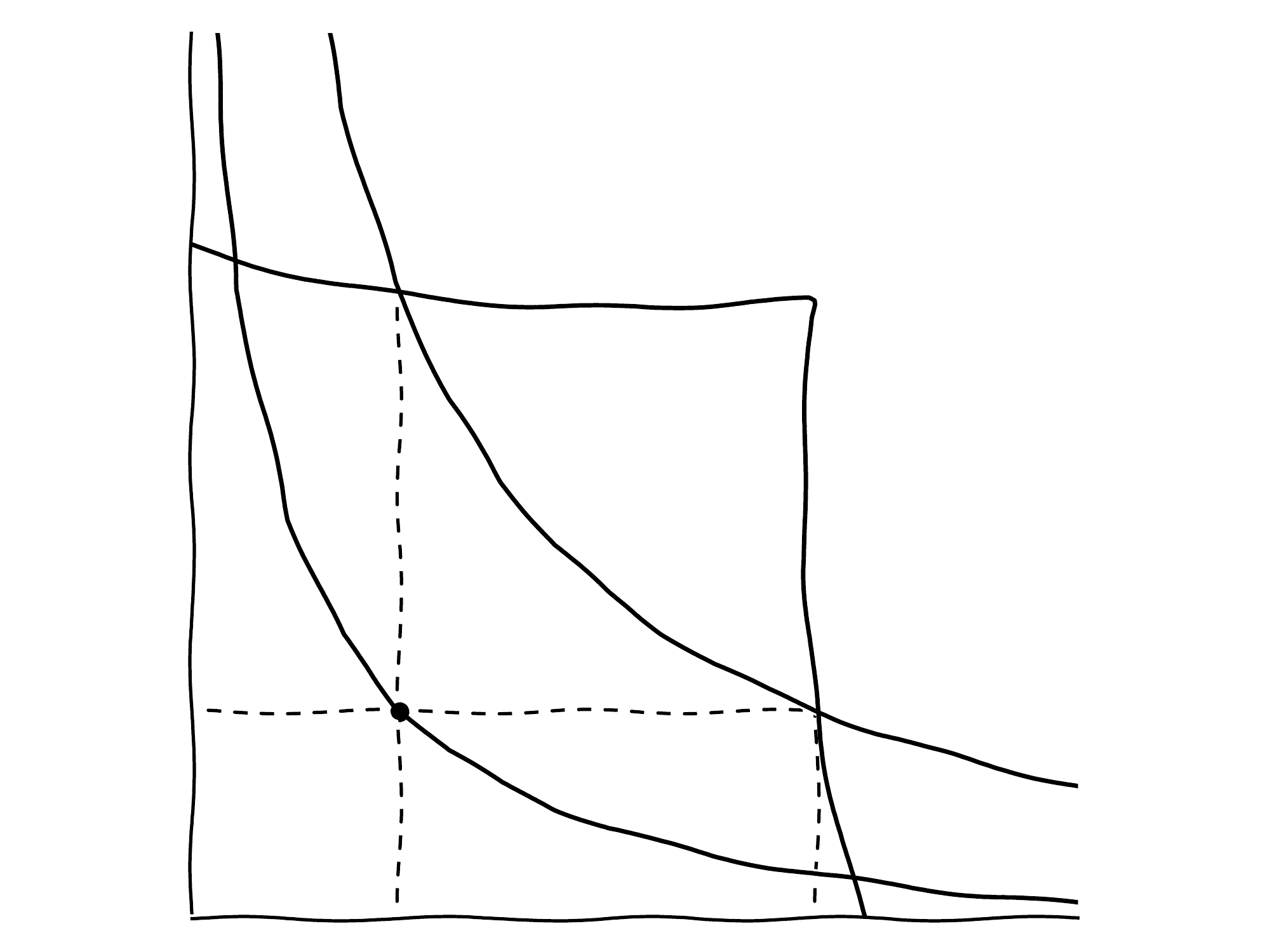}};
    \node [anchor=west] (ustar) at (4.8,0.2) {$u_{\max}$};
    \node [anchor=west] (uaxs) at (2,0.2) {$u$};
    \node [anchor=west] (vaxs) at (-0.5,3) {$v$};
    \node [anchor=west] (vstar) at (-1,5.5) {$v_{\max}$};
    \node [anchor=west] (unvn) at (1.55, 2.4) {$(u_n^*, v_n^*)$};
    \node [anchor=west] (zbig) at (1., 6.5) [fill=white] {$z = z^*$};
    \node [anchor=west] (zsel) at (.1, 3.5) [fill=white] {$z = 2\sqrt{u_n^* v_n^*}$};
    \node [anchor=west] (rad) at (4.2, 4.) [fill=white] {$\partial \Omega$};
\end{tikzpicture}}
\caption{}
\label{fig:cutsketch}
\end{subfigure}
        \caption{(a)  The Poincar\'e disk representation of 4 generations of distributed branch points in a disk of radius $R=4$. (b) Annotated illustration of an L-shaped cut in going from the $n$\textsuperscript{th} to the $n+1$\textsuperscript{th} generation. .}
\end{figure}

A schematic of the recursion procedure is illustrated in Fig.~\ref{fig:cutsketch}. The origin $u=v=0$ corresponds to a branch point in the $n$\textsuperscript{th} generation. Let $\varphi = \varphi(u,v)$ denote the angle between the asymptotic directions on the corresponding sector and we define $\phi_n = \varphi(0,0)$. An input to the recursion process is the given threshold $\phi^* < \pi$. If the locus of points where $\varphi(u,v) = \phi^*$ (denoted by $z = z^*$ in Fig.~\ref{fig:cutsketch}) intersects the boundary of the domain $\Omega$, then we need to introduce an $n+1$\textsuperscript{th} generation branch point. The location $(u_n^*,v_n^*)$ of this branch point is determined by the requirement that on the $L$-shaped region $[0,u_{\max}] \times [0,v_n^*] \bigcup [0,u_n^*] \times [0,v_{\max}]$  the angle satisfies $\varphi(u,v) \leq \phi^*$ guaranteeing that this region is bounded away from the singular edge. 

The angle $\phi_{n+1}$ for the next generation is given by $\phi_{n+1} = \frac{1}{3} \varphi(u_n^*,v_n^*)$. To analyze the recursion process and obtain scaling laws for the maximum curvature, we need to understand the relation between $\phi_n$ and $\phi_{n+1}$. Indeed, $\varphi$ is monotone in both $u$ and $v$ as it satisfies the $\varphi_{uv} = \sin \phi > 0$. The only mechanism that decreases $\varphi$ is the trisection at a branch point. Since the principal curvatures are given by $\pm \tan \frac{\varphi}{2}$ and $\mp \cot \frac{\varphi}{2}$, it follows that 
\begin{equation}
  \mathcal{E}_\infty = \max_{n,k}\left(\cot \frac{\phi_{n,k}}{2},\tan \frac{\phi^*}{2}\right)   
  \label{ebound}
\end{equation}
where $\phi_{n,k}$ is the angle at the $k$\textsuperscript{th} branch point in the $n$\textsuperscript{th} generation.

If the asymptotic curves $u=0$ and $v=0$ bounding a sector are geodesics, i.e. for Amsler sectors, we can analyze the relation between $\phi_{n+1}$ and $\phi_n$ in more detail. In this case, $\varphi(u,v)$ is a self-similar solution $\varphi = \varphi(2 \sqrt{uv})$ given by Eq.~\eqref{PIII}. We then have

\begin{lemma}
Let $\varphi$ be the solution of~\eqref{PIII} with $\varphi(0) = \phi_n > 0$ and let $u_{\max},v_{\max}, \phi^* < \pi$ be given. We also identify $\varphi(u,v) = \varphi(2 \sqrt{u v})$ as the corresponding solution of the sine-Gordon equation on the rectangle $[0,u_{\max}] \times [0,v_{\max}]$. There exist $u_n^*,v_n^* > 0$ such that  
\begin{enumerate}
\item  $\varphi(u,v) \leq \phi^*$ for all $(u,v) \in \Omega^* := [0,u_{\max}] \times [0,v_n^*] \bigcup [0,u_n^*] \times [0,v_{\max}]$.
\item $\varphi(u^*,v^*)  \geq \phi_n I_0\left( C(\phi^*) \zeta_n\right)$ where $I_0$ is the modified Bessel function of the first kind, $0 < C(\phi^*) <1$ is a constant that only depends on $\phi^*$, and
$$
\zeta_n = \frac{1}{2 \sqrt{u_{\max} v_{\max}}} \left(I_0^{-1} \left( \frac{\phi^*}{\phi_n}\right)\right)^2.
$$
\end{enumerate} 
\label{lem:amsler_recursion}
\end{lemma}

\begin{proof}
We can rewrite~\eqref{PIII} as the equivalent integral equation 
\begin{equation}
\varphi(z) = \phi_n + \int_0^z \int_0^w \sin(\varphi(\xi)) \xi \,d\xi \,\frac{dw}{w}. 
\label{eq:integral}
\end{equation}
$\varphi$ is therefore monotone increasing on an initial interval $[0,z^*]$ where $z^*$ is the smallest solution of $\varphi(z) = \phi^*$. For $0 < \phi_n \leq \varphi \leq \phi^*$, we have the elementary inequalities
\begin{equation}
C^2 \varphi \leq \sin \varphi \leq \varphi, \quad  \mbox{ where } C = C(\phi^*) = \sqrt{\sin \phi^*/\phi^*} < 1.
\label{eq:defC}
\end{equation}
Using these inequalities in conjunction with the integral equation~\eqref{eq:integral} and the closed form solution $u = \phi_n I_0(C z)$ for the linear differential equation $\varphi'' + z^{-1} \varphi' = C^2 \varphi, \varphi(0) = \phi_n, \varphi'(0) = 0$ \cite[\S 9.6]{Abram_Stegun}, we obtain the bounds 
\begin{equation}
I_0(C(\phi^*) z) \leq \frac{\varphi(z)}{\phi_n} \leq I_0 (z)
\label{boundCphi*}
\end{equation}
for all $0 \leq z \leq z^*$. Setting $u_n^* = \frac{1}{4 v_{\max}}\left(I_0^{-1}\left(\frac{\phi^*}{\phi_n}\right)\right)^2, v_n^* = \frac{1}{4 u_{\max}}\left(I_0^{-1}\left(\frac{\phi^*}{\phi_n}\right)\right)^2$ and recognizing that $\varphi(u^*_n,v^*_n) \geq \phi_n I_0(2 C(\phi^*)\sqrt{u_n^* v_n^*})$ the result follows. 
\end{proof}
\begin{remark}
From the preceding lemma, we get the recursion for an Amsler sector
$$
\phi_{n+1} \geq \frac{\phi_n}{3} I_0 \left(  \frac{C(\phi^*)}{2 \sqrt{u_{\max} v_{\max}}} \left(I_0^{-1} \left( \frac{\phi^*}{\phi_n}\right)\right)^2\right) \geq \frac{\phi_n}{3} I_0 \left(  \frac{C(\phi^*)}{2 R} \left(I_0^{-1} \left( \frac{\phi^*}{\phi_n}\right)\right)^2\right),
$$
where the second inequality obtains from $u_{\max} \leq R, v_{\max} \leq R$. We thus get a relation  with explicit dependences on the parameters in the recursion, $R$ and $\phi^*$. Since $\phi_{n+1}/\phi_n \geq 1$ for sufficiently small $\phi_n$, it is also easy to see that, there is a constant $C'(\phi^*)$, {\em independent} of $R$, such that $\phi_{\min} := \phi^*/I_0(C'(\phi^*) \sqrt{R})$ has the property that $\phi_n \geq \phi_{\min}$ for all $n$ if $\phi_0 \geq \phi_{\min}$.  Note also that we are free to pick a particular value of $\phi^*$ (or even values from any compact set in $(0,\pi)$) and drop all the dependences on $\phi^*$.
\label{rem:amsler_scaling}
\end{remark}
The preceding analysis holds for Amsler sectors but most of the sectors generated by Alg.~\ref{alg:greedy-cuts} are not Amsler sectors. Rather, they are pseudo-Amsler sectors and only one boundary is a geodesic. Consequently, we cannot assume that the quantitative relation from Lemma~\ref{lem:amsler_recursion} will hold for these psuedo-Amsler sectors as well. The lessons we draw are qualitative -- that the analysis for Amsler sectors helps identify `good' sets of variables i.e. the appropriate combinations of $R, \phi_n,\phi_{n+1}$ that might satisfy `universal' relations.

For a general (not necessarily Amsler) sector we define the quantity 
\begin{equation}
    \alpha_n^2 = \frac{1}{4 s_n}\left(I_0^{-1}\left(\frac{\phi^*}{\phi_n}\right)\right)^2, 
\label{eq:cut-coordinates}
\end{equation}
where $s_n$ is the distance from the branch point to the boundary of the domain. The intuition for this choice is that $s_n^2 \geq u_{\max} v_{\max}$ and $\zeta_n$  is like $2 \alpha_n^2$ . The argument is Remark~\ref{rem:amsler_scaling} will apply to all sectors if we can prove an inequality $\frac{\phi_{n+1}}{\phi_n} \geq f(2 C \alpha_n^2)$  where $C>0$ is independent of $R$, and $f \geq 1$ for sufficiently large values of its argument.

We apply Alg.~\ref{alg:greedy-cuts} to disks of various sizes and  we record $\alpha_n$ (as defined in~\eqref{eq:cut-coordinates}) and the ratio $\frac{\phi_{n+1}}{\phi_n}$, relating the opening angles of the daughter sectors to the angle of the parent sector, at each cut $(u^*,v^*)$. Fig.~\ref{fig:frontier} shows the scatter plots of $\frac{\phi_{n+1}}{\phi_n}$~vs.~$\alpha_n^2$  for various choices of $R, \phi_0$ and $\phi^*$. On each of these plots, we have also drawn the curves $f_1(\alpha) = \frac{1}{3} I_0(2 \alpha^2)$ and its supporting quadratic $f_2(\alpha) = \left(\frac{\alpha}{\alpha^*}\right)^2$ where $\alpha^* =  \left(\sup_{w \geq  0} \frac{w}{\sqrt{f_1(w)}}\right)$. The data suggest the following observations: 
\begin{enumerate}
    \item The plots are essentially the same if there are sufficiently many branch points, independent of $\frac{\pi}{6} \leq \phi_0 < \phi^* \leq \frac{4 \pi}{5}$ and  $R \geq 6$.
    \item The points are clustered in two families. The Amsler nodes satisfy $\frac{\phi_{n+1}}{\phi_n} \geq \frac{1}{3}I_0(2 \alpha_n^2)$ i.e. the best possible bound from Lemma~\ref{lem:amsler_recursion}, given by $C(\phi^*) = 1$.   The pseudo-Amsler nodes {\em do not satisfy this bound}. They {\em seem to satisfy} a weaker bound given by  $\frac{\phi_{n+1}}{\phi_n} \geq \max\left(\frac{1}{3},\left(\inf_{w > 0} \frac{I_0(2w^2)}{3 w^2} \right)\alpha_n^2\right) = \max\left(\frac{1}{3},\left(\frac{\alpha}{\alpha^*}\right)^2\right) $. 
\end{enumerate}
\begin{figure}[htbp]
\begin{subfigure}[t]{0.37\textwidth}
                \centering
                \includegraphics[height=6.5cm, trim={4cm 0.5cm 3cm 0.5cm}, clip]{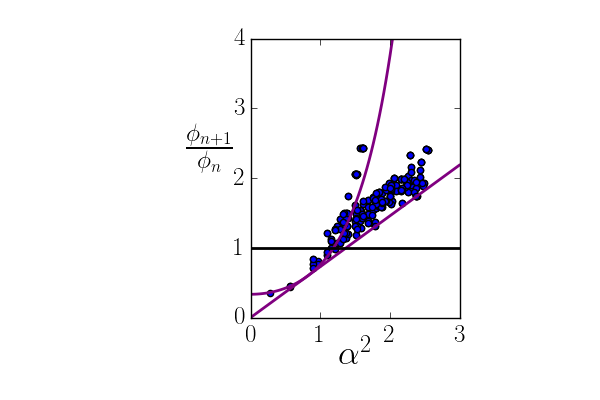}
                \caption{}
                \label{fig:frontierR8phi016}
\end{subfigure}%
\begin{subfigure}[t]{0.32\textwidth}
                \centering
                \includegraphics[height=6.5cm, trim={4cm 0.5cm 4cm 0.5cm}, clip]{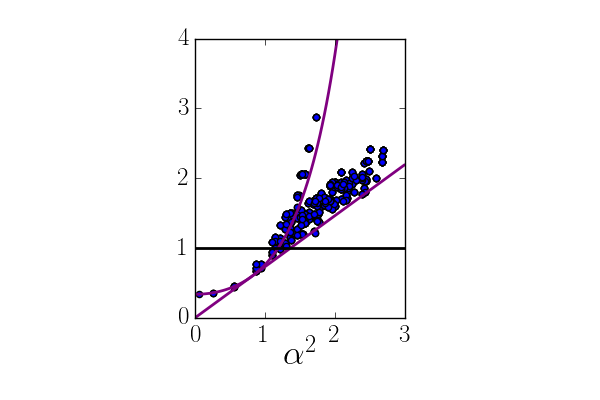}
                \caption{}
                \label{fig:frontierR8phi012}
\end{subfigure}%
\begin{subfigure}[t]{0.31\textwidth}
                \centering
                \includegraphics[height=6.5cm, trim={4.5cm 0.5cm 4cm 0.5cm}, clip]{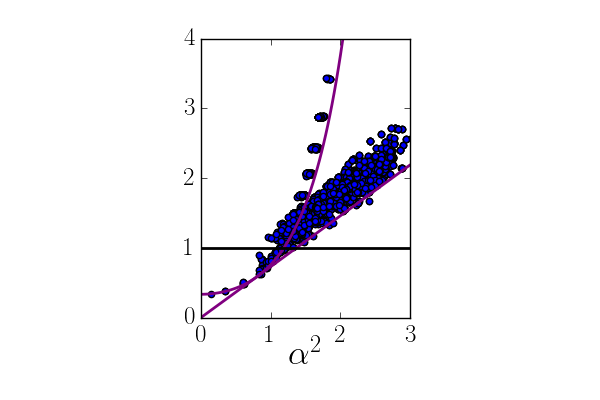}
                \caption{}
                \label{fig:frontierR10phi014}
\end{subfigure}%
        \caption{Scatter plots of $\frac{\phi_{n+1}}{\phi_n}$ vs. $\alpha_n^2$ for branched immersions generated by Alg.~\protect{\ref{alg:greedy-cuts}}. We plot $f_1(\alpha) = \frac{1}{3} I_0(2 \alpha^2)$ and the supporting quadratic $f_2(\alpha) \approx 0.73 \alpha^2$. The parameters for the individual plots are: (a) $R = 8, \phi_0 = \frac{\pi}{6}, \phi^* = \frac{4 \pi}{5}$, (b) $R = 8, \phi_0 = \frac{\pi}{2}, \phi^* = \frac{3 \pi}{4}$, and (c) $R = 10, \phi_0 = \frac{\pi}{4}, \phi^* = \frac{4 \pi}{5}$.}
        \label{fig:frontier}
\end{figure}
Assuming that the inequalities suggested by the numerical results indeed hold for all $R$, the same argument as in Remark~\ref{rem:amsler_scaling} gives a conservative estimate of $\phi_{\min}$ by setting
$$
(\alpha^*)^2 \equiv \frac{1}{4R} \left[I_0^{-1} \left( \frac{\phi^*}{\phi_{\min}}\right)\right]^2 \quad \implies \quad \phi_{\min} = \frac{\phi^*}{I_0(2 \alpha^* \sqrt{R})} \sim \exp(-\alpha^*\sqrt{2R}),
$$
since, from $s_n < R$, we are guaranteed that $\phi_{n+1} /\phi_n \geq (\alpha_n/\alpha^*)^2 \geq 1$ if $\phi_n$ is ever as small as $\phi_{\min}$. Eq.~\eqref{ebs}, our energy bound for isometries with branch points, now follows from combining $\phi_{n,k} \geq \phi_{\min}$ for all branch points with Eq.~\eqref{ebound}.

From the bound~\eqref{ebs} for $\mathcal{E}_\infty$ and the estimate in~\eqref{e2} for $C^2$ patches devoid of branch points, it follows that we cannot have a region of size about $\sqrt{R}$ that is free of branch points.
The area of a disk with radius $R$ scales like $\exp(R)$ while the ``largest" size of  regions free of branch points can only be $\exp(\sqrt{R})$. Consequently, we get that the number of branch points scales like $\exp(R - \sqrt{R})$. Since each parent has (at most) 3 daughter branch points in Algorithm~\ref{alg:greedy-cuts}, the number of branch points grows (roughly) exponentially with the number of generations, and it follows that the cut depth scales like 
$$
n \sim \max(R - \sqrt{R},0),
$$
corresponding to a function ``nearly" linear function whose slope increases slowly, precisely as we observe in Fig.~\ref{fig:cutdepthwithradius}.

\section{Discussion} \label{sec:discussion}

 Branch points are novel topological defects in $C^{1,1}$ hyperbolic surfaces that allow significant shape changes, while they do not concentrate stretching energy. They are unique in this aspect, since most other defects in condensed matter systems do concentrate energy.

In our view, these are some of the key results from this work --
\begin{enumerate}
 \item In definition~\ref{def:A-complex} we introduce the notion of an asymptotic complex that encodes the combinatorics of the asymptotic network and  characterizes the nontrivial topology induced by the ramification of the corresponding Gauss normal map.
    \item We define a topological index for branch points and prove it is ``robust" (Theorem~\ref{thm:branched}). 
    \item We prove a generalization of the sine-Gordon equation for surfaces with branch points in Theorem~\ref{thm:sg-branched}. This result illustrates why optimizing the bending energy among isometric immersions of pseudospherical surfaces naturally leads to distributed branch points  (see Remark~\ref{rmk:distributed}). 
    \item In \S\ref{sec:poincare-ddg} we introduce a new discrete net for the basic object of interest in elasticity, the deformation map from the Lagrangian to the Eulerian frame for pseudospherical surfaces. Our method {\em does encode} the asymptotic complex and the topology of branch points and thereby distinguishes $C^{1,1}$ immersions from $C^2$ immersions, in contrast to finite difference/FEM methods which are `branch-point agnostic'. 
    \item We formulate an algorithm, Alg.~\ref{alg:greedy-cuts}, to generate pseudospherical surfaces with distributed branch points and (relatively) slower growth in the maximum curvature with the size of the domain, than for $C^2$ immersions.
    \item \label{item:motivation} We numerically find an energy gap between branched and smooth pseudospherical surfaces that leads to recursive/self-similar, fractal-like patterns in the distribution of branch points, and partially answers our motivating question -- {\em why do we observe `universal' buckling patterns in hyperbolic surfaces?}
\end{enumerate}

We now expand on item~\ref{item:motivation}, which is the central motivating question for this work.  Bounded subsets of smooth hyperbolic manifolds can always be embedded smoothly and isometrically in $\mathbb{R}^3$. There is thus no need for these sheets to stretch, and their morphology results from a `global' competition between the two principal curvatures \cite{EPL_2016} (See also Example~\ref{ex:bobbin}).  This is in contrast to other multi-scale phenomena in thin sheets \cite{Muller_review} which are manifestly driven by a competition between stretching and bending energies \cite{benAmar1997Crumpled,science.paper,venkataramani2003lower,bella2014metric,olbermann2016d-cone} or more generally, energies of different physical origins \cite{Davidovitch2011Prototypical,chopin2014roadmap,bella2014wrinkles,davidovitch2019geometrically,tobasco2020curvaturedriven}.

We argue that branch points arise from the dependence of the max curvature/bending energy on the regularity class of the immersion $y:B_R \to \mathbb{R}^3$. The results in \S\ref{sec:branchedsurfaces} and algorithm in \S\ref{sec:ddg} are steps towards a  {\em quantitative} expression of this idea.
Our numerical results and (a non-rigorous) scaling argument suggest, for a disk of radius $R$ and Gauss curvature $K=-1$ immersed in $\mathbb{R}^3$, the optimal max curvature $\mathcal{E}_\infty = \kappa_{max}$ grows as 
\begin{equation}
\log \inf_{y:B_R \to \mathbb{R}^3} \kappa_{max} \sim \begin{cases} R & C^2 \mbox{ or smoother isometries,} \\
\sqrt{R} & C^{1,1} \mbox{ branched isometries.}  \end{cases}
\label{eq:conjecture}
\end{equation}
The evidence for this conjecture is presented in  Fig.~\ref{fig:energywithradius}.

 If true, conjecture~\eqref{eq:conjecture} would explain why, for sufficiently large disks, isometries with distributed branch points are preferred. The related argument for cut-depth indicates how the branch points will be distributed, and ``explains" the observed self-similar buckling patterns in thin hyperbolic objects. The energy gap  in~\eqref{eq:conjecture} would constitute an entirely new class of examples of the Lavrentiev phenomenon in nonlinear elasticity \cite{Foss2003Lavretiev,Ball1985One}. The Lavrentiev phenomenon is known to be an obstacle for numerical minimization of the energy functional since discrete approximations often converge to a smooth {\em pseudominimizer} rather than the true singular minimizer \cite{Ball1987Numerical}. It is thus of considerable interest to investigate the convergence properties of our DDG based methods, that discretize $C^{1,1}$ isometries, and compare the results with existing FEM and finite difference methods for shells and plates.

\appendix

\section{Asymptotics of Painlev\'{e} III} \label{appndx:amsler}

We can get  more accurate estimates than implied by the bounds in~\eqref{boundCphi*}. For $\varphi \ll 1,$ the Painlev\'{e} III equation~\eqref{PIII}, and the associated boundary conditions, reduce to
\begin{equation*}
\varphi''(z) + \frac{\varphi'(z)}{z} - \varphi(z) = 0, \quad \varphi(0) = \varphi_0, \quad \varphi'(0) = 0.
\end{equation*}
The solution is given by $\varphi(z) = \varphi_0 I_0(z)$, where $I_0$ is the modified Bessel function of the first kind \cite[\S9.6]{Abram_Stegun}. 
From the small and large $z$ asymptotics of $I_0$ \cite[\S9.7]{Abram_Stegun}, we get
\begin{align*}
\varphi_{\textrm{inner}}(z) &= \varphi_0\left(1 + \frac{z^2}{4} + O\left(z^4 \right)\right), \textrm{ for } z\ll 1, \\
\varphi_{\textrm{outer}}(z) &= \varphi_0 \frac{e^{z}}{\sqrt{2\pi z}}\left(1 + \frac{1}{8z} + O\left(\frac{1}{z^2}\right)\right), \textrm{for} z \gg 1.
\end{align*}
For the regime $z \gg 1, \varphi \approx \pi$, we have the weakly damped pendulum equation:
\begin{align}
\varphi''(z) - \sin\varphi(z) = - \frac{\varphi'(z)}{z} \approx 0,
\end{align}
with asymptotic solutions of the form
\begin{equation}
\varphi_{\textrm{pend}}(z) \approx \pi - A\sin(z^* - z),
\end{equation}
for a slowly-varying amplitude $A$ that changes over many cycles of the pendulum. We are only interested in the first crossing  $\phi(z^*) = \pi$, so we can assume that $A$ is constant and determine  $A$ by matching the large $z$ asymptotics of the Bessel solution with the pendulum solution. From the Bessel solution, we derive initial data for the pendulum equation, fixing the energy level for this conservative system:
\begin{equation}
(\varphi_{\text{pend}}(0), \varphi_{\text{pend}}'(0)) \approx \left(\frac{\varphi_0 e^{z}}{\sqrt{2\pi z}}, \frac{\varphi_0 e^{z}}{\sqrt{2\pi z}}\right) = (\delta, \delta),
\end{equation}
where we match at such a point $z$ that $z \gg 1, \delta \ll 1$. The energy of the pendulum solution is given by
\begin{equation}
E = \frac{\varphi'^2}{2} + \cos\varphi \approx 1 + \frac{\delta^4}{24},
\end{equation}
as $\cos\varphi$ is the potential and $\delta \ll 1$. Substituting the data into the energy we find
\begin{align*}
1 + \frac{\delta^4}{24} &\approx \frac{1}{2}\left(A'\sin(z^*-z) + A\cos(z^*-z)\right)^2 + \cos\varphi, \\
            &\approx \frac{1}{2}\left(A'\sin(z^*-z) + A\cos(z^*-z)\right)^2 - 1 + \frac{(\pi - \varphi)^2}{2}, \\
            &\approx \frac{1}{2}\left(A'\sin(z^*-z) + A\cos(z^*-z)\right)^2 - 1 +\frac{1}{2}A^2\sin^2(z^*-z), \\
            &\approx -1 \frac{1}{2}\left[A'^2\sin^2(z^*-z) - 2A'A\sin(z^*-z)\cos(z^*-z) + A^2\right]. \\
\intertext{Which in the case of slowing varying $A$ simplifies to}
A &\approx 2\sqrt{1 + \frac{\delta^4}{48}}.
\end{align*}
\begin{figure}[htbp]
    \centering
    \includegraphics[width=.5\linewidth]{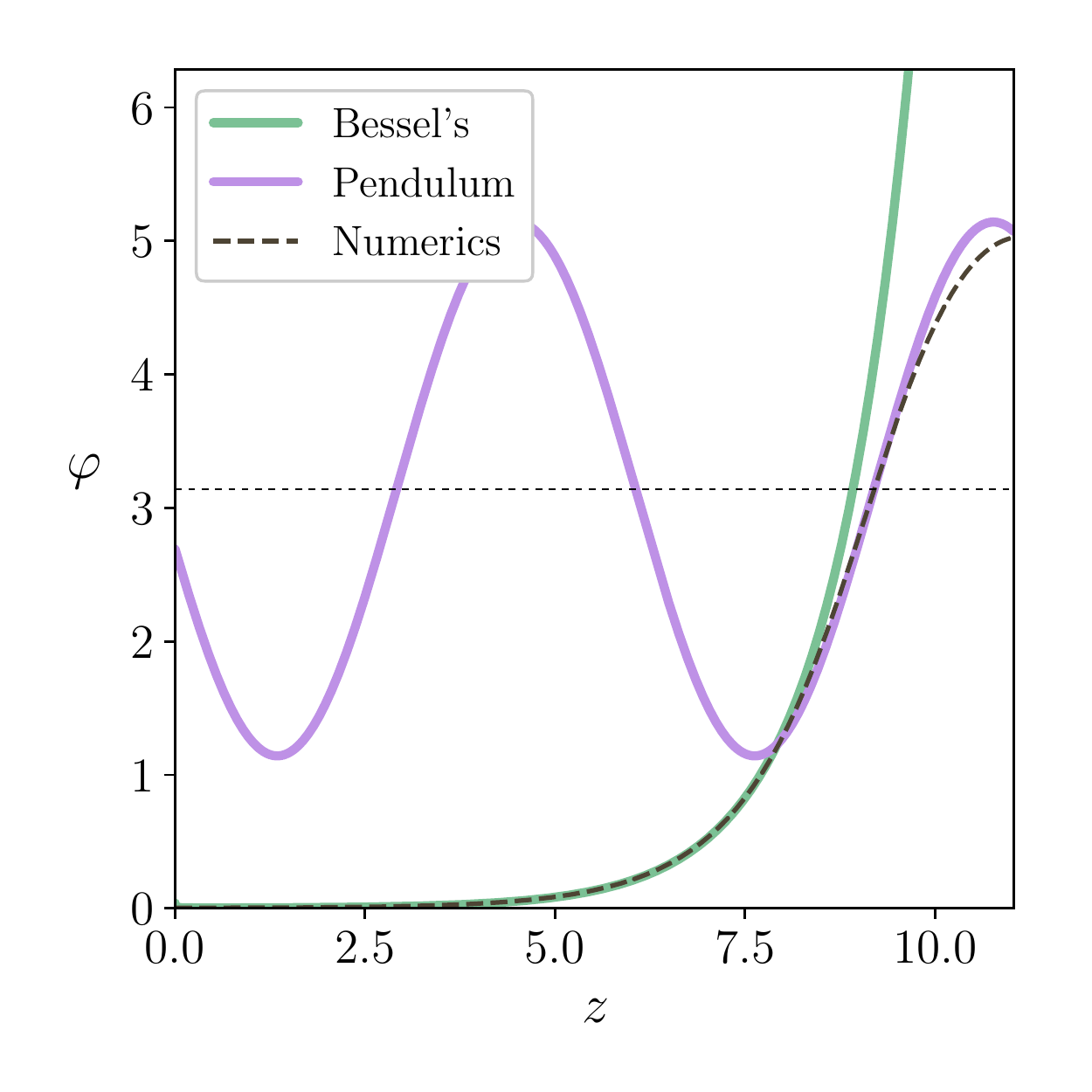}
    \caption{Asymptotics using the Pendulum and Bessel approximations in the $\varphi_0\to 0$ limit compared to the numerical solution of the Painlev\'{e} equation for $\varphi_0 = \frac{\pi}{100}$. Our interest is in approximating the exact solution well on an interval $[0,z^*]$ where $z = z^* \approx 9$ is the first instance where $\varphi(z) = \pi$, depicted by the dashed horizontal line in the figure.}
    \label{fig:painleveasymptotics}
\end{figure}

We are now equipped with a complete asymptotic description of the solutions to Painlev\'{e} III for an initial angle $\varphi_0$. The description is divided into three regimes: $z \ll 1$ and $\varphi_0 \lesssim \varphi \ll \pi$, $z \gg 1$ and $\varphi_0 \ll \varphi \lesssim \pi$, and finally $z\gg 1$ and $\varphi \approx \pi$:
\begin{align}
\varphi(z) \approx \left\{\begin{matrix}
\varphi_0\left(1 + \frac{z^2}{4}\right),\ z\ll 1\textrm{ and } \varphi_0 \lesssim \varphi \ll \pi\\
\varphi_0 \frac{e^{z}}{\sqrt{2\pi z}}\left(1 + \frac{1}{8z}\right),\ z\gg 1\textrm{ and } \varphi_0 \ll \varphi \lesssim \pi\\
 \pi - 2\sqrt{1 + \frac{e^{4 z}}{192\pi^2 z^2}}\sin(z^* - z),\ \varphi\approx \pi, z \lesssim z^* \approx -\log(\varphi_0)
\end{matrix}\right.
\label{eqa5}
\end{align}
A numerical validation of these asymptotic relations is illustrated in Fig.~\ref{fig:painleveasymptotics} (we consider $\varphi_0 = \frac{\pi}{100}$). Using the expressions in~\eqref{eqa5} instead of the bounds~\eqref{boundCphi*} gives the optimal constant $C(\phi^*) =1$ in Lemma~\ref{lem:amsler_recursion}.

\section*{Acknowledgments}
 We are grateful to Amit Acharya, Andrew Sageman-Furnas, David Glickenstein, Eran Sharon, John Gemmer and Kenneth Yamamoto for many stimulating discussions. SV gratefully acknowledges the hospitality of the Center for Nonlinear Analysis at Carnegie Mellon University, the Oxford Center for Industrial and Applied Math at Oxford University and the Hausdorff Institute at the University of Bonn where portions of this work were carried out. TS was partially supported  by a Michael Tabor fellowship from the Graduate Interdisciplinary Program in Applied Mathematics at the University of Arizona.  SV was partially supported by the Simons Foundation through awards 524875 and 560103 and partially supported by the NSF award DMR-1923922.

\section*{Author Contributions}
 This article grew out of the Ph.D thesis work of TS, supervised by SV. TS wrote the initial draft. SV revised the draft and incorporated additional material/proofs. Both authors contributed to performing the research reported here. Both authors read and approved the final manuscript.
 
\newcommand{\etalchar}[1]{$^{#1}$}


\begin{thebibliography}{GMVM19}

\bibitem[AB02]{audoly2002ruban}
Basile Audoly and Arezki Boudaoud.
\newblock `ruban {\`a} godets': an elastic model for ripples in plant leaves.
\newblock {\em Comptes Rendus Mecanique}, 330(12):831--836, 2002.

\bibitem[AB03]{audoly2003self}
B~Audoly and A~Boudaoud.
\newblock Self-similar structures near boundaries in strained systems.
\newblock {\em Phys. {R}ev. {L}ett.}, 91(8):086105, 2003.

\bibitem[ADH98]{asratian1998bipartite}
Armen~S Asratian, Tristan M.~J. Denley, and Roland H\"{a}ggkvist.
\newblock {\em Bipartite graphs and their applications}.
\newblock Cambridge University Press, Cambridge U.K., New York, 1998.

\bibitem[Ams55]{amsler1955surfaces}
Marc-Henri Amsler.
\newblock Des surfaces {\`a} courbure n{\'e}gative constante dans l'espace
  {\`a} trois dimensions et de leurs singularit{\'e}s.
\newblock {\em Mathematische Annalen}, 130(3):234--256, 1955.

\bibitem[And05]{anderson2005hyperbolic}
James Anderson.
\newblock {\em Hyperbolic geometry}.
\newblock Springer, London New York, 2005.

\bibitem[AS92]{Abram_Stegun}
Milton Abramowitz and Irene~A. Stegun, editors.
\newblock {\em Handbook of mathematical functions with formulas, graphs, and
  mathematical tables}.
\newblock Dover Publications Inc., New York, 1992.
\newblock Reprint of the 1972 edition.

\bibitem[AV20]{acharya2020continuum}
Amit Acharya and Shankar~C. Venkataramani.
\newblock Mechanics of moving defects in growing sheets: 3-d, small deformation
  theory.
\newblock {\em Materials Theory}, 4(1):2, 2020.

\bibitem[BAP97]{benAmar1997Crumpled}
M.~Ben~Amar and Y.~Pomeau.
\newblock Crumpled paper.
\newblock {\em Proceedings of the Royal Society of London. Series A:
  Mathematical, Physical and Engineering Sciences}, 453(1959):729--755, 1997.

\bibitem[BE00]{bobenko2000painleve}
Alexander~I Bobenko and Ulrich Eitner.
\newblock {\em {P}ainlev{\'e} equations in the differential geometry of
  surfaces}, volume 1753.
\newblock Springer Science \& Business Media, 2000.

\bibitem[BK87]{Ball1987Numerical}
J.~M. Ball and G.~Knowles.
\newblock A numerical method for detecting singular minimizers.
\newblock {\em Numer. Math.}, 51(2):181--197, 1987.

\bibitem[BK14a]{bella2014metric}
Peter Bella and Robert~V Kohn.
\newblock Metric-induced wrinkling of a thin elastic sheet.
\newblock {\em J. Nonlinear Sci.}, 24(6):1147--1176, 2014.

\bibitem[BK14b]{bella2014wrinkles}
Peter Bella and Robert~V Kohn.
\newblock Wrinkles as the result of compressive stresses in an annular thin
  film.
\newblock {\em Communications on Pure and Applied Mathematics}, 67(5):693--747,
  2014.

\bibitem[BLS16]{bhattacharya2016plates}
Kaushik Bhattacharya, Marta Lewicka, and Mathias Sch{\"a}ffner.
\newblock Plates with incompatible prestrain.
\newblock {\em Archive for Rational Mechanics and Analysis}, 221(1):143--181,
  2016.

\bibitem[BM85]{Ball1985One}
J.~M. Ball and V.~J. Mizel.
\newblock One-dimensional variational problems whose minimizers do not satisfy
  the {E}uler-{L}agrange equation.
\newblock {\em Arch. Rational Mech. Anal.}, 90(4):325--388, 1985.

\bibitem[BN95]{Brezis1995Degree}
H.~Brezis and L.~Nirenberg.
\newblock Degree theory and {BMO}. {I}. {C}ompact manifolds without boundaries.
\newblock {\em Selecta Math. (N.S.)}, 1(2):197--263, 1995.

\bibitem[BN96]{Brezis1996Degree}
Ha\"{\i}m Brezis and Louis Nirenberg.
\newblock Degree theory and {BMO}. {II}. {C}ompact manifolds with boundaries.
\newblock {\em Selecta Math. (N.S.)}, 2(3):309--368, 1996.
\newblock With an appendix by the authors and Petru Mironescu.

\bibitem[Bor59]{Borisov1959On}
Yu.~F. Borisov.
\newblock On the connection between the spatial form of smooth surfaces and
  their intrinsic geometry.
\newblock {\em Vestnik Leningrad. Univ.}, 14(13):20--26, 1959.

\bibitem[Bor04]{Borisov2004Irregular}
Yu.~F. Borisov.
\newblock Irregular surfaces of the class {$C^{1,\beta}$} with an analytic
  metric.
\newblock {\em Sibirsk. Mat. Zh.}, 45(1):25--61, 2004.
\newblock English translation in Siberian Math. J. {\bf 45} (2004), no. 1,
  19--52.

\bibitem[BS08]{bobenko2008bdiscrete}
Alexander~I. Bobenko and Yuri~B. Suris.
\newblock {\em Discrete differential geometry: Integrable structure}, volume~98
  of {\em Graduate Studies in Mathematics}.
\newblock American Mathematical Society, Providence, RI, 2008.

\bibitem[CDD14]{chopin2014roadmap}
Julien Chopin, Vincent D{\'e}mery, and Benny Davidovitch.
\newblock Roadmap to the morphological instabilities of a stretched twisted
  ribbon.
\newblock {\em Journal of Elasticity}, 119(1-2):137--189, 2014.

\bibitem[CDLS12]{Conti2012hPrinciple}
Sergio Conti, Camillo De~Lellis, and L\'{a}szl\'{o} Sz\'{e}kelyhidi, Jr.
\newblock {$h$}-principle and rigidity for {$C^{1,\alpha}$} isometric
  embeddings.
\newblock In {\em Nonlinear partial differential equations}, volume~7 of {\em
  Abel Symp.}, pages 83--116. Springer, Heidelberg, 2012.

\bibitem[Ces83]{Cesari1983Optimization}
Lamberto Cesari.
\newblock {\em Optimization---theory and applications}, volume~17 of {\em
  Applications of Mathematics (New York)}.
\newblock Springer-Verlag, New York, 1983.
\newblock Problems with ordinary differential equations.

\bibitem[Cia80]{ciarlet1980justification}
Philippe~G Ciarlet.
\newblock A justification of the von {K}{\'a}rm{\'a}n equations.
\newblock {\em Archive for Rational Mechanics and Analysis}, 73(4):349--389,
  1980.

\bibitem[DLI20]{DeLellis2020Isometric}
Camillo De~Lellis and Dominik Inauen.
\newblock {$C^{1, \alpha}$} isometric embeddings of polar caps.
\newblock {\em Adv. Math.}, 363:106996, 39, 2020.

\bibitem[DLIS18]{DeLellis2018immersions}
Camillo De~Lellis, Dominik Inauen, and L\'{a}szl\'{o} Sz\'{e}kelyhidi, Jr.
\newblock A {N}ash-{K}uiper theorem for {$C^{1,1/5-\delta}$} immersions of
  surfaces in 3 dimensions.
\newblock {\em Rev. Mat. Iberoam.}, 34(3):1119--1152, 2018.

\bibitem[DS16]{dorfmeister2016pseudospherical}
Josef~F. Dorfmeister and Ivan Sterling.
\newblock Pseudospherical surfaces of low differentiability.
\newblock {\em Adv. Geom.}, 16(1):1--20, 2016.

\bibitem[DSG19]{davidovitch2019geometrically}
Benny Davidovitch, Yiwei Sun, and Gregory~M. Grason.
\newblock Geometrically incompatible confinement of solids.
\newblock {\em Proceedings of the National Academy of Sciences},
  116(5):1483--1488, 2019.

\bibitem[DSV{\etalchar{+}}11]{Davidovitch2011Prototypical}
Benny Davidovitch, Robert~D. Schroll, Dominic Vella, Mokhtar Adda-Bedia, and
  Enrique~A. Cerda.
\newblock Prototypical model for tensional wrinkling in thin sheets.
\newblock {\em Proceedings of the National Academy of Sciences},
  108(45):18227--18232, 2011.

\bibitem[Efi62]{efimov1963impossibility}
N.~V. Efimov.
\newblock Impossibility of an isometric imbedding in {E}uclidean {$3$}-space of
  certain manifolds with negative {G}aussian curvature.
\newblock {\em Dok. Akad. Nauk SSSR}, 146:296--299, 1962.

\bibitem[Efi64]{efimov1964generation}
Nikolai~Vladimirovich Efimov.
\newblock Generation of singularites on surfaces of negative curvature.
\newblock {\em Matematicheskii Sbornik}, 106(2):286--320, 1964.

\bibitem[Eis09]{eisenhart1909treatise}
Luther~Pfahler Eisenhart.
\newblock {\em A treatise on the differential geometry of curves and surfaces}.
\newblock Ginn, 1909.

\bibitem[EKAS07]{efrati2007spontaneous}
Efi Efrati, Yael Klein, Hillel Aharoni, and Eran Sharon.
\newblock Spontaneous buckling of elastic sheets with a prescribed
  non-{E}uclidean metric.
\newblock {\em Physica D: Nonlinear Phenomena}, 235(1):29--32, 2007.

\bibitem[ESK09]{efrati2009elastic}
Efi Efrati, Eran Sharon, and Raz Kupferman.
\newblock Elastic theory of unconstrained non-{E}uclidean plates.
\newblock {\em Journal of the Mechanics and Physics of Solids}, 57(4):762--775,
  2009.

\bibitem[ESK13]{Efrati2013Metric}
Efi Efrati, Eran Sharon, and Raz Kupferman.
\newblock The metric description of elasticity in residually stressed soft
  materials.
\newblock {\em Soft Matter}, 9(34):8187--8197, 2013.

\bibitem[Eva98]{evans2}
Lawrence~C. Evans.
\newblock {\em Partial differential equations}.
\newblock American Mathematical Society, 1998.

\bibitem[FHM03]{Foss2003Lavretiev}
M.~Foss, W.~J. Hrusa, and V.~J. Mizel.
\newblock The {L}avrentiev gap phenomenon in nonlinear elasticity.
\newblock {\em Arch. Ration. Mech. Anal.}, 167(4):337--365, 2003.

\bibitem[FJM02]{friesecke2002foppl}
Gero Friesecke, Richard~D James, and Stefan M{\"u}ller.
\newblock The {F}{\"o}ppl--von {K}{\'a}rm{\'a}n plate theory as a low energy
  {$\Gamma$}-limit of nonlinear elasticity.
\newblock {\em Comptes Rendus Mathematique}, 335(2):201--206, 2002.

\bibitem[FJM06]{friesecke2006hierarchy}
Gero Friesecke, Richard~D James, and Stefan M{\"u}ller.
\newblock A hierarchy of plate models derived from nonlinear elasticity by
  {G}amma-convergence.
\newblock {\em Archive for rational mechanics and analysis}, 180(2):183--236,
  2006.

\bibitem[GMVM19]{Guven2019Isometric}
Jemal Guven, Martin~Michael M{\"u}ller, and Pablo V{\'a}zquez-Montejo.
\newblock Isometric bending requires local constraints on free edges.
\newblock {\em Mathematics and Mechanics of Solids}, 24(12):4051--4077,
  2020/06/26 2019.

\bibitem[Gra98]{gray1998modern}
Alfred Gray.
\newblock {\em Modern differential geometry of curves and surfaces with
  {M}athematica}.
\newblock CRC Press, Boca Raton, FL, second edition, 1998.

\bibitem[GSSV16]{EPL_2016}
John Gemmer, Eran Sharon, Toby Shearman, and Shankar~C. Venkataramani.
\newblock Isometric immersions, energy minimization and self-similar buckling
  in non-{E}uclidean elastic sheets.
\newblock {\em Europhys. Lett.}, 114(2):24003, 2016.

\bibitem[GV11]{gemmer2011shape}
John~A Gemmer and Shankar~C Venkataramani.
\newblock Shape selection in non-{E}uclidean plates.
\newblock {\em Physica D: Nonlinear Phenomena}, 240(19):1536--1552, 2011.

\bibitem[GV12]{gemmer2012defects}
JA~Gemmer and SC~Venkataramani.
\newblock Defects and boundary layers in non-{E}uclidean plates.
\newblock {\em Nonlinearity}, 25(12):3553, 2012.

\bibitem[GV13]{Gemmer2013Shape}
John~A. Gemmer and Shankar~C. Venkataramani.
\newblock Shape transitions in hyperbolic non-{E}uclidean plates.
\newblock {\em Soft Matter}, 9(34):8151--8161, 2013.

\bibitem[Ham24]{hamburger1921}
Hans Hamburger.
\newblock {\"U}ber kurvennetze mit isolierten singularit{\"a}ten auf
  geschlossenen fl{\"a}chen.
\newblock {\em Math. Z.}, 19(1):50--66, December 1924.

\bibitem[Hat02]{HatcherAlgTop}
Allen Hatcher.
\newblock {\em Algebraic topology}.
\newblock Cambridge University Press, Cambridge, 2002.

\bibitem[HH06]{han2006isometric}
Qing Han and Jia-Xing Hong.
\newblock {\em Isometric Embedding of {R}iemannian manifolds in Euclidean
  spaces}, volume 130.
\newblock American Mathematical Society Providence, RI, 2006.

\bibitem[Hil01]{hilbert1901}
David Hilbert.
\newblock {\"U}ber {F}l{\"a}chen von constanter {G}aussscher {K}r{\"u}mmung.
\newblock {\em Transactions of the American Mathematical Society}, 2(1):87--99,
  1901.

\bibitem[HN59]{Hartman1959spherical}
Philip Hartman and Louis Nirenberg.
\newblock On spherical image maps whose {J}acobians do not change sign.
\newblock {\em Amer. J. Math.}, 81:901--920, 1959.

\bibitem[Hol02]{holmgren1902surfaces}
Erik Holmgren.
\newblock Sur les surfaces {\`a} courbure constante n{\'e}gative.
\newblock {\em CR Acad. Sci. Paris}, 134:740--743, 1902.

\bibitem[Hon93]{hong1993realization}
Jia~Xing Hong.
\newblock Realization in {${\bf R}^3$} of complete {R}iemannian manifolds with
  negative curvature.
\newblock {\em Comm. Anal. Geom.}, 1(3-4):487--514, 1993.

\bibitem[Hor11]{Hornung2011Approximation}
Peter Hornung.
\newblock Approximation of flat {$W^{2,2}$} isometric immersions by smooth
  ones.
\newblock {\em Arch. Ration. Mech. Anal.}, 199(3):1015--1067, 2011.

\bibitem[HT01]{henderson2001crocheting}
David~W Henderson and Daina Taimina.
\newblock Crocheting the hyperbolic plane.
\newblock {\em The Mathematical Intelligencer}, 23(2):17--28, 2001.

\bibitem[HV18]{Hornung2018Regularity}
Peter Hornung and Igor Vel\v{c}i\'{c}.
\newblock Regularity of intrinsically convex {$W^{2,2}$} surfaces and a
  derivation of a homogenized bending theory of convex shells.
\newblock {\em J. Math. Pures Appl. (9)}, 115:1--23, 2018.

\bibitem[HVR14]{Huhnen-Venedey2014Discretization}
Emanuel Huhnen-Venedey and Thilo R{\"o}rig.
\newblock Discretization of asymptotic line parametrizations using hyperboloid
  surface patches.
\newblock {\em Geometriae Dedicata}, 168(1):265--289, 2014.

\bibitem[HW51]{hartman1951asymptotic}
Philip Hartman and Aurel Wintner.
\newblock On the asymptotic curves of a surface.
\newblock {\em American Journal of Mathematics}, 73(1):149--172, 1951.

\bibitem[HWQ{\etalchar{+}}18]{Huang2018Differential}
Changjin Huang, Zilu Wang, David Quinn, Subra Suresh, and K.~Jimmy Hsia.
\newblock Differential growth and shape formation in plant organs.
\newblock {\em Proceedings of the National Academy of Sciences},
  115(49):12359--12364, 2018.

\bibitem[IL03]{ivey2003cartan}
Thomas~A Ivey and JM~Landsberg.
\newblock {\em Cartan for beginners}, volume~61 of {\em Graduate Studies in
  Mathematics}.
\newblock American Mathematical Society Providence, RI, 2003.

\bibitem[IM06]{Ishikawa2006Singularities}
Go-O Ishikawa and Yoshinori Machida.
\newblock Singularities of improper affine spheres and surfaces of constant
  {G}aussian curvature.
\newblock {\em Internat. J. Math.}, 17(3):269--293, 2006.

\bibitem[Joh68]{John1968quasi-isometricI}
Fritz John.
\newblock On quasi-isometric mappings. {I}.
\newblock {\em Comm. Pure Appl. Math.}, 21:77--110, 1968.

\bibitem[Joh69]{John1969quasi-isometricII}
Fritz John.
\newblock On quasi-isometric mappings. {II}.
\newblock {\em Comm. Pure Appl. Math.}, 22:265--278, 1969.

\bibitem[KES07]{klein2007shaping}
Yael Klein, Efi Efrati, and Eran Sharon.
\newblock Shaping of elastic sheets by prescription of non-{E}uclidean metrics.
\newblock {\em Science}, 315(5815):1116--1120, 2007.

\bibitem[KHB{\etalchar{+}}12]{Kim2012Hydrogel}
Jungwook Kim, James~A Hanna, Myunghwan Byun, Christian~D Santangelo, and Ryan~C
  Hayward.
\newblock Designing responsive buckled surfaces by halftone gel lithography.
\newblock {\em Science}, 335(6073):1201--1205, 2012.

\bibitem[KHHS12]{kim2012thermally}
Jungwook Kim, James~A Hanna, Ryan~C Hayward, and Christian~D Santangelo.
\newblock Thermally responsive rolling of thin gel strips with discrete
  variations in swelling.
\newblock {\em Soft Matter}, 8(8):2375--2381, 2012.

\bibitem[Kir01]{kirchheim2001Rigidity}
Bernd Kirchheim.
\newblock {\em Rigidity and {G}eometry of {M}icrostructures}.
\newblock Habilitation, University of Leipzig, 2001.

\bibitem[KMM04]{KMM2004}
Tomasz Kaczynski, Konstantin Mischaikow, and Marian Mrozek.
\newblock {\em Computational homology}, volume 157 of {\em Applied Mathematical
  Sciences}.
\newblock Springer-Verlag, New York, 2004.

\bibitem[KS14]{kupferman2014riemannian}
Raz Kupferman and Jake~P Solomon.
\newblock A {R}iemannian approach to reduced plate, shell, and rod theories.
\newblock {\em Journal of Functional Analysis}, 266(5):2989--3039, 2014.

\bibitem[Kui55]{kuiper1955c1}
Nicolaas~H. Kuiper.
\newblock On {$C^1$}-isometric imbeddings. {I}, {II}.
\newblock {\em Nederl. Akad. Wetensch. Proc. Ser. A. {\bf 58} = Indag. Math.},
  17:545--556, 683--689, 1955.

\bibitem[KVS11]{Klein2011Experimental}
Yael Klein, Shankar Venkataramani, and Eran Sharon.
\newblock Experimental {S}tudy of {S}hape {T}ransitions and {E}nergy {S}caling
  in {T}hin {N}on-{E}uclidean {P}lates.
\newblock {\em Phys. Rev. Lett.}, 106(11):118303, March 2011.

\bibitem[Lav26]{Lavrentieff1927Quelques}
M~Lavrentieff.
\newblock Sur quelques problemes du calcul des variations.
\newblock {\em Annali di Matematica Pura ed Applicata}, 4(1):7--28, 1926.

\bibitem[LGL{\etalchar{+}}95]{science.paper}
A.~Lobkovsky, S.~Gentges, H.~Li, D.~Morse, and T.~A. Witten.
\newblock Scaling properties of stretching ridges in a crumpled elastic sheet.
\newblock {\em Science}, 270:1482, 1995.

\bibitem[LM09]{Liang2009shape}
Haiyi Liang and L.~Mahadevan.
\newblock The shape of a long leaf.
\newblock {\em Proceedings of the National Academy of Sciences},
  106(52):22049--22054, 2009.

\bibitem[LM11]{LiangMaha2011}
Haiyi Liang and L~Mahadevan.
\newblock Growth, geometry, and mechanics of a blooming lily.
\newblock {\em Proceedings of the National Academy of Sciences},
  108(14):5516--5521, April 2011.

\bibitem[LMP14]{lewicka2014models}
Marta Lewicka, L.~Mahadevan, and Mohammad~Reza Pakzad.
\newblock Models for elastic shells with incompatible strains.
\newblock {\em Proc. Roy. Soc. London Ser. A}, 470(2165):20130604, 2014.

\bibitem[Lov92]{love2013treatise}
Augustus Edward~Hough Love.
\newblock {\em A treatise on the mathematical theory of elasticity}.
\newblock Cambridge university press, 1892.

\bibitem[LR20]{floraform2014}
Jesse Louis-Rosenberg.
\newblock Floraform.
\newblock \url{http://n-e-r-v-o-u-s.com/blog/?p=6721}, 2014 (accessed June 21,
  2020).

\bibitem[LRP11]{lewicka2011scaling}
Marta Lewicka and Mohammad Reza~Pakzad.
\newblock Scaling laws for non-{E}uclidean plates and the {$ W^{2, 2}$}
  isometric immersions of {R}iemannian metrics.
\newblock {\em ESAIM: Control, Optimisation and Calculus of Variations},
  17(04):1158--1173, 2011.

\bibitem[Mar03]{marder2003shape}
M~Marder.
\newblock The shape of the edge of a leaf.
\newblock {\em Foundations of Physics}, 33(12):1743--1768, 2003.

\bibitem[Mey20]{bridges2013}
Gabriele Meyer.
\newblock 2013 bridges conference: Mathematical art galleries.
\newblock
  \url{http://gallery.bridgesmathart.org/exhibitions/2013-bridges-conference/gabriele_meyer},
  2013 (accessed June 21, 2020).

\bibitem[Mil72]{milnor1972efimov}
Tilla~Klotz Milnor.
\newblock Efimov's theorem about complete immersed surfaces of negative
  curvature.
\newblock {\em Advances in Math.}, 8(3):474--543, 1972.

\bibitem[MSSR03]{marder2003theory}
M~Marder, E~Sharon, S~Smith, and Benoit Roman.
\newblock Theory of edges of leaves.
\newblock {\em EPL (Europhysics Letters)}, 62(4):498, 2003.

\bibitem[M{\"u}l17]{Muller_review}
Stefan M{\"u}ller.
\newblock Mathematical problems in thin elastic sheets: Scaling limits,
  packing, crumpling and singularities.
\newblock In John Ball and Paolo Marcellini, editors, {\em Vector-Valued
  Partial Differential Equations and Applications: Cetraro, Italy 2013}, pages
  125--193. Springer International Publishing, Cham, 2017.

\bibitem[MV88]{Martio1988Elliptic}
O.~Martio and J.~V\"{a}is\"{a}l\"{a}.
\newblock Elliptic equations and maps of bounded length distortion.
\newblock {\em Math. Ann.}, 282(3):423--443, 1988.

\bibitem[Nas54]{nash1954c1}
John Nash.
\newblock {$C^1$} isometric imbeddings.
\newblock {\em Annals of mathematics}, pages 383--396, 1954.

\bibitem[NP15]{nechaev2015buckling}
Sergei Nechaev and Kirill Polovnikov.
\newblock Buckling and wrinkling from geometric and energetic viewpoints, 2015.

\bibitem[NP17]{Nechaev2017From}
Sergei Nechaev and Kirill Polovnikov.
\newblock From geometric optics to plants: the eikonal equation for buckling.
\newblock {\em Soft Matter}, 13:1420--1429, 2017.

\bibitem[NV01]{nechaev2001plant}
Sergei Nechaev and Rapha{\"e}l Voituriez.
\newblock On the plant leaf's boundary, jupe {\`a} godets' and conformal
  embeddings.
\newblock {\em J. Phys. A: Mathematical and General}, 34(49):11069, 2001.

\bibitem[Olb16]{olbermann2016d-cone}
Heiner Olbermann.
\newblock The one-dimensional model for d-cones revisited.
\newblock {\em Adv. Calc. Var.}, 9(3):201--215, 2016.

\bibitem[Pak04]{Pakzad2004Sobolev}
Mohammad~Reza Pakzad.
\newblock On the {S}obolev space of isometric immersions.
\newblock {\em J. Differential Geom.}, 66(1):47--69, 2004.

\bibitem[Roz62a]{Rozendorn1962Complete}
\`E.~R. Rozendorn.
\newblock On complete surfaces of negative curvature {$K\leq -1$} in the
  {E}uclidean spaces {$E_{3}$} and {$E_{4}$}.
\newblock {\em Mat. Sb. (N.S.)}, 58 (100):453--478, 1962.

\bibitem[Roz62b]{Rozendorn1962Asymptotic}
\`E.~R. Rozendorn.
\newblock Properties of asymptotic lines on surfaces with slowly varying
  negative curvature.
\newblock {\em Dokl. Akad. Nauk SSSR}, 145:538--540, 1962.

\bibitem[Roz66]{Rozendorn1966Weakly}
\`E.~R. Rozendorn.
\newblock Weakly irregular surfaces of negative curvature.
\newblock {\em Uspehi Mat. Nauk}, 21(5 (131)):59--116, 1966.

\bibitem[Roz92]{Rozendorn-Chapter-92}
E.~R. Rozendorn.
\newblock {S}urfaces of {N}egative {C}urvature.
\newblock In Yu.~D. Burago and V.~A. Zalgaller, editors, {\em Geometry III},
  volume~48 of {\em Encyclopaedia of Mathematical Sciences}, pages 87--178.
  Springer Berlin Heidelberg, 1992.

\bibitem[RS02]{rogers2002backlund}
Colin Rogers and Wolfgang~Karl Schief.
\newblock {\em {B}{\"a}cklund and {D}arboux transformations: geometry and
  modern applications in soliton theory}, volume~30.
\newblock Cambridge University Press, 2002.

\bibitem[Sau50]{sauer1950parallelogrammgitter}
Robert Sauer.
\newblock Parallelogrammgitter als {M}odelle pseudosph\"{a}rischer
  {F}l\"{a}chen.
\newblock {\em Mathematische Zeitschrift}, 52(1):611--622, 1950.

\bibitem[Sch07a]{Schmidt2007Minimal}
Bernd Schmidt.
\newblock Minimal energy configurations of strained multi-layers.
\newblock {\em Calc. Var. Partial Differential Equations}, 30(4):477--497,
  2007.

\bibitem[Sch07b]{Schmidt2007Plate}
Bernd Schmidt.
\newblock Plate theory for stressed heterogeneous multilayers of finite bending
  energy.
\newblock {\em J. Math. Pures Appl. (9)}, 88(1):107--122, 2007.

\bibitem[SMS04]{eran2004leaves}
Eran Sharon, Michael Marder, and Harry~L Swinney.
\newblock Leaves, flowers and garbage bags: making waves.
\newblock {\em American Scientist}, 92(3):254, 2004.

\bibitem[SRM{\etalchar{+}}02]{sharon2002buckling}
Eran Sharon, Beno{\^\i}t Roman, Michael Marder, Gyu-Seung Shin, and Harry~L.
  Swinney.
\newblock Buckling cascades in free sheets.
\newblock {\em Nature}, 419(6907):579--579, 2002.

\bibitem[SRS07]{sharon2007geometrically}
Eran Sharon, Beno{\^\i}t Roman, and Harry~L. Swinney.
\newblock Geometrically driven wrinkling observed in free plastic sheets and
  leaves.
\newblock {\em Physical Review E (Statistical, Nonlinear, and Soft Matter
  Physics)}, 75(4):046211, 2007.

\bibitem[SS18]{sharon2018mechanics}
Eran Sharon and Michal Sahaf.
\newblock The mechanics of leaf growth on large scales.
\newblock In Anja Geitmann and Joseph Gril, editors, {\em Plant Biomechanics:
  From Structure to Function at Multiple Scales}, pages 109--126. Springer
  International Publishing, 2018.

\bibitem[Sto89]{stoker}
J.~J. Stoker.
\newblock {\em Differential geometry}.
\newblock Wiley Classics Library. John Wiley \& Sons Inc., 1989.
\newblock Reprint of the 1969 original, A Wiley-Interscience Publication.

\bibitem[Tim59]{timoshenko1959theory}
Stephen Timoshenko.
\newblock {\em Theory of plates and shells}.
\newblock McGraw-Hill, New York, 1959.

\bibitem[Tob20]{tobasco2020curvaturedriven}
Ian Tobasco.
\newblock Curvature-driven wrinkling of thin elastic shells.
\newblock (to appear) Arch. Ration. Mech. Anal., 2020.

\bibitem[Ven03]{venkataramani2003lower}
Shankar~C Venkataramani.
\newblock Lower bounds for the energy in a crumpled elastic sheet---a minimal
  ridge.
\newblock {\em Nonlinearity}, 17(1):301, 2003.

\bibitem[VSJ{\etalchar{+}}13]{vetter2013subdivision}
Roman Vetter, Norbert Stoop, Thomas Jenni, Falk~K Wittel, and Hans~J Herrmann.
\newblock Subdivision shell elements with anisotropic growth.
\newblock {\em International Journal for Numerical Methods in Engineering},
  95(9):791--810, 2013.

\bibitem[Wei96]{Lorentz_surfaces_Weinstein}
Tilla Weinstein.
\newblock {\em An introduction to {L}orentz surfaces}, volume~22 of {\em De
  Gruyter Expositions in Mathematics}.
\newblock Walter de Gruyter \& Co., Berlin, 1996.

\bibitem[Wis72]{wissler1972}
Ch. Wissler.
\newblock Globale {T}schebyscheff-{N}etze auf {R}iemannschen
  {M}annigfaltigkeiten und {F}ortsetzung von {F}l\"achen konstanter negativer
  {K}r\"ummung.
\newblock {\em Comment. Math. Helv.}, 47:348--372, 1972.

\bibitem[Wun51]{wunderlich1951differenzengeometrie}
Walter Wunderlich.
\newblock Zur {D}ifferenzengeometrie der {F}l{\"a}chen konstanter negativer
  {K}r{\"u}mmung.
\newblock {\em {\"O}sterreich. Akad. Wiss. Math.-Nat. Kl. S.-B. IIa.},
  160:39--77, 1951.

\bibitem[WW15]{wertheim2015crochet}
Margaret Wertheim and Christine Wertheim.
\newblock {\em Crochet {C}oral {R}eef}.
\newblock Institute for Figuring, Los Angeles, 2015.
\newblock With contributions by Leslie Dick, Marion Endt-Jones and Anna Mayer
  and a foreword by Donna Haraway.

\end{thebibliography}
\end{document}